%% file: main.tex
\documentclass{amsart}
\input{preambulo}

\title[An explicit description of the colored mutation class of $\widetilde{\mathbb{A}}_n$-quivers ]{An explicit description of the colored mutation class of $\widetilde{\mathbb{A}}_n$-quivers}

\author[V. Gubitosi]{Viviana Gubitosi}
\address{Instituto de Matem\'{a}tica y Estad\'{\i}stica Rafael Laguardia, Facultad de Ingenier\'{\i}a - UdelaR, Montevideo, Uruguay, 11200 }
\email{gubitosi@fing.edu.uy}

\author[P. Rosero]{Pablo Rosero}
\address{Programa de desarrollo de las Ciencias Básicas - PEDECIBA - UdelaR, Montevideo, Uruguay }
\email{prosero@yachaytech.edu.ec}

\date{\today} 


\begin{document}

\begin{abstract}
This paper addresses the combinatorial structure of $m$-colored mutation classes. We provide an explicit and purely combinatorial description of the $m$-colored quivers that arise within the $m$-colored mutation class of a quiver of type $\mathbb{\widetilde{A}}$. Our results generalize and extend existing work, specifically recovering a description by Bastian \cite{bastian2012mutation} when the case $m=1$ is considered. 
\end{abstract}

\maketitle

% In this paper we give an explicit and purely combinatorial description of the $m$-colored
% quivers that appears in the $m$-colored mutation class of a quiver of type $\widetilde{\mathbb{A}}$. 
% In particular, our description generalizes a result of Bastian \cite{Bastian} \textcolor{blue}{referenciar}, which we recover when $m = 1$.

\section*{Introduction}

Given an integer $m\geq 1$, colored quivers or $m$-colored quivers, introduced by Buan and Thomas in 2008 \cite{buan}, have provided a fundamental combinatorial framework for the study of $m$-cluster tilted algebras. This construction assigns a color $c \in \{0, \dots, m\}$ to each arrow of a quiver with no loops and satisfying the properties of monochromaticity and skew-symmetry.

In this context, the colored quiver mutation is compatible with the mutation of $m$-cluster tilting objects. For the case $m=1$, this operation reduces to the classical Fomin and Zelevinsky mutation denoted by FZ-mutation \cite{fomin2002cluster}. Since its definition, the description of FZ-mutation classes for various types of quivers  has been the subject of intensive research, addressed by Henrich, Vatne, Bastian, and others \cites{buan2008derived, bastian2012mutation, henrich2011mutation, bastian2011counting}.

In the previous work \cite{gubitosi2024coloured}, the authors resolved the problem of explicitly describing the colored mutation class for colored quivers of type $\mathbb{A}_n$. They demonstrated that the mutation class of $\mathbb{A}_n$-quivers coincides with a readily recognizable class, denoted by $\mathcal{Q}^m_n$, providing in addition a purely combinatorial description of the quivers of $m$-cluster tilted algebras of type $\mathbb{A}_n$. While the previous work focused on Dynkin type $\mathbb{A}_n$,  the present paper addresses the colored mutation class of quivers of extended Dynkin type $\widetilde{\mathbb{A}}_n$.

The central aim of this work is to provide an explicit and combinatorial description of the colored mutation class of $\widetilde{\mathbb{A}}_n$-quivers. For $m=1$ the problem was addressed by Bastian in \cite{bastian2012mutation}.

Similar to \cite{gubitosi2024coloured}, the method we employ to obtain this description is purely combinatorial, relying on an extension of the clique-elimination algorithm and on the properties of almost extremal cliques within the cyclic structure of $\widetilde{\mathbb{A}}_n$.\\

Extending the class $\mathcal{Q}^m_n$, in Definition \ref{ref: definicion clase A tilde} we present a class $\mathcal{Q}_{p,q}^m$ of $m$-colored quivers with $p+q$ vertices which includes all colorings of a quiver of type $\widetilde{\mathbb{A}}_{p,q}$. 
The main result of this paper can be stated as follows:

\subsection*{Theorem A}\textit{ A connected  $m$-colored quiver $Q$ is mutation equivalent to  $\widetilde{\mathbb{A}}_{p,q}$ if and only if  $Q$ belongs to the class $\mathcal{Q}_{p,q}^m$. }

\medskip

In particular, specializing to the case $m=1$, we recover known results of \cite{bastian2012mutation}; and in addition, our description of the colored mutation class of $\widetilde{\mathbb{A}}_{p,q}$-colored quivers gives a complete description of quivers of $m$-cluster-tilted algebras of type $\widetilde{\mathbb{A}}_{p,q}$. These quivers are already known, see \cite{gubitosi2018m}, but the method used in this paper to obtain the description is purely combinatorial, and no prerequisites are needed.

\medskip
This paper is organized as follows. We begin with a preliminary section where we fix the notations and recall fundamental concepts needed later. Section 2 reviews the definition of colored quivers, colored mutation, and the description of the mutation class of type $\mathbb{A}_n$. Section 3 is dedicated to the study of a special type of cycles, which we call central cycles. In Section 4, we describe the special class of colored quivers with $p+q$ vertices, denoted by $\mathcal{Q}_{p,q}^m$. This class will be shown to be the colored mutation class of type $\widetilde{\mathbb{A}}_{p,q}$. We conclude Section 4 by demonstrating important properties of the class $\mathcal{Q}_{p,q}^m$. Section 5 contains the proof of our main result. Finally, Section 6 presents several consequences, including the recovery of the known results mentioned previously.

\section{Preliminaries}

\begin{defi}
A quiver $Q=(Q_0,Q_1)$ consists of the data of two sets, $Q_0$ (vertices) and $Q_1$ (arrows); together with two functions $\func{s,t}{Q_1}{Q_0}$ that assign to each arrow $\alpha$ its starting vertex $s(\alpha)$ and its ending vertex $t(\alpha)$. We write $\alpha\colon s(\alpha)\rightarrow t(\alpha)$ for the arrow $\alpha$ from $s(\alpha)$ to $t(\alpha)$.
\end{defi}

If we denote by $q_{ij}$ the number of arrows from $i$ to $j$ in $Q$, we say that the quiver $Q$ is \textit{simple} if there is at most one arrow between any pair of distinct vertices, that is, $q_{ij} \leq 1$ for $i \neq j$. In this case, every arrow $\alpha$ can be identified with its starting vertex $s(\alpha) = i$ and its ending vertex $t(\alpha) = j$, and is usually denoted by $\alpha = ij$. A \textit{path} of length $k \geq 0$ in $Q$ is a sequence of arrows $\alpha_1\alpha_2\ldots\alpha_k$ such that $\func{\alpha_i}{x_i}{x_{i+1}} \in Q_1$ for $1 \leq i \leq k$. A path is called \textit{simple} when the vertices $x_1, x_2, \ldots, x_{k+1}$ are pairwise distinct. A path is said to be \textit{closed} if $x_1 = x_{k+1}$. A $k$-\textit{cycle} ($k \geq 3$) in $Q$ is a closed path $\alpha_1\alpha_2\ldots\alpha_k$ such that $\alpha_1\alpha_2\ldots\alpha_{k-1}$ is a simple path. If $Q$ is a simple quiver, we denote the path $\alpha_1\alpha_2\ldots\alpha_k$ by $x_1x_2\ldots x_{k+1}$, and the $k$-cycle $x_1x_2\ldots x_kx_1$ by $(x_1x_2\ldots x_k)$.

\medskip
A subquiver $Q^{\prime}$ of $Q$ is called \textit{induced} if every $\alpha \in Q_1$ such that $s(\alpha), t(\alpha) \in Q_0^{\prime}$ satisfies $\alpha \in Q_1^{\prime}$. An \textit{induced cycle} is a cycle that is also an induced subquiver. A \textit{hole} is an induced cycle of length at least four. A quiver $Q$ is said to be \textit{hole-free} if it contains no holes.

\medskip

Let $I \subset Q_0$. We denote by $Q[I]$ the subquiver of $Q$ induced by $I$. A quiver is said to be \textit{complete} if it is a quiver in which every pair of distinct vertices is connected by a pair of arrows (exactly one in each direction). A $k$-\textit{clique} $\C_k$ is a subset of $k$ vertices of $Q_0$ such that $Q[\C_k]$ is a complete quiver. When necessary, we identify the clique $\C_k$ with the induced complete quiver $\C_k$. A triangle $T$ is a $3$-clique.

\section{Colored Quivers and Colored Mutations}

Let $m$ be a positive integer. An $m$-\textit{colored quiver} $Q=\left(Q_0,Q_1,\kappa\right)$ is a quiver $Q=(Q_0,Q_1)$  together with a coloring function $\func{\kappa}{Q_1}{\adel{0,1,\ldots,m}}$ that assigns to each arrow $\alpha$ the color $\kappa(\alpha).$ We denote by \(\alpha\colon a\xrightarrow{(c)} b\) the arrow $\alpha$ with $s(\alpha)=a$ and $t(\alpha)=b$ and color $\kappa(\alpha)=c$. %or simply, $\overline{\alpha}=c.$ 

\medskip

If $Q$ is a simple $m-$colored quiver we can simply denote by   $\overline{ij}=\kappa(\alpha)$ the color of the arrow $\alpha=ij$. The \textit{color of a path} $\alpha_1\ldots\alpha_k$ is defined as the sum \[\overline{\alpha_1\ldots\alpha_k}=\sum_{i=i}^k \overline{\alpha_i}.\]

If we denote by $q_{ij}^{(c)}$ the number of arrows from $i$ to $j$ with color $c$, we will say that $Q$ has no loops if $q_{ii}^{(c)}=0$ for every color $c$; $Q$ is monochromatic whenever $q_{ij}^{(c)}\neq 0$, then $q_{ij}^{(c^{\prime})}=0$ for $c\neq c^{\prime}$. Finally, we say that $Q$ is skew-symmetric if $q_{ij}^{(c)}=q_{ji}^{(m-c)}$ for all $i,j\in Q_0.$

Unless stated otherwise, all colored quivers considered in this paper are assumed to be skew-symmetric, monochromatic, and contain no loops. \medskip

Given such a colored quiver $Q$, we have the following operation $\mu_j$, called the colored mutation of the quiver $Q$ at the vertex $j$. Thus, $\mu_j(Q)=\widetilde{Q}$ is a new colored quiver with $Q_0=\widetilde{Q}_0$ and such that

%\[\widetilde{q}_{ik}^{(c)}=\begin{cases} q_{ik}^{(c+1)}& \text{if } j=k,  \\ q_{ik}^{(c-1)} & \text{if } j=i, \\ \max\left\{0,q_{ik}^{(c)}-\sum_{t\neq c} q_{ik}^{(t)}+\left(q_{ij}^{(c)}-q_{ij}^{(c-1)}\right)q_{jk}^{(0)}+q_{ij}^{(m)}\left(q_{jk}^{(c)}-q_{jk}^{(c+1)}\right)\right\}  & \text{if } i \neq j \neq k. \end{cases}\]

\[
\widetilde{q}_{ik}^{(c)}=
\begin{cases}
 q_{ik}^{(c+1)} & \text{if } j=k,  \\[4pt]
 q_{ik}^{(c-1)} & \text{if } j=i, \\[4pt]
 \begin{aligned}
 \max\Bigl\{0,\; q_{ik}^{(c)} - \sum_{t\neq c} q_{ik}^{(t)} 
 &+ \left(q_{ij}^{(c)} - q_{ij}^{(c-1)}\right) q_{jk}^{(0)} \\
 &+ q_{ij}^{(m)}\left(q_{jk}^{(c)} - q_{jk}^{(c+1)}\right)
 \Bigr\}
 \end{aligned}
 & \text{if } i \neq j \neq k.
\end{cases}
\]

Alternatively, this operation can be described by the following algorithm, see \cite{buan}.

\begin{enumerate}
\item For every pair of arrows  
\[
i\xlongrightarrow{(c)} j \xlongrightarrow{0} k
\]  
with $i\neq k,$ where $c$ is an arbitrary color, add an arrow from $i$ to $k$ of color $c$ and an arrow from $k$ to $i$ of color $m-c$.

\item If after Step 1 the quiver ceases to be monochromatic, i.e., if there exist arrows of two distinct colors between the vertices $i$ and $k$, then cancel an equal number of arrows of each color so that the quiver becomes monochromatic again.

\item Add one to each color of the arrows arriving at vertex $j$ and subtract one from the color of any arrow leaving $j$. This operation must be performed modulo $m+1.$
\end{enumerate}

Two quivers are \textit{mutation-equivalent} if one can be obtained from the other (and vice versa) through a finite sequence of mutations. An equivalence class in this sense is called a colored mutation class.

\begin{obs}
The mutation is a local operation in the sense that, for an $m-$colored quiver $Q$, the mutation $\mu_v$ at any vertex $v\in Q_0$ only modifies the shape of $Q$ by changing the colors of the arrows incident to $v$ and by possibly removing or adding arrows between the neighbors of $v$.  
\label{ref:local}
\end{obs}

For a colored quiver $Q$ with no loops, monochromatic, and skew-symmetric, the underlying graph of $Q$ is the graph obtained by replacing each pair of arrows
$\begin{tikzcd}
             i \arrow[r,"c" ',bend right=20]& j \arrow[l,"m-c" ',bend right=20] 
\end{tikzcd} $
with an  edge  \(\begin{tikzcd}[cramped] i \arrow[r,dash]& j. \end{tikzcd}\)

\medskip

In the sequel, a  quiver $Q$ whose underlying graph is a Dynkin diagram of type $\mathbb{A}_n$ (or $\widetilde{\mathbb{A}}_n$) will be called a colored quiver of type $\mathbb{A}_n$ (or $\widetilde{\mathbb{A}}_n$, respectively). %, or simply an $\mathbb{A}_n$-quiver.

\begin{figure}[h]
\[\begin{tikzcd}[row sep=tiny]
	1 & 2 & \cdot & \cdot & {n-1} & n
	\arrow[no head, from=1-1, to=1-2]
	\arrow[no head, from=1-2, to=1-3]
	\arrow[dashed, no head, from=1-3, to=1-4]
	\arrow[no head, from=1-4, to=1-5]
	\arrow[no head, from=1-5, to=1-6]
\end{tikzcd}\]
\caption{Dynkin diagram of type $\mathbb{A}_n$}
\end{figure}

\begin{figure}[h]
\[\begin{tikzcd}[column sep=small, row sep=small]
	& 1 && n \\
	2 &&&& \cdot \\
	\\
	\cdot &&&& \cdot \\
	& \cdot && \cdot
	\arrow[no head, from=4-1, to=5-2]
	\arrow[dashed, no head, from=5-2, to=5-4]
	\arrow[no head, from=5-4, to=4-5]
	\arrow[no head, from=4-5, to=2-5]
	\arrow[no head, from=1-2, to=1-4]
	\arrow[no head, from=2-5, to=1-4]
	\arrow[no head, from=1-2, to=2-1]
	\arrow[no head, from=2-1, to=4-1]
\end{tikzcd}\]
\caption{Dynkin diagram of type $\widetilde{\mathbb{A}}_n$}
\end{figure}
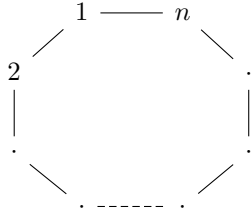

\subsection{Mutation class of type $\widetilde{\mathbb{A}}_{p,q}$}

Fix $p,q\geq 1$ two integers. Let $\widetilde{A}_{p,q}$ be the following colored quiver with arrows only of color $0$ and $m$ whose underlying graph is a Dynkin graph of type $\widetilde{\mathbb{A}}_n$ for $n=p+q$. See Figure 3. %\ref{fig:apq}.

% https://q.uiver.app/#q=WzAsOCxbMSwwLCJcXHRpbnl7MX0iXSxbMCwxLCJcXHRpbnl7Mn0iXSxbMiwwLCJcXHRpbnl7cCtxfSJdLFswLDIsIlxcdGlueXszfSJdLFsxLDMsIlxcdGlueXtwfSJdLFsyLDMsIlxcdGlueXtwKzF9Il0sWzMsMSwiXFx0aW55e3ArcS0xfSJdLFszLDIsIlxcdGlueXtwKzJ9Il0sWzAsMSwiMCIsMCx7ImNvbG91ciI6WzAsNjAsNjBdfSxbMCw2MCw2MCwxXV0sWzEsMywiMCIsMCx7ImNvbG91ciI6WzAsNjAsNjBdfSxbMCw2MCw2MCwxXV0sWzMsNCwiIiwyLHsiY29sb3VyIjpbMCw2MCw2MF0sInN0eWxlIjp7ImJvZHkiOnsibmFtZSI6ImRhc2hlZCJ9LCJoZWFkIjp7Im5hbWUiOiJub25lIn19fV0sWzQsNSwiMCIsMCx7ImNvbG91ciI6WzAsNjAsNjBdfSxbMCw2MCw2MCwxXV0sWzUsNywibSJdLFs3LDYsIiIsMix7ImNvbG91ciI6WzAsNjAsNjBdLCJzdHlsZSI6eyJib2R5Ijp7Im5hbWUiOiJkYXNoZWQifSwiaGVhZCI6eyJuYW1lIjoibm9uZSJ9fX1dLFs2LDIsIm0iXSxbMiwwLCJtIl0sWzEsMCwibSIsMCx7Im9mZnNldCI6LTN9XSxbMywxLCJtIiwwLHsib2Zmc2V0IjotM31dLFs1LDQsIm0iLDAseyJvZmZzZXQiOi0zfV0sWzcsNSwiMCIsMCx7Im9mZnNldCI6LTMsImNvbG91ciI6WzAsNjAsNjBdfSxbMCw2MCw2MCwxXV0sWzIsNiwiMCIsMCx7Im9mZnNldCI6LTMsImNvbG91ciI6WzAsNjAsNjBdfSxbMCw2MCw2MCwxXV0sWzAsMiwiMCIsMCx7Im9mZnNldCI6LTMsImNvbG91ciI6WzAsNjAsNjBdfSxbMCw2MCw2MCwxXV1d

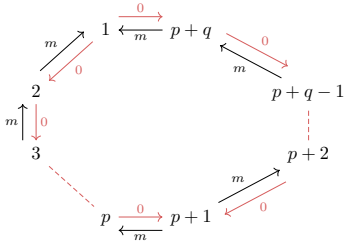
\begin{figure}[H]\label{fig:apq}
\adjustbox{scale=.7,center}{
\begin{tikzcd}
	& {\tiny{1}} & {\tiny{p+q}} \\
	{\tiny{2}} &&& {\tiny{p+q-1}} \\
	{\tiny{3}} &&& {\tiny{p+2}} \\
	& {\tiny{p}} & {\tiny{p+1}}
	\arrow["0", shift left=3, color={rgb,255:red,214;green,92;blue,92}, from=1-2, to=1-3]
	\arrow["0", color={rgb,255:red,214;green,92;blue,92}, from=1-2, to=2-1]
	\arrow["m", from=1-3, to=1-2]
	\arrow["0", shift left=3, color={rgb,255:red,214;green,92;blue,92}, from=1-3, to=2-4]
	\arrow["m", shift left=3, from=2-1, to=1-2]
	\arrow["0", color={rgb,255:red,214;green,92;blue,92}, from=2-1, to=3-1]
	\arrow["m", from=2-4, to=1-3]
	\arrow["m", shift left=3, from=3-1, to=2-1]
	\arrow[color={rgb,255:red,214;green,92;blue,92}, dashed, no head, from=3-1, to=4-2]
	\arrow[color={rgb,255:red,214;green,92;blue,92}, dashed, no head, from=3-4, to=2-4]
	\arrow["0", shift left=3, color={rgb,255:red,214;green,92;blue,92}, from=3-4, to=4-3]
	\arrow["0", color={rgb,255:red,214;green,92;blue,92}, from=4-2, to=4-3]
	\arrow["m", from=4-3, to=3-4]
	\arrow["m", shift left=3, from=4-3, to=4-2]
\end{tikzcd}}

\caption{The quiver $\widetilde{A}_{p,q}$}
%\label{fig:apq}
\end{figure}
 
Once we take a quiver with a cycle as above, we can fix one drawing of it, i.e., one embedding into the plane. Thus, we can speak of clockwise and counterclockwise oriented arrows of the cycle. But we have to consider that this notation is unique up to reflection of the cycle, i.e., up to changing the roles of clockwise and counterclockwise oriented arrows.

\medskip

By the mutation class of type $\widetilde{\mathbb{A}}_{p,q}$  we mean the set of all quivers mutation equivalent to the quiver $Q=\widetilde{A}_{p,q}$ above. \medskip

In the sequel, we fix an embedding of the cycle in such a way that there are   $p$ arrows of color $0$ counterclockwise oriented and $q$ arrows of color $0$ clockwise oriented. 

\medskip

\begin{ejem}\label{2-clase de A31}
The following figure shows all non-isomorphic (i.e., up to relabelling of the vertices) $2$-coloured quivers in the mutation class of type $\widetilde{\mathbb{A}}_{p,q}$ for $p=3$ and $q=1$. In the last three quivers, $c$ can range from $0$ to $2$.

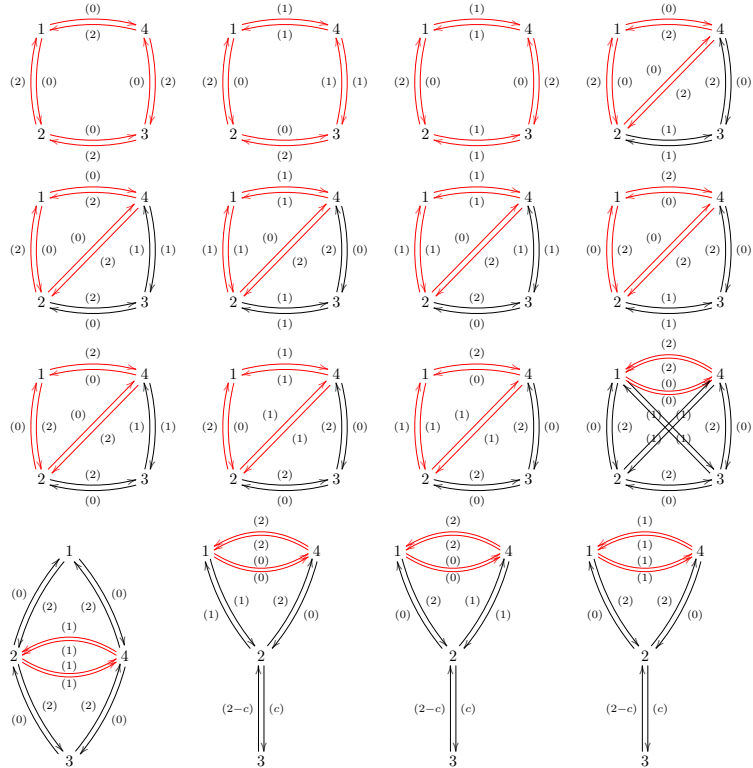
\begin{figure}[H]
\begin{center}
\scalebox{0.6}{$
\begin{array}{cccc}
\xymatrix{1 \ar@[red]@<0.4ex>@/_5pt/[dd]^{(0)} \ar@[red]@<0.4ex>@/^5pt/[rr]^{(0)} & & 4 \ar@[red]@<0.4ex>@/^5pt/[dd]^{(2)} \ar@[red]@<0.4ex>@/_5pt/[ll]^{(2)} \\ && \\ 2 \ar@[red]@<0.4ex>@/_5pt/[rr]^{(0)} \ar@[red]@<0.4ex>@/^5pt/[uu]^{(2)} & & 3 \ar@[red]@<0.4ex>@/_5pt/[uu]^{(0)} \ar@[red]@<0.4ex>@/^5pt/[ll]^{(2)}} & 
\xymatrix{1 \ar@[red]@<0.4ex>@/_5pt/[dd]^{(0)} \ar@[red]@<0.4ex>@/^5pt/[rr]^{(1)} & & 4 \ar@[red]@<0.4ex>@/^5pt/[dd]^{(1)} \ar@[red]@<0.4ex>@/_5pt/[ll]^{(1)} \\ && \\ 2 \ar@[red]@<0.4ex>@/_5pt/[rr]^{(0)} \ar@[red]@<0.4ex>@/^5pt/[uu]^{(2)} & & 3 \ar@[red]@<0.4ex>@/_5pt/[uu]^{(1)} \ar@[red]@<0.4ex>@/^5pt/[ll]^{(2)}} & 
\xymatrix{1 \ar@[red]@<0.4ex>@/_5pt/[dd]^{(0)} \ar@[red]@<0.4ex>@/^5pt/[rr]^{(1)} & & 4 \ar@[red]@<0.4ex>@/^5pt/[dd]^{(2)} \ar@[red]@<0.4ex>@/_5pt/[ll]^{(1)} \\ && \\ 2 \ar@[red]@<0.4ex>@/_5pt/[rr]^{(1)} \ar@[red]@<0.4ex>@/^5pt/[uu]^{(2)} & & 3 \ar@[red]@<0.4ex>@/_5pt/[uu]^{(0)} \ar@[red]@<0.4ex>@/^5pt/[ll]^{(1)}} & 
\xymatrix{1 \ar@[red]@<0.4ex>@/_5pt/[dd]^{(0)} \ar@[red]@<0.4ex>@/^5pt/[rr]^{(0)} & & 4 \ar@<0.4ex>@/^5pt/[dd]^{(0)} \ar@[red]@<0.4ex>@/_5pt/[ll]^{(2)} \ar@[red]@<0.4ex>[lldd]^{(2)} \\ && \\ 2 \ar@[red]@<0.4ex>[rruu]^{(0)} \ar@<0.4ex>@/_5pt/[rr]^{(1)} \ar@[red]@<0.4ex>@/^5pt/[uu]^{(2)} & & 3 \ar@<0.4ex>@/_5pt/[uu]^{(2)} \ar@<0.4ex>@/^5pt/[ll]^{(1)}} \\
%%%%%%%%%%%%%%%%%%%%%%%%%\hline
\xymatrix{1 \ar@[red]@<0.4ex>@/_5pt/[dd]^{(0)} \ar@[red]@<0.4ex>@/^5pt/[rr]^{(0)} & & 4 \ar@<0.4ex>@/^5pt/[dd]^{(1)} \ar@[red]@<0.4ex>@/_5pt/[ll]^{(2)} \ar@[red]@<0.4ex>[lldd]^{(2)} \\ && \\ 2 \ar@[red]@<0.4ex>[rruu]^{(0)} \ar@<0.4ex>@/_5pt/[rr]^{(2)} \ar@[red]@<0.4ex>@/^5pt/[uu]^{(2)} & & 3 \ar@<0.4ex>@/_5pt/[uu]^{(1)} \ar@<0.4ex>@/^5pt/[ll]^{(0)}} &
\xymatrix{1 \ar@[red]@<0.4ex>@/_5pt/[dd]^{(1)} \ar@[red]@<0.4ex>@/^5pt/[rr]^{(1)} & & 4 \ar@<0.4ex>@/^5pt/[dd]^{(0)} \ar@[red]@<0.4ex>@/_5pt/[ll]^{(1)} \ar@[red]@<0.4ex>[lldd]^{(2)} \\ && \\ 2 \ar@[red]@<0.4ex>[rruu]^{(0)} \ar@<0.4ex>@/_5pt/[rr]^{(1)} \ar@[red]@<0.4ex>@/^5pt/[uu]^{(1)} & & 3 \ar@<0.4ex>@/_5pt/[uu]^{(2)} \ar@<0.4ex>@/^5pt/[ll]^{(1)}} &
\xymatrix{1 \ar@[red]@<0.4ex>@/_5pt/[dd]^{(1)} \ar@[red]@<0.4ex>@/^5pt/[rr]^{(1)} & & 4 \ar@<0.4ex>@/^5pt/[dd]^{(1)} \ar@[red]@<0.4ex>@/_5pt/[ll]^{(1)} \ar@[red]@<0.4ex>[lldd]^{(2)} \\ && \\ 2 \ar@[red]@<0.4ex>[rruu]^{(0)} \ar@<0.4ex>@/_5pt/[rr]^{(2)} \ar@[red]@<0.4ex>@/^5pt/[uu]^{(1)} & & 3 \ar@<0.4ex>@/_5pt/[uu]^{(1)} \ar@<0.4ex>@/^5pt/[ll]^{(0)}} &
\xymatrix{1 \ar@[red]@<0.4ex>@/_5pt/[dd]^{(2)} \ar@[red]@<0.4ex>@/^5pt/[rr]^{(2)} & & 4 \ar@<0.4ex>@/^5pt/[dd]^{(0)} \ar@[red]@<0.4ex>@/_5pt/[ll]^{(0)} \ar@[red]@<0.4ex>[lldd]^{(2)} \\ && \\ 2 \ar@[red]@<0.4ex>[rruu]^{(0)} \ar@<0.4ex>@/_5pt/[rr]^{(1)} \ar@[red]@<0.4ex>@/^5pt/[uu]^{(0)} & & 3 \ar@<0.4ex>@/_5pt/[uu]^{(2)} \ar@<0.4ex>@/^5pt/[ll]^{(1)}} \\
%%%%%%%%%%%%%%%%%%%%%%%%%\hline
\xymatrix{1 \ar@[red]@<0.4ex>@/_5pt/[dd]^{(2)} \ar@[red]@<0.4ex>@/^5pt/[rr]^{(2)} & & 4 \ar@<0.4ex>@/^5pt/[dd]^{(1)} \ar@[red]@<0.4ex>@/_5pt/[ll]^{(0)} \ar@[red]@<0.4ex>[lldd]^{(2)} \\ && \\ 2 \ar@[red]@<0.4ex>[rruu]^{(0)} \ar@<0.4ex>@/_5pt/[rr]^{(2)} \ar@[red]@<0.4ex>@/^5pt/[uu]^{(0)} & & 3 \ar@<0.4ex>@/_5pt/[uu]^{(1)} \ar@<0.4ex>@/^5pt/[ll]^{(0)}}  &
\xymatrix{1 \ar@[red]@<0.4ex>@/_5pt/[dd]^{(0)} \ar@[red]@<0.4ex>@/^5pt/[rr]^{(1)} & & 4 \ar@<0.4ex>@/^5pt/[dd]^{(0)} \ar@[red]@<0.4ex>@/_5pt/[ll]^{(1)} \ar@[red]@<0.4ex>[lldd]^{(1)} \\ && \\ 2 \ar@[red]@<0.4ex>[rruu]^{(1)} \ar@<0.4ex>@/_5pt/[rr]^{(2)} \ar@[red]@<0.4ex>@/^5pt/[uu]^{(2)} & & 3 \ar@<0.4ex>@/_5pt/[uu]^{(2)} \ar@<0.4ex>@/^5pt/[ll]^{(0)}}&
\xymatrix{1 \ar@[red]@<0.4ex>@/_5pt/[dd]^{(1)} \ar@[red]@<0.4ex>@/^5pt/[rr]^{(2)} & & 4 \ar@<0.4ex>@/^5pt/[dd]^{(0)} \ar@[red]@<0.4ex>@/_5pt/[ll]^{(0)} \ar@[red]@<0.4ex>[lldd]^{(1)} \\ && \\ 2 \ar@[red]@<0.4ex>[rruu]^{(1)} \ar@<0.4ex>@/_5pt/[rr]^{(2)} \ar@[red]@<0.4ex>@/^5pt/[uu]^{(1)} & & 3 \ar@<0.4ex>@/_5pt/[uu]^{(2)} \ar@<0.4ex>@/^5pt/[ll]^{(0)}} & 
\xymatrix{1  \ar@<0.4ex>[rrdd]^{(1)} \ar@<0.4ex>@/_5pt/[dd]^{(2)} \ar@[red]@<0.6ex>@/_13pt/[rr]_{(0)} \ar@[red]@<0.1ex>@/_13pt/[rr]^{(0)} & & 4 \ar@[red]@<0.6ex>@/_13pt/[ll]^{(2)} \ar@[red]@<0.1ex>@/_13pt/[ll]_{(2)} \ar@<0.4ex>@/^5pt/[dd]^{(0)} \ar@<0.4ex>[lldd]^{(1)} \\ && \\ 2 \ar@<0.4ex>@/_5pt/[rr]^{(2)} \ar@<0.4ex>[rruu]^{(1)} \ar@<0.4ex>@/^5pt/[uu]^{(0)} && 3 \ar@<0.4ex>@/^5pt/[ll]^{(0)} \ar@<0.4ex>@/_5pt/[uu]^{(2)} \ar@<0.4ex>[lluu]^{(1)} }\\
%%%no sé como poner bien los 1 aquí. 
%%%%%%%%%%%%%%%%%%%%%%%%%\hline1
\xymatrix{&1  \ar@<0.4ex>@/^5pt/[rdd]^{(0)} \ar@<0.4ex>@/_5pt/[ldd]^{(2)}& \\ &&& \\2 \ar@<0.4ex>@/_5pt/[ddr]^{(2)} \ar@<0.4ex>@/^5pt/[uur]^{(0)} \ar@[red]@<0.6ex>@/_13pt/[rr]_{(1)} \ar@[red]@<0.1ex>@/_13pt/[rr]^{(1)} & & 4 \ar@<0.4ex>@/_5pt/[uul]^{(2)} \ar@<0.4ex>@/^5pt/[ddl]^{(0)} \ar@[red]@<0.6ex>@/_13pt/[ll]^{(1)} \ar@[red]@<0.1ex>@/_13pt/[ll]_{(1)} \\ &&& \\ & 3 \ar@<0.4ex>@/^5pt/[uul]^{(0)} \ar@<0.4ex>@/_5pt/[uur]^{(2)} & }
& \xymatrix{1 \ar@<0.4ex>@/_5pt/[ddr]^{(1)}  \ar@[red]@<0.6ex>@/_13pt/[rr]_{(0)} \ar@[red]@<0.1ex>@/_13pt/[rr]^{(0)} & & 4 \ar@<0.4ex>@/^5pt/[ddl]^{(0)} \ar@[red]@<0.6ex>@/_13pt/[ll]^{(2)} \ar@[red]@<0.1ex>@/_13pt/[ll]_{(2)} \\ &&& \\ & 2 \ar@<0.4ex>@/^5pt/[uul]^{(1)} \ar@<0.4ex>@/_5pt/[uur]^{(2)} \ar@<0.4ex>[dd]^{(c)} & \\ \\ & 3 \ar@<0.4ex>[uu]^{(2-c)} & } & 
\xymatrix{1 \ar@<0.4ex>@/_5pt/[ddr]^{(2)} \ar@[red]@<0.6ex>@/_13pt/[rr]_{(0)} \ar@[red]@<0.1ex>@/_13pt/[rr]^{(0)} & & 4 \ar@<0.4ex>@/^5pt/[ddl]^{(1)} \ar@[red]@<0.6ex>@/_13pt/[ll]^{(2)} \ar@[red]@<0.1ex>@/_13pt/[ll]_{(2)}  \\ &&& \\ & 2 \ar@<0.4ex>@/^5pt/[uul]^{(0)} \ar@<0.4ex>@/_5pt/[uur]^{(1)} \ar@<0.4ex>[dd]^{(c)} & \\ \\ & 3 \ar@<0.4ex>[uu]^{(2-c)} & } &
\xymatrix{1 \ar@<0.4ex>@/_5pt/[ddr]^{(2)} \ar@[red]@<0.6ex>@/_13pt/[rr]_{(1)} \ar@[red]@<0.1ex>@/_13pt/[rr]^{(1)} & & 4 \ar@<0.4ex>@/^5pt/[ddl]^{(0)} \ar@[red]@<0.6ex>@/_13pt/[ll]^{(1)} \ar@[red]@<0.1ex>@/_13pt/[ll]_{(1)} \\ &&& \\ & 2 \ar@<0.4ex>@/^5pt/[uul]^{(0)} \ar@<0.4ex>@/_5pt/[uur]^{(2)} \ar@<0.4ex>[dd]^{(c)} & \\ \\ & 3 \ar@<0.4ex>[uu]^{(2-c)} & }\\
%%%%%%%%%%%%%%%%%%%%%%%%%\hline
\end{array}
$}
\end{center}

\caption{The 2-colored mutation class of type  $\widetilde{\mathbb{A}}_{3,1}$.}
\label{fig:apq}
\end{figure}

\end{ejem}

\medskip

Since all of our colored quivers are skew-symmetric, the color of every arrow $\xymatrix{v_i \ar[r]^{(c_{ij})} & v_j}$ determines the color of the arrow $\xymatrix{v_j \ar[r]^{(c_{ji})} & v_i}$ according to the equation $c_{ji}=m-c_{ij}$. Then, it will cause no confusion if we simply draw $\xymatrix{v_i \ar[r]^{c_{ij}} & v_j}$ instead of $\xymatrix{v_i \ar@<0.6ex>^{(c_{ij})}[r] & v_j\ar@<0.6ex>^{(m-c_{ij})}[l] }$.

\subsection{Mutation class of type $\mathbb{A}_n$}

In \cite{gubitosi2024coloured} an explicit description of the $m-$colored quivers that appear in the $m-$colored mutation class of quivers of type $\mathbb{A}_n$ is given. This class was denoted by $\mathcal{Q}_n^m$, and is defined as follows.

\begin{defi}\cite[Definition 4.1]{gubitosi2024coloured} 
\label{ref: definicion clase A}
    Let $\mathcal{Q}_n^m$ be the class of connected, simple, $m-$colored quivers $Q$ with $n$ vertices and no holes, satisfying the following conditions:
\begin{enumerate}
    \item[(a)] For every vertex $v\in Q_0$ with $z\geq 1$ neighbors, there exist two cliques $\mathcal{C}_r$ and $\mathcal{B}_k$ such that 
    $v\in\mathcal{C}_r\cap\mathcal{B}_k$,
    where $r+k=z+2$, and $r,k\leq m+2.$ Additionally, there are no arrows between two vertices $i\in\mathcal{C}_r$ and  $j\in \mathcal{B}_k$.
    
    \item[(b)] Every triangle $T=(xyz)$ is $m-$admissible, i.e.,  $\overline{xy}+\overline{yz}+\overline{zx} \in \{m-1,2m+1\}$. 
\end{enumerate}
\end{defi}

The next proposition contains several fundamental properties of the class $\mathcal{Q}_n^m$.

\begin{prop}\cite[Lemmas 4.4, 4.7 and 4.8]{gubitosi2024coloured}
\label{ref:lema colores}
\begin{enumerate}
    \item The class $\mathcal{Q}^m_{n}$ is closed under colored mutations.
    \item  If $Q\in\mathcal{Q}^m_{n}$ and $v$ is a vertex in $Q$, then $\mu_v^{m+1}(Q)=Q.$
    \item Let $v$, $v_1$, and $v_2$ be the vertices of a triangle in the class $\mathcal{Q}_n^m$. Then, precisely one of the following holds: $\overline{vv_1}<\overline{vv_2}$ and $\overline{v_1v_2}=\overline{vv_2}-\overline{vv_1}-1$ or $\overline{vv_2}<\overline{vv_1}$ and $\overline{v_1v_2}=\overline{vv_2}-\overline{vv_1}+m+1$. 
    
    %with the color of the arrows label as in the following figure:
%https://q.uiver.app/#q=WzAsMyxbMCwxLCJ2Il0sWzIsMCwidl8xIl0sWzIsMiwidl8yIl0sWzAsMSwiY18xIl0sWzAsMiwiY18yIiwyXSxbMSwyLCJjX3sxMn0iXV0=
% \adjustbox{scale=.7,center}{\begin{tikzcd}
% 	&& {v_1} \\
% 	v \\
% 	&& {v_2}
% 	\arrow["{c_{12}}", from=1-3, to=3-3]
% 	\arrow["{c_1}", from=2-1, to=1-3]
% 	\arrow["{c_2}"', from=2-1, to=3-3]
% \end{tikzcd}}
% Then, precisely one of the following holds: $c_1<c_2$ and $c_{12}=c_2-c_1-1;$ or $c_2<c_1$ and $c_{12}=c_2-c_1+m+1.$

\end{enumerate}    
\end{prop}

\section{ Central cycles }

\begin{defi}
     A cycle \((a_1a_2\ldots a_l)\)  of length  $l$ in a $m$-colored quiver $Q$ is called $(l,h)-$\textit{central} if there exits   a natural number $h$ with $0<h<l$ such that
      $$\overline{(a_1a_2\ldots a_l)}=h\cdot m.$$
\end{defi}

\begin{obs} Notice that if $(a_1a_2\cdots a_l)$ is a cycle $(l,h)$-central then the cycle $(a_la_{l-1}\ldots a_1)$ is $(l,l-h)$-central, i.e., satisfies  
$\overline{(a_la_{l-1}\ldots a_1)}=(l-h)\cdot m.$
\end{obs}

In the following lemmas we will show that if a cycle is  $(l,h)$-central  then it is  mutation equivalent to the quiver $\widetilde{A}_{l-h,h}$.

\begin{lem}
\label{ref:lema camino}
Let $l$ and $h$ be positive integers with $l\geq 2$ and $0<h<l$. Consider the $m$-colored quiver $Q$,
% https://q.uiver.app/#q=WzAsNSxbMCwwLCJhXzEiXSxbMiwwLCJhXzIiXSxbNiwwLCJhX3tsLTF9Il0sWzgsMCwiYV9sIl0sWzQsMCwiXFxsZG90cyJdLFswLDEsImNfMSJdLFsyLDMsImNfe2wtMX0iXSxbMSw0LCJjXzIiXSxbNCwyLCJjX3tsLTJ9Il1d
\[\begin{tikzcd}
	{a_1} && {a_2} && \ldots && {a_{l-1}} && {a_l}
	\arrow["{c_1}", from=1-1, to=1-3]
	\arrow["{c_2}", from=1-3, to=1-5]
	\arrow["{c_{l-2}}", from=1-5, to=1-7]
	\arrow["{c_{l-1}}", from=1-7, to=1-9]
\end{tikzcd}\] with 
\[\sum_{i=1}^{l-1}c_{i}=h\cdot m.\] 
Then $Q$ is mutation-equivalent, without performing mutations at the vertices $a_1$ and $a_l$, to the quiver $Q^{\prime}$ 
\[\begin{tikzcd} {a_1} && {a_2} && \ldots && {a_{l-1}} && {a_l} \arrow["{m}", from=1-1, to=1-3] \arrow["{c_2^{\prime}}", from=1-3, to=1-5] \arrow["{c_{l-2}^{\prime}}", from=1-5, to=1-7] \arrow["{c_{l-1}^{\prime}}", from=1-7, to=1-9] \end{tikzcd}\] 
with \[\sum_{i=2}^{l-1}c_{i}^{\prime}=(h-1)\cdot m.\] 
\end{lem}

\begin{proof}
If $c_1=m$, there is nothing to do. If $c_1\neq m$, let us take $d=\min\left\{m-c_1,c_2\right\}$. By performing the operation $\mu_{a_{2}}^d(Q)$ we have two possibilities. If $d=m-c_1\leq c_2$ we obtain the quiver
 % https://q.uiver.app/#q=WzAsNSxbMiwwLCJhXzIiXSxbNCwwLCJhXzMiXSxbMCwwLCJhXzEiXSxbNiwwLCJcXGxkb3RzIl0sWzgsMCwiYV9sIl0sWzAsMSwiY18xK2NfMi1tIl0sWzIsMCwibSJdLFsxLDMsImNfMyJdLFszLDQsImNfe2wtMX0iXV0=
\[\begin{tikzcd}
	{a_1} && {a_2} && {a_3} && \ldots && {a_l}
	\arrow["m", from=1-1, to=1-3]
	\arrow["{c_1+c_2-m}", from=1-3, to=1-5]
	\arrow["{c_3}", from=1-5, to=1-7]
	\arrow["{c_{l-1}}", from=1-7, to=1-9]
\end{tikzcd}\]
and the statement is proved. If $d=c_2\leq m-c_1$, we obtain the quiver
% https://q.uiver.app/#q=WzAsNSxbMiwwLCJhXzIiXSxbNCwwLCJhXzMiXSxbMCwwLCJhXzEiXSxbNiwwLCJcXGxkb3RzIl0sWzgsMCwiYV9sIl0sWzAsMSwiMCJdLFsyLDAsImNfMStjXzIiXSxbMSwzLCJjXzMiXSxbMyw0LCJjX3tsLTF9Il1d
\[\begin{tikzcd}
	{a_1} && {a_2} && {a_3} && \ldots && {a_l}
	\arrow["{c_1+c_2}", from=1-1, to=1-3]
	\arrow["0", from=1-3, to=1-5]
	\arrow["{c_3}", from=1-5, to=1-7]
	\arrow["{c_{l-1}}", from=1-7, to=1-9]
\end{tikzcd}\]
In this case, if $c_i=0$ for all $3\leq i\leq l-1$, then \(h\cdot m=c_1+c_2\leq m,\) which is only possible if $h=1$ and therefore $c_1+c_2=m$. This would conclude the proof. Consequently, if there exists $j$, with $3\leq j\leq l-1$, such that $c_j\neq 0$, take $j$ to be the smallest of such indices. Observe that $c_i=0$ for all $3\leq i\leq j-1$. Therefore, we obtain a quiver of the form

% https://q.uiver.app/#q=WzAsNyxbMiwwLCJhXzIiXSxbMCwwLCJhXzEiXSxbNCwwLCJcXGNkb3RzIl0sWzYsMCwiYV9qIl0sWzgsMCwiYV97aisxfSJdLFsxMCwwLCJcXGNkb3RzIl0sWzEyLDAsImFfe2wtMX0iXSxbMSwwLCJjXzErY18yIl0sWzAsMiwiMCJdLFsyLDMsIjAiXSxbMyw0LCJjX2oiXSxbNCw1LCJjX3tqKzF9Il0sWzUsNiwiY197bC0xfSJdXQ==
\[\begin{tikzcd}[column sep=small]
	{a_1} && {a_2} && \cdots && {a_j} && {a_{j+1}} && \cdots && {a_{l}}
	\arrow["{c_1+c_2}", from=1-1, to=1-3]
	\arrow["0", from=1-3, to=1-5]
	\arrow["0", from=1-5, to=1-7]
	\arrow["{c_j}", from=1-7, to=1-9]
	\arrow["{c_{j+1}}", from=1-9, to=1-11]
	\arrow["{c_{l-1}}", from=1-11, to=1-13]
\end{tikzcd}\]

Applying the operation $\mu_{a_3}^{c_j}\circ\cdots\circ\mu_{a_j}^{c_j}$, we obtain the quiver

% https://q.uiver.app/#q=WzAsNyxbMiwwLCJhXzIiXSxbMCwwLCJhXzEiXSxbNCwwLCJcXGNkb3RzIl0sWzYsMCwiYV9qIl0sWzgsMCwiYV97aisxfSJdLFsxMCwwLCJcXGNkb3RzIl0sWzEyLDAsImFfe2wtMX0iXSxbMSwwLCJjXzErY18yIl0sWzAsMiwiY197an0iXSxbMiwzLCIwIl0sWzMsNCwiMCJdLFs0LDUsImNfe2orMX0iXSxbNSw2LCJjX3tsLTF9Il1d
\[\begin{tikzcd}[column sep = small]
	{a_1} && {a_2} && \cdots && {a_j} && {a_{j+1}} && \cdots && {a_{l}}
	\arrow["{c_1+c_2}", from=1-1, to=1-3]
	\arrow["{c_{j}}", from=1-3, to=1-5]
	\arrow["0", from=1-5, to=1-7]
	\arrow["0", from=1-7, to=1-9]
	\arrow["{c_{j+1}}", from=1-9, to=1-11]
	\arrow["{c_{l-1}}", from=1-11, to=1-13]
\end{tikzcd}\]

Again, we take $c_1^{\prime}=c_1+c_2$, $c_2^{\prime}=c_j$ and $d^{\prime}=\min\left\{m-c_1^{\prime},c_2^{\prime}\right\}$ and we  perform the mutation $\mu_{a_2}^{d^{\prime}}$ on the previous quiver. If $d^{\prime}=m-c_1^{\prime}\leq c_2^{\prime}$, this proves the statement. If $d^{\prime}=c_2^{\prime}\leq m-c_1^{\prime}$, we obtain the quiver

\[\begin{tikzcd}[column sep = small]
	{a_1} && {a_2} && \cdots && {a_j} && {a_{j+1}} && \cdots && {a_{l}}
	\arrow["{c_1+c_2+c_j}", from=1-1, to=1-3]
	\arrow["0", from=1-3, to=1-5]
	\arrow["0", from=1-5, to=1-7]
	\arrow["0", from=1-7, to=1-9]
	\arrow["{c_{j+1}}", from=1-9, to=1-11]
	\arrow["{c_{l-1}}", from=1-11, to=1-13]
\end{tikzcd}\]
Perform this process iteratively until obtaining the desired quiver. If for no $j$ the sum of the colors $c_1+\cdots+c_j=m$, with $j\leq l-2$, then we have the quiver

\[\begin{tikzcd}[column sep=small]
	{a_1} && {a_2} && \cdots && {a_j} && {a_{j+1}} && \cdots && {a_{l}}
	\arrow["{c_1+\cdots+c_{l-2}}", from=1-1, to=1-3]
	\arrow["{c_{l-1}}", from=1-3, to=1-5]
	\arrow["0", from=1-5, to=1-7]
	\arrow["0", from=1-7, to=1-9]
	\arrow["0", from=1-9, to=1-11]
	\arrow["0", from=1-11, to=1-13]
\end{tikzcd}\]
with \( c_1+\cdots+c_{l-2}<m\)
and \(h\cdot m = c_1+\cdots+c_{l-2}+c_{l-1} <2m,\) therefore $h=1$ and by performing the mutation $\mu_{a_2}^{c_{l-1}}$ we obtain the quiver

\[\begin{tikzcd}[column sep=small]
	{a_1} && {a_2} && \cdots && {a_j} && {a_{j+1}} && \cdots && {a_{l}}
	\arrow["m", from=1-1, to=1-3]
	\arrow["0", from=1-3, to=1-5]
	\arrow["0", from=1-5, to=1-7]
	\arrow["0", from=1-7, to=1-9]
	\arrow["0", from=1-9, to=1-11]
	\arrow["0", from=1-11, to=1-13]
\end{tikzcd}\]
\end{proof}

\begin{lem}
\label{ref:lema ciclo equivalente}
   Let $Q$ be the $m$-colored quiver  
   % https://q.uiver.app/#q=WzAsOCxbMSwwLCJhXzEiXSxbMiwwLCJhX3tsfSJdLFswLDEsImFfMiJdLFszLDEsImFfe2wtMX0iXSxbMCwyLCJhX2kiXSxbMywyLCJhX3tsLTJ9Il0sWzEsMywiYV97aSsxfSJdLFsyLDMsImFfe2krMn0iXSxbMCwyLCJcXG92ZXJsaW5le1xcYWxwaGFfMX0iLDAseyJjb2xvdXIiOlswLDYwLDYwXX0sWzAsNjAsNjAsMV1dLFsyLDQsIiIsMCx7ImNvbG91ciI6WzAsNjAsNjBdLCJzdHlsZSI6eyJib2R5Ijp7Im5hbWUiOiJkYXNoZWQifX19XSxbNCw2LCJcXG92ZXJsaW5le1xcYWxwaGFfaX0iLDAseyJjb2xvdXIiOlswLDYwLDYwXX0sWzAsNjAsNjAsMV1dLFs2LDcsIlxcb3ZlcmxpbmV7XFxhbHBoYV97aSsxfX0iLDAseyJjb2xvdXIiOlswLDYwLDYwXX0sWzAsNjAsNjAsMV1dLFs3LDUsIiIsMCx7ImNvbG91ciI6WzAsNjAsNjBdLCJzdHlsZSI6eyJib2R5Ijp7Im5hbWUiOiJkYXNoZWQifX19XSxbNSwzLCJcXG92ZXJsaW5le1xcYWxwaGFfe2wtMn19IiwwLHsiY29sb3VyIjpbMCw2MCw2MF19LFswLDYwLDYwLDFdXSxbMywxLCJcXG92ZXJsaW5le1xcYWxwaGFfe2wtMX19IiwwLHsiY29sb3VyIjpbMCw2MCw2MF19LFswLDYwLDYwLDFdXSxbMSwwLCJcXG92ZXJsaW5le1xcYWxwaGFfe2x9fSIsMCx7ImNvbG91ciI6WzAsNjAsNjBdfSxbMCw2MCw2MCwxXV1d

\adjustbox{scale=.7,center}{
\begin{tikzcd}
	& {a_1} & {a_{l}} \\
	{a_2} &&& {a_{l-1}} \\
	{a_i} &&& {a_{l-2}} \\
	& {a_{i+1}} & {a_{i+2}}
	\arrow["{c_1}", color=red, from=1-2, to=2-1]
	\arrow["{c_l}", color=red, from=1-3, to=1-2]
	\arrow[draw=red, dashed, no head, from=2-1, to=3-1]
	\arrow["{c_{l-1}}", color=red, from=2-4, to=1-3]
	\arrow["{c_i}", color=red, from=3-1, to=4-2]
	\arrow["{c_{l-2}}", color=red, from=3-4, to=2-4]
	\arrow["{c_{i+1}}", color=red, from=4-2, to=4-3]
	\arrow[draw=red, dashed, no head, from=4-3, to=3-4]
\end{tikzcd}}
where the cycle \((a_1a_2\cdots a_l)\) is $(l,h)$-central. Then $Q$ is mutation-equivalent to the quiver $\widetilde{A}_{l-h,h}$.
\label{ref:ciclo central}
\end{lem}

\begin{proof}
We may assume that $c_j=0$ for some $j$. Otherwise, if all colors $c_i$ are neither zero nor $m$; by taking $i=l$ and $c=\min\left\{m-c_{l-1},c_l\right\}$, the mutation $\mu_{a_{l}}^c(Q)$ guarantees that $c_l=0$ or $c_{l-1}=m$. Without loss of generality we can assume that $j=l.$

Since $l\geq 2$ and $0<h<l$, we have that not all colors $c_i$ are zero. Let $s$ be the smallest $i$ such that $c_i\neq 0$. Therefore, within $Q$ we have the following subquiver $Q^{\prime}$:

$\begin{tikzcd}
	{a_s} && {a_{s+1}} && \ldots && {a_{l-1}} && {a_l}
	\arrow["{c_s}", from=1-1, to=1-3]
	\arrow["{c_{s+1}}", from=1-3, to=1-5]
	\arrow["{c_{l-2}}", from=1-5, to=1-7]
	\arrow["{c_{l-1}}", from=1-7, to=1-9]
\end{tikzcd}$

with \[\sum_{i=s}^{l-1}c_{i}=h\cdot m.\]

Thanks to Lemma \autoref{ref:lema camino}, we have that $Q^{\prime}$ is mutation–equivalent to the quiver

\[\begin{tikzcd}
	{a_s} && {a_{s+1}} && \ldots && {a_{l-1}} && {a_l}
	\arrow["{m}", from=1-1, to=1-3]
	\arrow["{c_{s+1}^{\prime}}", from=1-3, to=1-5]
	\arrow["{c_{l-2}^{\prime}}", from=1-5, to=1-7]
	\arrow["{c_{l-1}^{\prime}}", from=1-7, to=1-9]
\end{tikzcd}\] with \[\sum_{i={s+1}}^{l-1}c_{i}^{\prime}=(h-1)\cdot m.\] 

If $h=1$, then the process ends, since $Q$ has been transformed, via mutations, into a cycle with $l-1$ arrows of color $0$ and $h=1$ arrows of color $m$. If $h>1$, apply Lemma \autoref{ref:lema camino} to the quiver

\[\begin{tikzcd}
	{a_{s+1}} && \ldots && {a_{l-1}} && {a_l}
	\arrow["{c_{s+1}^{\prime}}", from=1-1, to=1-3]
	\arrow["{c_{l-2}^{\prime}}", from=1-3, to=1-5]
	\arrow["{c_{l-1}^{\prime}}", from=1-5, to=1-7]
\end{tikzcd}\]
and thus we may assume that $c_{s+1}^{\prime}=m$. This process is repeated $h-1$ more times, transforming the first $h-1$ arrows of the subquiver
\[\begin{tikzcd}
	{a_{s+2}} && \ldots && {a_{l-1}} && {a_l}
	\arrow["{c_{s+2}^{\prime}}", from=1-1, to=1-3]
	\arrow["{c_{l-2}^{\prime}}", from=1-3, to=1-5]
	\arrow["{c_{l-1}^{\prime}}", from=1-5, to=1-7]
\end{tikzcd}\]
into color $m$ and the remaining $l-s-h$ arrows into color $0$, obtaining that $ Q^{\prime}$ is mutation–equivalent to the quiver
% https://q.uiver.app/#q=WzAsNyxbMCwwLCJhX3MiXSxbMiwwLCJhX3tzKzF9Il0sWzYsMCwiYV97cytofSJdLFs4LDAsImFfe3MraCsxfSJdLFsxMCwwLCJcXGNkb3RzIl0sWzQsMCwiXFxjZG90cyJdLFsxMiwwLCJhX2wiXSxbMCwxLCJtIl0sWzIsMywiMCJdLFszLDQsIjAiXSxbMSw1LCJtIl0sWzUsMiwibSJdLFs0LDYsIjAiXV0=
\[\begin{tikzcd}[column sep=small]
	{a_s} && {a_{s+1}} && \cdots && {a_{s+h}} && {a_{s+h+1}} && \cdots && {a_l}
	\arrow["m", from=1-1, to=1-3]
	\arrow["m", from=1-3, to=1-5]
	\arrow["m", from=1-5, to=1-7]
	\arrow["0", from=1-7, to=1-9]
	\arrow["0", from=1-9, to=1-11]
	\arrow["0", from=1-11, to=1-13]
\end{tikzcd}\]

Finally, $Q$ is mutation–equivalent to the quiver

% https://q.uiver.app/#q=WzAsMTEsWzMsMCwiYV9zIl0sWzQsMCwiYV97cysxfSJdLFs2LDAsImFfe3MraH0iXSxbNywwLCJhX3tzK2grMX0iXSxbOCwwLCJcXGNkb3RzIl0sWzUsMCwiXFxjZG90cyJdLFs5LDAsImFfbCJdLFsyLDAsIlxcY2RvdHMiXSxbMSwwLCJhXzIiXSxbMCwwLCJhXzEiXSxbMTAsMCwiYV8xIl0sWzAsMSwibSJdLFsyLDMsIjAiXSxbMyw0LCIwIl0sWzEsNSwibSJdLFs1LDIsIm0iXSxbNCw2LCIwIl0sWzcsMCwiMCJdLFs4LDcsIjAiXSxbOSw4LCIwIl0sWzYsMTAsIjAiXV0=
\[\begin{tikzcd}[column sep=small]
	{a_1} & {a_2} & \cdots & {a_s} & {a_{s+1}} & \cdots & {a_{s+h}} & {a_{s+h+1}} & \cdots & {a_l} & {a_1}
	\arrow["0", from=1-1, to=1-2]
	\arrow["0", from=1-2, to=1-3]
	\arrow["0", from=1-3, to=1-4]
	\arrow["m", from=1-4, to=1-5]
	\arrow["m", from=1-5, to=1-6]
	\arrow["m", from=1-6, to=1-7]
	\arrow["0", from=1-7, to=1-8]
	\arrow["0", from=1-8, to=1-9]
	\arrow["0", from=1-9, to=1-10]
	\arrow["0", from=1-10, to=1-11]
\end{tikzcd}\]

which is the quiver $\widetilde{A}_{l-h,h}.$

\end{proof}

Lemma \autoref{ref:lema ciclo equivalente} proves that all $(l,h)$–central cycles are mutation–equivalent and therefore the colored mutation class of a quiver of type $\widetilde{\mathbb{A}}_{p,q}$ can be defined as the colored mutation class of the quiver $\widetilde{A}_{p,q}$ described in Figure 3. %\autoref{fig:apq}.

\section{The class $\mathcal{Q}^m_{p,q}$}

In this section we define a special class of $m$-colored quivers with $p+q$ vertices which will turning
out the colored mutation class of type $\widetilde{\mathbb{A}}_{p,q}$.

\begin{defi}
\label{ref: definicion clase A tilde}
    Let $\mathcal{Q}^m_{p,q}$ be the class of connected, $m$-colored quivers $Q$ with $p+q$ vertices, satisfying the following conditions:
    \begin{enumerate}
        \item[(a)] There exists a unique set of special vertices $\{a_1, a_2, \dots, a_l\}$, with $l \ge 2$, which we will call \textit{central vertices}, such that the induced subquiver $Q[a_1, a_2, \dots, a_l]$ is a cycle $\Delta_l^m = (a_1 a_2 \dots a_l)$ that is $(l, h)$-central. 

        \item[(b)] If $l=2$, for every pair of vertices $u,v \notin \{a_1, a_2\}$, the number of arrows from $u$ to $v$ is at most one. Additionally, the central vertices are connected by double arrows: $q_{a_1 a_2}^{(c)}= q_{a_2 a_1}^{(m-c)} = 2$.  If $l\geq 3$, $Q$ is simple.  

        \item[(c)] If $2 \leq l\leq 3$, $Q$ has no holes. If, however, $l \geq 4$, $Q$ has no holes other than those determined by the cycle $\Delta_l^m$. 
        
\item[(d)]   For every  vertex $v\in Q_0$ with $z\geq 1$ neighbors, there exist two cliques $\mathcal{C}_r$ and $\mathcal{B}_k$ such that $v\in \mathcal{C}_r\cap\mathcal{B}_k$, where $1\leq r,k\leq m+2$, and $z=r+k-2-\delta_{l2}$.   Additionally, there are no arrows between two (different) vertices $u\in\mathcal{C}_r$ and  $w\in \mathcal{B}_k$, except for the case where $(uvw)=\Delta_3^m$ and we have only the central arrow $uw$.

\item[(e)]  For each arrow $\alpha_i:a_i\rightarrow a_{i+1}$ in the cycle $\Delta_l^m$  there exists a unique clique $D_{i}$, of size at most $m+2$, such that $\alpha_i\in (D_{i})_{1}$.

\item [(f)] Every triangle is $m$-admissible, with the exception of $\Delta ^m_3$, if applicable.
\end{enumerate}

\medskip
We will call $D_i$ a \textit{boundary subquiver}. Observe that $D_i $ can be just the clique $Q[a_i, a_{i+1}]$. %In addition, there are no arrows between vertices of  $D_i\setminus \Delta_l^m $ and  $D_j\setminus \Delta_l^m $ if $i\neq j$. 
If $w\in (D_i)_0\setminus \{a_i,a_{i+1}\}$ we will call it a
\textit{peripheral vertex} and we denote by  $W$ the set of all peripheral vertices of $Q$. 

For every $w\in W$, there exists a unique $1\leq i\leq l$ such that $w$ is a vertex of a triangle  $T_w=(a_i\,\,a_{i+1}\,\,w)$, which is a subquiver of $D_i$. (Here, we assume $l+1=1$). 

Then, we can decompose $W$ into two subsets $W_p=\left\{w\in W \,\, | \,\, \overline{T_w}=m-1\right\}$ and $W_q=\left\{w\in W\,\, | \,\, \overline{T_w}=2m+1\right\}$. 

In addition, if $\mathcal{A}(\C)$ denotes the set of all arrows belonging to $\C$, we have that $Q\setminus ( \mathcal{A}(\cup_i D_i)\cup \Delta^m_l)=\bigsqcup_{w\in W} Q_w$
where each  $Q_w\in \mathcal{Q}_{n_{w}}^m,$ for some $n_{w}\geq 0$. 

\medskip

\begin{enumerate}
\item [(g)] With the previous notations, if $x_{p}=\sum_{w\in W_p} n_w$ and  $x_{q}=\sum_{w\in W_q} n_w$, then: $p=l-h+x_{p}$ and $q= h+x_{q}$.

\end{enumerate}

\medskip
Observe that $h$ is uniquely determined as $h=l+x_{p}-p=q-x_{q}$. \\

For example, if $p+q=4$ and $m=2$, the class $\mathcal{Q}^2_{3,1}$ contains $(4,1)$-central cycles (see Example \ref{2-clase de A31}) but lacks  $(4,2)$-central cycles. However, the class $\mathcal{Q}^2_{2,2}$ contains $(4,2)$-central cycles but lacks  $(4,1)$-central cycles.

%\item Removing from $Q$ the arrows of all the boundary subquivers $D_i$, we obtain a disjoint union of all the central vertices together with subquivers $Q_w\in \mathcal{Q}_{n_{w}}^m,$ for $n_{w}\geq 0$, with $w\in W$. If we define the sets
%\[W_p=\left\{w\in W \,\, | \,\, \overline{T_w}=m-1\right\},\quad W_q=\left\{w\in W\,\, | \,\, \overline{T_w}=2m+1\right\},\] and the quantities
%\[x_{p}=\sum_{w\in W_p} n_w,\qquad x_{q}=\sum_{w\in W_q} n_w,\] then \[p=l-h+x_{p},\quad q= h+x_{q}. \]
%\end{enumerate}   

\end{defi}

\medskip
The quiver in \autoref{ref:clase} illustrates Definition \ref{ref: definicion clase A tilde} with a central cycle $\Delta^m_l=(a_1\ldots a_l)$ and a boundary subquiver $D_i=Q[a_i,a_{i+1},w_1,\ldots,w_{k_i}]$.

\begin{figure}[H]
\adjustbox{scale=.9,center}{
\begin{tikzpicture}[scale=0.55, x=2cm, y=2cm,  every node/.style={inner sep=1pt}]

\draw[smooth cycle, tension=1] plot coordinates{(-0.2,-0.4)(1,0.7)(1.4,0.3)(0.4,-0.5)};

\draw[smooth cycle, tension=1] plot coordinates{(-0.1,1.1)(1,1.8)(1.5,1.5)(0.4,0.8)};

\draw[smooth cycle, tension=1] plot coordinates{(3.3,-0.4)(2.7,0.7)(3.6,0.5)(4.3,-0.7)};

% --- Capa alta ---
\node (qw1) at (0.4,2) {$Q_{w_1}$};
\node (qwj) at (1.2,-0.6) {$Q_{w_2}$};
\node (qwki) at (4.5,0.1) {$Q_{w_{k_i}}$};
\node (a1)  at (3,3.7) {$a_{1}$};
\node (al)  at (4,3.7) {$a_{l}$};
\node (a2)  at (2,3) {$a_{2}$};
\node (al1) at (5.3,3) {$a_{l-1}$};
\node (al2) at (5.3,2) {$a_{l-2}$};

% --- Pentágono {a_i, a_{i+1}, w_1, w_j, w_{k_i}} ---
\node (ai)   at (2,2.1)   {$a_{i}$};
\node (ai1)  at (3,1.6)   {$a_{i+1}$};
\node (wki)  at (3,0.5)   {$w_{k_i}$};
\node (wj)   at (1.2,0.5)   {$w_{2}$};
\node (w1)   at (1.2,1.6) {$w_{1}$};

% --- a_{i+2} para α_{i+1} ---
\node (ai2)  at (4.3,1.6)   {$a_{i+2}$};

% --- Bullets de las Q ---
\node (bL)  at (0,1)   {.};
\node (bBL) at (0,-0.4)  {.};
\node (bBR) at (4,-0.4)  {.};

% --- Flechas (α con etiqueta; Q sin texto) ---
\draw[->,red]        (a1)  -- node[above] {\small $\alpha_{1}$}   (a2);
\draw[->,red]        (al)  -- node[above] {\small $\alpha_{l}$}   (a1);
\draw[-,red,dashed]  (a2)  --                                     (ai);
\draw[->,red]        (al1) -- node[right] {\small $\alpha_{l-1}$} (al);
\draw[->,red]        (al2) -- node[right] {\small $\alpha_{l-2}$} (al1);

\draw[-]            (ai)  --  (w1);
\draw[->,red]        (ai)  -- node[above] {\small $\alpha_{i}$}   (ai1);
\draw[->,red]        (ai1) -- node[above] {\small $\alpha_{i+1}$} (ai2);

% *** Conexión pedida ***
\draw[-,red,dashed] (ai2) -- (al2);  % a_{i+2} -> a_{l-2}

% Conexiones del pentágono / externas
\draw[-]            (ai1) -- (wki);
\draw[-]            (wj)  -- (ai);
\draw[-]            (wj)  -- (ai1);
\draw[-]            (w1)  -- (ai1);
\draw[-]            (w1)  -- (wki);
\draw[-,dashed]      (wj)  -- (wki);

% Flechas Q (sin label en la flecha)
\draw[-,dashed] (w1) -- (bL);   % Q_{w_1}
\draw[-,dashed] (wj)  -- (bBL); % Q_{w_j}
\draw[-,dashed] (wki) -- (bBR); % Q_{w_{k_i}}

% Extras pedidas
\draw[-] (w1) -- (wj);  % w1 -> wj
\draw[-]       (wki) -- (ai); % w_{k_i} -> a_i
\end{tikzpicture}}
\caption{Example of a quiver in $\mathcal{Q}^m_{p,q}$. }
\label{ref:clase}
\end{figure}
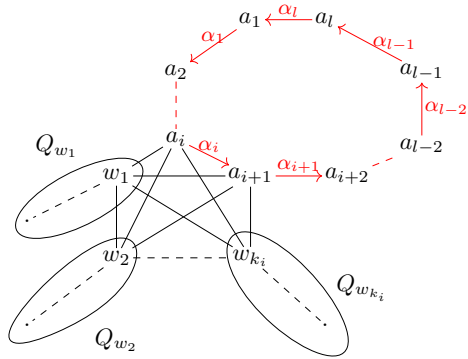

For a quiver $Q\in \mathcal{Q}^m_{p,q}$, it will be possible to apply the algorithm developed in \cite{gubitosi2024coloured} to the subquivers $Q_w$, never mutating at the peripheral vertex $w$, so that $Q_w$ is mutation–equivalent to a quiver of type $\mathbb{A}_{n_w}$, with $n_w=\vert (Q_w)_0\vert$.

\medskip

The class $\mathcal{Q}_n^m$ has a special type of cliques, called almost extremal cliques. Let $Q\in \mathcal{Q}_n^m$ and let $\C$ be a clique in $Q$ with at least three vertices. If $\mathcal{A}(\C)$ denotes the set of all arrows belonging to $\C$, then the clique $\C$ is said to be \textit{almost extremal} if there exists a connected component of $Q\setminus \mathcal{A}(\C)$ that is a quiver of type $\mathbb{A}_k$ for some $k\geq 1.$\medskip

For a quiver $Q\in \mathcal{Q}^m_{p,q}$, the algorithm developed by \cite{gubitosi2024coloured} can be applied to the subquivers $Q_w$ to show that $Q_w$ is mutation equivalent to a quiver of type $\mathbb{A}_{n_w}$, where $n_w=\vert (Q_w)_0\vert$. Critically, this is done without ever mutating on the peripheral vertex $w$.\medskip

To achieve this, it suffices to observe that any almost extremal clique in the subquiver $Q_w$ remains almost extremal when viewed in the quiver $Q_w\cup D$. Consequently, the algorithm presented in \cite[Theorem 5.9 ]{gubitosi2024coloured} guarantees the elimination of all cliques in $Q_w$ and its transformation into a quiver of type $\mathbb{A}_{n_w}$ without mutating at vertex $w$.\medskip

The preceding can be summarized in the following lemmas. 

\begin{lem}
\label{ref:prop_casi_extemal} Take $Q\in \mathcal{Q}^m_{p,q}$. Let $w$ be a peripheral vertex of $Q$. Let $Q_w\in \mathcal{Q}_{n_w}^m$ and let  $D$ be the corresponding boundary subquiver. Suppose that $Q_w$ is not of type $\mathbb{A}_{n_w}$ (with $n_w = \vert (Q_w)_0 \vert$). Then,  there exists a clique $\mathcal{C}$ in $Q_w$ that is almost extremal in $Q_w \cup D$, and therefore in $Q$.

\end{lem}

\begin{proof}
   Observe that it can be  proved exactly as  \cite[Lemma 5.3]{gubitosi2024coloured} by considering  the longest path \[\mathcal{P}: \,\, w=v_1\longrightarrow v_2 \longrightarrow \cdots \longrightarrow v_t.\] in $Q_w$ starting at $w$ that does not repeat vertices. 
\end{proof} 

\begin{lem}\label{Qw es mut eq a Anw}
 Let  $w\in W$  be a peripheral vertex and $Q_w$ the corresponding quiver belonging to the class  $\mathcal{Q}_{n_w}^m$ with $n_w\geq 1$. Then, $Q_w$ is mutation equivalent to an $\mathbb{A}_{n_w}$- quiver where all mutations are performed on vertices of $Q_w$ other than $w$. 
\end{lem}

\begin{proof}
 Suppose that $Q_w$ is not an $\mathbb{A}_{n_w}$- quiver. By the previous lemma, $Q_w \cup D$ has an almost extremal clique  $\mathcal{C}$, which is not $D$. It is clear that the algorithm in \cite[Theorem 5.9]{gubitosi2024coloured} allows us to perform mutations on vertices of $Q_w$ other than $w$ such that $Q_w$ is mutation equivalent to a quiver of type $\mathbb{A}_{n_w}$ in the case where the clique $\mathcal{C}$ does not contain the vertex $w$. If instead, $w$ is a vertex in $\mathcal{C}$, the following procedure of mutations allows us to "move" the clique one place in such a way that the new obtained clique $\mathcal{C'}$ does not contain the vertex $w$.

Let $\mathcal{C}=Q[w,v_1,\ldots,v_r]$. We denote by $d_i$ the color of the arrow $v_1\xlongrightarrow{d_i} w$ in $\mathcal{C}$, and by $c_{ij}$ the color of the arrow $v_i\xlongrightarrow{c_{ij}} v_j$ (for all $i\neq j$). The following figure illustrates this situation.

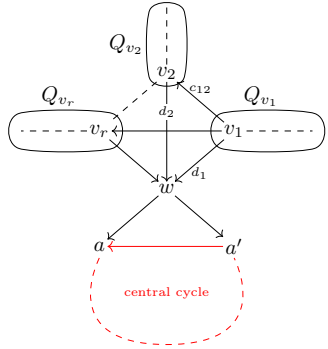
\adjustbox{scale=.8,center}{\begin{tikzpicture}[scale=0.8, x=1.4cm,y=1.2cm, every node/.style={inner sep=1pt}]
  %%% NODOS (fila r, columna c) en (c, -r)
  % Row 1
  \node (r1c5) at (5,-1) {$v_2$};
  \node (r2c4) at (4,-2) {$v_{r}$};
  \node (r2c6) at (6,-2) {$v_1$};
  %\node (x) at (6.8,-2) {$x$};
  \node (r2c8) at (7.2,-2) {};
  \node (r2c1) at (2.8,-2) {};
  \node (r0c6) at (5,0.2) {};
  \node (r2c9) at (6.4,-1.4) {$Q_{v_1}$};
  \node (r3c1) at (3.4,-1.4) {$Q_{v_r}$};
  \node (r0c1) at (4.4,-0.5) {$Q_{v_2}$};
  \draw[dashed,-] (r2c6) -- (r2c8);
  \node (r3c5) at (5,-3) {$w$};
  \node (r4c4) at (4,-4) {$a$};
  \node (r4c6) at (6,-4) {$a^\prime$};
  \node (r5c4) at (4,-5) {};
  \node (r5c6) at (6,-5) {};

  \draw[smooth cycle, tension=0.5] plot coordinates{(5.8,-1.7)(5.8,-2.3)(7.2,-2.3)(7.2,-1.7)};
  \draw[smooth cycle, tension=0.5] plot coordinates{(2.8,-1.7)(2.8,-2.3)(4.2,-2.3)(4.2,-1.7)};
  \draw[smooth cycle, tension=0.5] plot coordinates{(4.75,-1.1)(4.75,0.1)(5.25,0.1)(5.25,-1.1)};
  \draw[dashed,-] (r1c5) -- (r2c4);
  \draw[dashed,-] (r2c4) -- (r2c1);
   \draw[dashed,-] (r1c5) -- (r0c6);
  \draw[->] (r1c5) -- node[fill=white, pos=0.3]{\tiny  $d_2$} (r3c5);
  \draw[->] (r2c4) -- (r3c5);
  \draw[->] (r2c6) -- node[above, yshift=3pt]{\tiny  $c_{12}$}(r1c5);
  \draw[->] (r2c6) -- (r2c4);
  \draw[->] (r2c6) -- node[below, yshift=-2pt]{\tiny  $d_1$} (r3c5);
  \draw[->] (r3c5) -- (r4c4);
  \draw[->] (r3c5) -- (r4c6);
  \draw[red,->] (r4c6) -- (r4c4);
  \draw[red, smooth, dashed, tension=3] plot coordinates{(4,-4.2) (5,-5.7) (6,-4.2)};
  \node[red] (CC) at (5,-4.8) {\tiny central cycle};
\end{tikzpicture}}

Without loss of generality, we may assume that $d_1$ is the smallest among the colors $d_i$. Since $d_1<d_j$ for all $j\neq 1$, we have that $c_{1j}=d_1-d_j+m+1<m$. If $d_1>c_{1j}$ for some $j$, then $d_j>m+1$, which is impossible. Therefore, $d_1<c_{1j}$ for all $j\neq 1$.
By \cite[Theorem 5.9]{gubitosi2024coloured} we may assume that $Q_{v_1}$ has the form

%\medskip

\adjustbox{scale=.8,center}{\begin{tikzcd}
	{v_1} && {x_1} && {\cdots} && {x_r}
	\arrow["{d_1}", from=1-1, to=1-3]
	\arrow[from=1-3, to=1-5]
	\arrow[from=1-5, to=1-7]
\end{tikzcd}}
with $r\geq 0.$ If $r\geq 1$, by performing the mutations $\mu_{v_1}^{d_1+1}$ we obtain

\medskip

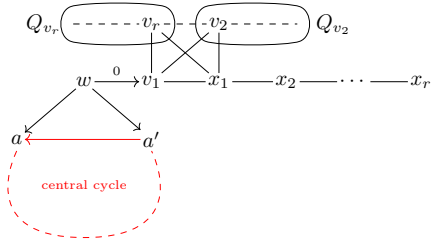
\adjustbox{scale=.8,center}{\begin{tikzpicture}[scale=0.8, x=1.4cm,y=1.2cm, every node/.style={inner sep=1pt}]
  %%% NODOS (fila r, columna c) en (c, -r)
  \draw[smooth cycle, tension=0.5] plot coordinates{(6.8,-1.7)(6.8,-2.3)(8.2,-2.3)(8.2,-1.7)};
  \draw[smooth cycle, tension=0.5] plot coordinates{(4.8,-1.7)(4.8,-2.3)(6.2,-2.3)(6.2,-1.7)};
  % Row 1
  \node (r0c1) at (4.4,-2) {$Q_{v_r}$};
  \node (r0c1) at (8.7,-2) {$Q_{v_2}$};
  \node (r1c15) at (4.8,-2) {};
  \node (r1c6) at (8.2,-2) {};
  \node (r1c5) at (6,-2) {$v_r$};
  \node (r2c4) at (7,-2) {$v_{2}$};
  \node (r2c6) at (6,-3) {$v_1$};
  \node (x) at (7,-3) {$x_1$};
  \node (r2c8) at (8,-3) {$x_2$};
  \node (r2c9) at (9,-3) {$\cdots$};
  \node (r2c10) at (10,-3) {$x_r$};
  \node (r3c5) at (5,-3) {$w$};
  \node (r4c4) at (4,-4) {$a$};
  \node (r4c6) at (6,-4) {$a^\prime$};
  \node (r5c4) at (4,-5) {};
  \node (r5c6) at (6,-5) {};

  %%% FLECHAS 
  \draw[-] (x) -- (r2c8);
  \draw[-] (r2c8) -- (r2c9);
  \draw[-] (r2c9) -- (r2c10);
  \draw[dashed,-] (r2c4) -- (r1c6);
  \draw[dashed,-] (r1c15) -- (r1c5);
  \draw[dashed,-] (r1c5) -- (r2c4);
  \draw[-] (r2c6) -- (x);
  \draw[-] (x) -- (r2c4);
  \draw[-] (x) -- (r1c5);
  \draw[-] (r2c6) -- (r1c5);
  \draw[-] (r2c6) -- (r2c4);
  \draw[->] (r3c5) -- node[above, yshift=2pt]{\tiny  $0$} (r2c6);
  \draw[->] (r3c5) -- (r4c4);
  \draw[->] (r3c5) -- (r4c6);
  \draw[red,->] (r4c6) -- (r4c4);
  \draw[red, smooth, dashed, tension=3] plot coordinates{(4,-4.2) (5,-5.7) (6,-4.2)};
  \node[red] (CC) at (5,-4.8) {\tiny central cycle};
\end{tikzpicture}}

Note that if $r=0$, we obtain the following particular situation.

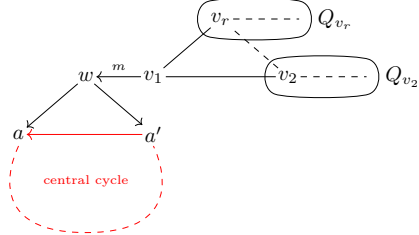
\adjustbox{scale=.8,center}{\begin{tikzpicture}[scale=0.8, x=1.4cm,y=1.2cm, every node/.style={inner sep=1pt}]
  %%% NODOS (fila r, columna c) en (c, -r)
   \draw[smooth cycle, tension=0.5] plot coordinates{(6.8,-1.7)(6.8,-2.3)(8.2,-2.3)(8.2,-1.7)};
   \draw[smooth cycle, tension=0.5] plot coordinates{(7.8,-2.7)(7.8,-3.3)(9.2,-3.3)(9.2,-2.7)};
  % Row 1
  \node (r0c1) at (9.7,-3) {$Q_{v_2}$};
  \node (r0c1) at (8.7,-2) {$Q_{v_r}$};
  \node (r1c15) at (8.2,-2) {};
  \node (r1c6)  at (9.2,-3) {};
  \node (r1c5) at (8,-3) {$v_2$};
  \node (r2c4) at (7,-2) {$v_{r}$};
  \node (r2c6) at (6,-3) {$v_1$};
  \node (r3c5) at (5,-3) {$w$};
  \node (r4c4) at (4,-4) {$a$};
  \node (r4c6) at (6,-4) {$a^\prime$};
  \node (r5c4) at (4,-5) {};
  \node (r5c6) at (6,-5) {};

  %%% FLECHAS
  \draw[dashed,-] (r2c4) -- (r1c15);
  \draw[dashed,-] (r1c5) -- (r1c6);
  \draw[dashed,-] (r1c5) -- (r2c4);
  \draw[-] (r2c6) -- (r1c5);
  \draw[-] (r2c6) -- (r2c4);
  \draw[->] (r2c6) -- node[above, yshift=2pt]{\tiny  $m$} (r3c5);
  \draw[->] (r3c5) -- (r4c4);
  \draw[->] (r3c5) -- (r4c6);
  \draw[red,->] (r4c6) -- (r4c4);
  \draw[red, smooth, dashed, tension=3] plot coordinates{(4,-4.2) (5,-5.7) (6,-4.2)};
  \node[red] (CC) at (5,-4.8) {\tiny central cycle};
\end{tikzpicture}}

Therefore, the clique $\mathcal{C}$ was moved one step away from vertex $w$, in such a way that the new obtained clique $\mathcal{C'}=Q'[x_1, v_1, \ldots, v_r]$ does not contain the vertex $w$. It is thus clear that the algorithm from \cite[Theorem 5.9]{gubitosi2024coloured} allows us to eliminate the clique without mutating at $w$.

\end{proof}

We continue with an important property of the class $\mathcal{Q}^m_{p,q}$.

\begin{prop} \label{ref:cerrado} The class $\mathcal{Q}^m_{p,q}$ is closed under colored quiver mutation.
   
\end{prop}

\begin{proof}
Let $Q$ be a quiver in the class $\mathcal{Q}^m_{p,q}$ and let $v$ be a vertex of $Q$. 
Let $Q^{\prime}=\mu_v(Q)$, we want to prove that $Q^{\prime}$ belongs to the class $\mathcal{Q}^m_{p,q}.$ 

 Similar to the $\mathbb{A}$-case, a vertex $v$ may have at most two outgoing arrows of color zero.
 If no such arrows exist, the mutation $\mu_v$ exclusively alters the colors of the arrows incident to $v$ and everything else remains unaffected. 
 We will now examine the cases where $v$ has one or two outgoing arrows of color zero. 
 In both cases, if $v$ is neither central nor peripheral,   $v$ belongs to a subquiver of $Q$ that belongs to the class $\mathcal{Q}^m_{n}$ that is closed under colored mutations by Proposition \autoref{ref:lema colores}.
 
 We first assume that $v$ has only one outgoing arrow $ v \xlongrightarrow{0} b$ of color zero. In this case, if $v$ belongs to a clique $\mathcal{C}$ but $b$ does not, the size of the clique $\mathcal{C}$ must be at most $m+1$. If both $v$ and $b$ belong to the same clique $\mathcal{D}$, its size can reach the value $m+2$. 
 In light of the previous considerations, we have the following  4 cases. \\
 
%\begin{enumerate}    
    %\item 
    
(1) Both $v$ and $b$ are central vertices. In this case, the mutation $\mu_v$ acts as illustrated in the following figure:

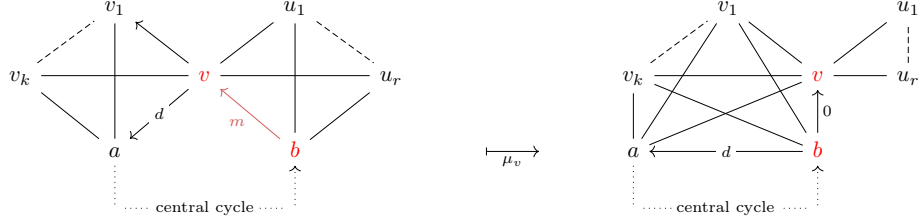
\adjustbox{scale=.85,center }{
\begin{tikzcd}
	& {v_1} && {u_1} &&&&& {v_1} && {u_1} \\
	{v_k} && {\textcolor{red}{v}} && {u_r} &&& {v_k} && {\textcolor{red}{v}} & {u_r} \\
	& a && {\textcolor{red}{b}} && {} & {} & a && {\textcolor{red}{b}} \\
	& {} && {} &&&& {} && {}
	\arrow[no head, from=1-2, to=3-2]
	\arrow[dashed, no head, from=1-4, to=2-5]
	\arrow[no head, from=1-9, to=3-8]
	\arrow[dashed, no head, from=2-1, to=1-2]
	\arrow[no head, from=2-1, to=3-2]
	\arrow[from=2-3, to=1-2]
	\arrow[no head, from=2-3, to=1-4]
	\arrow[no head, from=2-3, to=2-1]
	\arrow[no head, from=2-3, to=2-5]
	\arrow["d"{description}, from=2-3, to=3-2]
	\arrow[dashed, no head, from=2-8, to=1-9]
	\arrow[no head, from=2-8, to=3-8]
	\arrow[no head, from=2-10, to=1-9]
	\arrow[no head, from=2-10, to=1-11]
	\arrow[no head, from=2-10, to=2-8]
	\arrow[no head, from=2-10, to=2-11]
	\arrow[dashed, no head, from=2-11, to=1-11]
	\arrow[dotted, no head, from=3-2, to=4-2]
	\arrow[no head, from=3-4, to=1-4]
	\arrow["m", color={rgb,255:red,214;green,92;blue,92}, from=3-4, to=2-3]
	\arrow[no head, from=3-4, to=2-5]
	\arrow["{\mu_v}"', maps to, from=3-6, to=3-7]
	\arrow[no head, from=3-8, to=2-10]
	\arrow[dotted, no head, from=3-8, to=4-8]
	\arrow[no head, from=3-10, to=1-9]
	\arrow[no head, from=3-10, to=2-8]
	\arrow["0"', from=3-10, to=2-10]
	\arrow["d"{description}, from=3-10, to=3-8]
	\arrow["{{\scriptsize\text{central cycle}}}"{description}, dotted, no head, from=4-2, to=4-4]
	\arrow[dotted, from=4-4, to=3-4]
	\arrow["{\scriptsize{\text{central cycle}}}"{description}, dotted, no head, from=4-8, to=4-10]
	\arrow[dotted, from=4-10, to=3-10]
\end{tikzcd}}

The mutation $\mu_v$ transforms the $(l,h)$-central cycle $(b v\ldots a)$ into the $(l-1,h-1)$-central cycle $(a\ldots b)$ in $Q^\prime$. As noted above, the size of the clique $\mathcal{C}=Q[v,v_1, \cdots, v_k]$ must be at most $m+1$ and consequently,  the size of the clique $Q'[v,v_1, \cdots, v_k,a,b]$  is bounded by $m+2$ and it is a new boundary clique in $Q'$. Clearly, the boundary clique $Q[b,v,u_1, \cdots, u_r]$ changes to a smaller clique $Q'[v,u_1, \cdots, u_r]$.\\

%\item 
(2) $v$ is central and $b$ is peripheral. The mutation $\mu_v$ acts as is illustrated in the following figure:\\

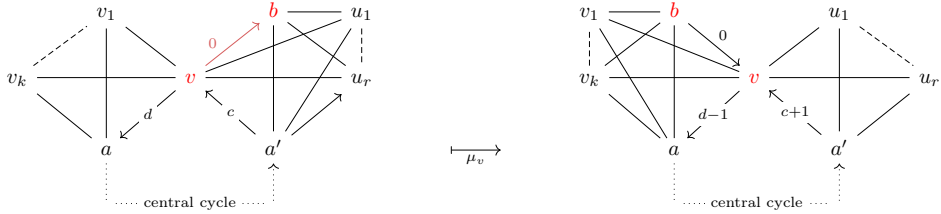
\adjustbox{width=1\linewidth,center}{\begin{tikzcd}
	& {v_1} && {\textcolor{red}{b}} & {u_1} &&& {v_1} & {\textcolor{red}{b}} && {u_1} \\
	{v_k} && {\textcolor{red}{v}} && {u_r} &&& {v_k} && {\textcolor{red}{v}} && {u_r} \\
	& a && {a'} && {} & {} && a && {a'} \\
	& {} && {} &&&&& {} && 
	\arrow[no head, from=1-2, to=3-2]
	\arrow[no head, from=1-4, to=1-5]
	\arrow[no head, from=1-4, to=2-5]
	\arrow[dashed, no head, from=1-5, to=2-5]
	\arrow[no head, from=1-8, to=3-9]
	\arrow[no head, from=1-9, to=1-8]
	\arrow[no head, from=1-9, to=2-8]
	\arrow["0", from=1-9, to=2-10]
	\arrow[dashed, no head, from=1-11, to=2-12]
	\arrow[dashed, no head, from=2-1, to=1-2]
	\arrow[no head, from=2-1, to=3-2]
	\arrow[no head, from=2-3, to=1-2]
	\arrow["0"{pos=0.3}, color={rgb,255:red,214;green,92;blue,92}, from=2-3, to=1-4]
	\arrow[no head, from=2-3, to=1-5]
	\arrow[no head, from=2-3, to=2-1]
	\arrow[no head, from=2-3, to=2-5]
	\arrow["d"{description}, from=2-3, to=3-2]
	\arrow[dashed, no head, from=2-8, to=1-8]
	\arrow[no head, from=2-8, to=3-9]
	\arrow[no head, from=2-10, to=1-8]
	\arrow[no head, from=2-10, to=1-11]
	\arrow[no head, from=2-10, to=2-8]
	\arrow[no head, from=2-10, to=2-12]
	\arrow["{\scriptsize{d-1}}"{description}, from=2-10, to=3-9]
	\arrow[dotted, no head, from=3-2, to=4-2]
	\arrow[no head, from=3-4, to=1-4]
	\arrow[no head, from=3-4, to=1-5]
	\arrow["c"{description}, from=3-4, to=2-3]
	\arrow[from=3-4, to=2-5]
	\arrow["{\mu_v}"', maps to, from=3-6, to=3-7]
	\arrow[no head, from=3-9, to=1-9]
	\arrow[dotted, no head, from=3-9, to=4-9]
	\arrow[no head, from=3-11, to=1-11]
	\arrow["{\scriptsize{c+1}}"{description}, from=3-11, to=2-10]
	\arrow[no head, from=3-11, to=2-12]
	\arrow["{\scriptsize{\text{central cycle}}}"{description}, dotted, no head, from=4-2, to=4-4]
	\arrow[dotted, from=4-4, to=3-4]
	\arrow["{\scriptsize{\text{central cycle}}}"{description}, dotted, no head,from=4-9, to=4-11]
	\arrow[dotted, from=4-11, to=3-11]
\end{tikzcd}}

It is clear that the $(l,h)$-central cycle \((va\ldots a')\) in $Q$ it is transformed into a $(l,h)$-central cycle $(va\ldots a')$ in $Q'$. In addition, as in the previous case, the size of the clique $Q'[a,v,b, v_1, \cdots, v_k]$  is bounded by $m+2$. Clearly, the boundary clique $Q[a',v,b,u_1, \cdots, u_r]$ changes to a smaller boundary clique $Q'[a',v,u_1, \cdots, u_r]$.\\

%\item

(3) $v$ is peripheral and $b$ is central.  The mutation $\mu_v$, in this case, acts as illustrated in the following figure:\\

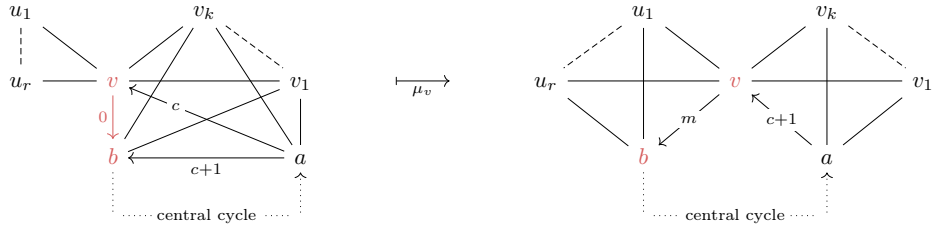
\adjustbox{width=1\linewidth,center}{
\begin{tikzcd}
	{u_1} && {v_k} &&&&& {u_1} && {v_k} \\
	{u_r} & \textcolor{rgb,255:red,214;green,92;blue,92}{v} && {v_1} & {} & {} & {u_r} && \textcolor{rgb,255:red,214;green,92;blue,92}{v} && {v_1} \\
	& \textcolor{rgb,255:red,214;green,92;blue,92}{b} && a &&&& \textcolor{rgb,255:red,214;green,92;blue,92}{b} && a \\
	& {} && {} &&&& {} && {}
	\arrow[dashed, no head, from=1-1, to=2-1]
	\arrow[no head, from=1-1, to=2-2]
	\arrow[no head, from=1-3, to=2-2]
	\arrow[dashed, no head, from=1-3, to=2-4]
	\arrow[no head, from=1-3, to=3-2]
	\arrow[no head, from=1-3, to=3-4]
	\arrow[dashed, no head, from=1-8, to=2-7]
	\arrow[no head, from=1-8, to=2-9]
	\arrow[no head, from=1-8, to=3-8]
	\arrow[no head, from=1-10, to=2-9]
	\arrow[dashed, no head, from=1-10, to=2-11]
	\arrow[no head, from=1-10, to=3-10]
	\arrow[no head, from=2-2, to=2-1]
	\arrow[no head, from=2-2, to=2-4]
	\arrow["0"', color={rgb,255:red,214;green,92;blue,92}, from=2-2, to=3-2]
	\arrow[no head, from=2-4, to=3-2]
	\arrow[no head, from=2-4, to=3-4]
	\arrow["{\mu_v}"', maps to, from=2-5, to=2-6]
	\arrow[no head, from=2-7, to=2-9]
	\arrow[no head, from=2-7, to=3-8]
	\arrow[no head, from=2-9, to=2-11]
	\arrow["m"{description}, from=2-9, to=3-8]
	\arrow[no head, from=2-11, to=3-10]
	\arrow[dotted, no head, from=3-2, to=4-2]
	\arrow["c"{description, pos=0.7}, from=3-4, to=2-2]
	\arrow["{c+1}", from=3-4, to=3-2]
	\arrow[dotted, no head, from=3-8, to=4-8]
	\arrow["{c+1}"{description}, from=3-10, to=2-9]
	\arrow[dotted, from=4-10, to=3-10]
	\arrow["{\text{central cycle}}"{description}, dotted, no head, from=4-2, to=4-4]
	\arrow[dotted, from=4-4, to=3-4]
	\arrow["{\text{central cycle}}"{description}, dotted, no head, from=4-8, to=4-10]
\end{tikzcd}}

The $(l,h)$-central cycle $(b\ldots a)$ is transformed into the $(l+1,h+1)$-central cycle $(vb\ldots a)$ in $Q^\prime$. Clearly, by the initial considerations, it follows that the sizes of the new boundary cliques are at most $m+2$.\\

   % \item 
    
(4) Both $v$ and $b$ are peripheral vertices. The mutation $\mu_v$, in this case, acts as is illustrated in the following figure:

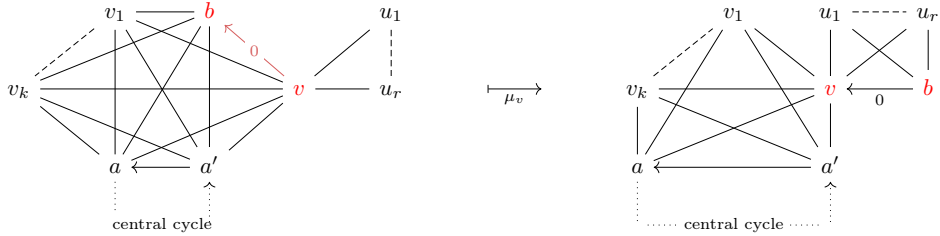
\adjustbox{width=1\linewidth,center}{\begin{tikzcd}
	& {v_1} & {\textcolor{red}{b}} && {u_1} &&&& {v_1} & {u_1} & {u_r} \\
	{v_k} &&& {\textcolor{red}{v}} & {u_r} & {} & {} & {v_k} && {\textcolor{red}{v}} & {\textcolor{red}{b}} \\
	& a & {a'} &&&&& a && {a'} \\
	& {} & {} &&&&& {} && {}
	\arrow[dashed, no head, from=1-2, to=2-1]
	\arrow[no head, from=1-2, to=3-2]
	\arrow[no head, from=1-2, to=3-3]
	\arrow[no head, from=1-3, to=1-2]
	\arrow[no head, from=1-3, to=2-1]
	\arrow[no head, from=1-3, to=3-2]
	\arrow[no head, from=1-3, to=3-3]
	\arrow[dashed, no head, from=1-5, to=2-5]
	\arrow[no head, from=1-9, to=3-8]
	\arrow[dashed, no head, from=1-10, to=1-11]
	\arrow[no head, from=2-1, to=3-2]
	\arrow[no head, from=2-4, to=1-2]
	\arrow["0"{description}, color={rgb,255:red,214;green,92;blue,92}, from=2-4, to=1-3]
	\arrow[no head, from=2-4, to=1-5]
	\arrow[no head, from=2-4, to=2-1]
	\arrow[no head, from=2-4, to=2-5]
	\arrow[no head, from=2-4, to=3-2]
	\arrow[no head, from=2-4, to=3-3]
	\arrow["{\mu_v}"', maps to, from=2-6, to=2-7]
	\arrow[dashed, no head, from=2-8, to=1-9]
	\arrow[no head, from=2-8, to=3-8]
	\arrow[no head, from=2-10, to=1-9]
	\arrow[no head, from=2-10, to=1-10]
	\arrow[no head, from=2-10, to=1-11]
	\arrow[no head, from=2-10, to=2-8]
	\arrow[no head, from=2-10, to=3-8]
	\arrow[no head, from=2-10, to=3-10]
	\arrow[no head, from=2-11, to=1-10]
	\arrow[no head, from=2-11, to=1-11]
	\arrow["0", from=2-11, to=2-10]
	\arrow[dotted, no head, from=3-2, to=4-2]
	\arrow[no head, from=3-3, to=2-1]
	\arrow[from=3-3, to=3-2]
	\arrow[dotted, no head, from=3-8, to=4-8]
	\arrow[no head, from=3-10, to=1-9]
	\arrow[no head, from=3-10, to=2-8]
	\arrow[from=3-10, to=3-8]
	\arrow["{\text{central cycle}}"{description}, dotted, from=4-2, to=4-3]
	\arrow[dotted, from=4-3, to=3-3]
	\arrow["{\text{central cycle}}"{description}, dotted, no head, from=4-8, to=4-10]
	\arrow[dotted, from=4-10, to=3-10]
\end{tikzcd}}

The mutation $\mu_v$, does not change the central cycle \((a\ldots a')\). In addition, as in the cases above,  the size of the clique $Q'[v,u_1, \cdots, u_r,b]$  is bounded by $m+2$ and the boundary clique in $Q$ reduces its size by one in $Q'$.\\

%\end{enumerate}

Now, we assume that $v$ has two outgoing arrows $v \xlongrightarrow{0} b$  and $ v \xlongrightarrow{0} b'$ of color zero. Consequently, the vertices $v$, $b$, and $b'$ cannot be part of a cycle $(vbb')$, according to Proposition \ref{ref:lema colores}.  This rules out the case where $v$ is peripheral and $b$, $b'$ are central, or the case where $v$ and $b$ are peripheral, and $b'$ is central.
Similar to the preceding case, the proof falls naturally into 5 cases.
Note that in every possible case, after mutating at $v$, the vertices $b$ and $b'$ exchange their positions, implying that the sizes of the cliques involved remain unchanged.\\ 

%\begin{enumerate}
    
   % \item  
   
   (1) The three vertices $v$, $b$, and $b'$ are central. The mutation $\mu_v$, in this case, acts as illustrated in the following figure:\\

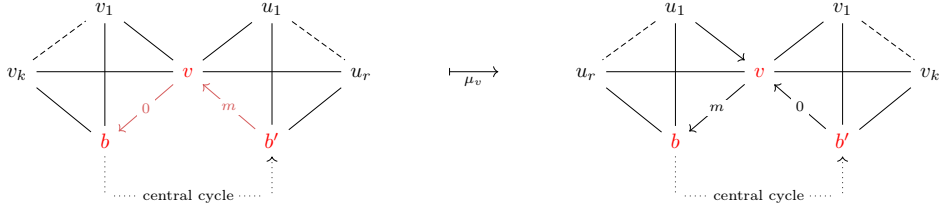
\adjustbox{width=1\linewidth,center}{\begin{tikzcd}
	& {v_1} && {u_1} &&&&& {u_1} && {v_1} \\
	{v_k} && {\textcolor{red}{v}} && {u_r} & {} & {} & {u_r} && {\textcolor{red}{v}} && {v_k} \\
	& {\textcolor{red}{b}} && {\textcolor{red}{b'}} &&&&& {\textcolor{red}{b}} && {\textcolor{red}{b'}} \\
	& {} && {} &&&&& {} && {}
	\arrow[no head, from=1-2, to=3-2]
	\arrow[dashed, no head, from=1-4, to=2-5]
	\arrow[dashed, no head, from=1-9, to=2-8]
	\arrow[from=1-9, to=2-10]
	\arrow[dashed, no head, from=1-11, to=2-12]
	\arrow[dashed, no head, from=2-1, to=1-2]
	\arrow[no head, from=2-1, to=3-2]
	\arrow[no head, from=2-3, to=1-2]
	\arrow[no head, from=2-3, to=1-4]
	\arrow[no head, from=2-3, to=2-1]
	\arrow[no head, from=2-3, to=2-5]
	\arrow["0"{description}, color={rgb,255:red,214;green,92;blue,92}, from=2-3, to=3-2]
	\arrow["{\mu_v}"', maps to, from=2-6, to=2-7]
	\arrow[no head, from=2-8, to=3-9]
	\arrow[no head, from=2-10, to=1-11]
	\arrow[no head, from=2-10, to=2-8]
	\arrow[no head, from=2-10, to=2-12]
	\arrow["m"{description}, from=2-10, to=3-9]
	\arrow[dotted, no head, from=3-2, to=4-2]
	\arrow[no head, from=3-4, to=1-4]
	\arrow["m"{description}, color={rgb,255:red,214;green,92;blue,92}, from=3-4, to=2-3]
	\arrow[no head, from=3-4, to=2-5]
	\arrow[no head, from=3-9, to=1-9]
	\arrow[dotted, no head, from=3-9, to=4-9]
	\arrow[no head, from=3-11, to=1-11]
	\arrow["0"{description}, from=3-11, to=2-10]
	\arrow[no head, from=3-11, to=2-12]
	\arrow["{\scriptsize{\text{central cycle}}}"{description}, dotted, no head, from=4-2, to=4-4]
	\arrow[dotted, from=4-4, to=3-4]
	\arrow["{\scriptsize{\text{central cycle}}}"{description}, dotted, no head, from=4-9, to=4-11]
	\arrow[dotted, from=4-11, to=3-11]
\end{tikzcd}}

\medskip

After applying  $\mu_v$, the $(l,h)$-central cycle \((vb\ldots b')\) in $Q$ remains a $(l,h)$-central cycle  \((vb\ldots b')\) in $Q'$. In addition,  the boundary cliques $Q[v,b,v_k,\ldots,v_1]$ and $Q[v,b',u_1,\ldots, u_r]$ of $Q$  give rise to new boundary cliques $Q'[v,b',v_k,\ldots,v_1]$ and $Q'[v,b,u_1,\ldots, u_r]$, respectively. \\

%\item 

(2) $v$ is central and $b,b'$ are peripheral.   The mutation $\mu_v$ acts as illustrated in the following figure: \\

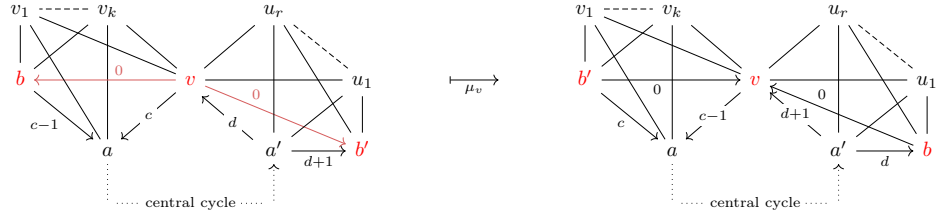
\adjustbox{width=1\linewidth,center}{
\begin{tikzcd}
	{v_1} & {v_k} && {u_r} &&&& {v_1} & {v_k} && {u_r} \\
	{\textcolor{red}{b}} && {\textcolor{red}{v}} && {u_1} & {} & {} & {\textcolor{red}{b'}} && {\textcolor{red}{v}} && {u_1} \\
	& a && {a'} & {\textcolor{red}{b'}} &&&& a && {a'} & {\textcolor{red}{b}} \\
	& {} && {} &&&&& {} && {}
	\arrow[no head, from=1-1, to=2-1]
	\arrow[dashed, no head, from=1-2, to=1-1]
	\arrow[no head, from=1-2, to=2-1]
	\arrow[no head, from=1-2, to=3-2]
	\arrow[dashed, no head, from=1-4, to=2-5]
	\arrow[dashed, no head, from=1-8, to=1-9]
	\arrow[no head, from=1-8, to=2-8]
	\arrow[no head, from=1-9, to=2-8]
	\arrow[no head, from=1-9, to=2-10]
	\arrow[no head, from=1-9, to=3-9]
	\arrow["{c-1}"', from=2-1, to=3-2]
	\arrow[no head, from=2-3, to=1-1]
	\arrow[no head, from=2-3, to=1-2]
	\arrow[no head, from=2-3, to=1-4]
	\arrow["0"'{pos=0.4}, color={rgb,255:red,214;green,92;blue,92}, from=2-3, to=2-1]
	\arrow[no head, from=2-3, to=2-5]
	\arrow["{\scriptsize{c}}"{description}, from=2-3, to=3-2]
	\arrow["0"{pos=0.3}, color={rgb,255:red,214;green,92;blue,92}, from=2-3, to=3-5]
	\arrow[no head, from=2-5, to=3-5]
	\arrow["{\mu_v}"', maps to, from=2-6, to=2-7]
	\arrow["0"'{pos=0.4}, from=2-8, to=2-10]
	\arrow["c"', from=2-8, to=3-9]
	\arrow[no head, from=2-10, to=1-8]
	\arrow[no head, from=2-10, to=1-11]
	\arrow["{\scriptsize{c-1}}"{description}, from=2-10, to=3-9]
	\arrow[dashed, no head, from=2-12, to=1-11]
	\arrow[no head, from=2-12, to=2-10]
	\arrow[no head, from=2-12, to=3-11]
	\arrow[no head, from=3-2, to=1-1]
	\arrow[dotted, no head, from=3-2, to=4-2]
	\arrow[no head, from=3-4, to=1-4]
	\arrow["{\scriptsize{d}}"{description, pos=0.4}, shift left, from=3-4, to=2-3]
	\arrow[no head, from=3-4, to=2-5]
	\arrow["{d+1}"', from=3-4, to=3-5]
	\arrow[no head, from=3-5, to=1-4]
	\arrow[no head, from=3-9, to=1-8]
	\arrow[dotted, no head, from=3-9, to=4-9]
	\arrow[no head, from=3-11, to=1-11]
	\arrow["{\scriptsize{d+1}}"{description}, from=3-11, to=2-10]
	\arrow["d"', from=3-11, to=3-12]
	\arrow[no head, from=3-12, to=1-11]
	\arrow["0"'{pos=0.7}, from=3-12, to=2-10]
	\arrow[no head, from=3-12, to=2-12]
	\arrow["{\scriptsize{\text{central cycle}}}"{description}, dotted, no head, from=4-2, to=4-4]
	\arrow[dotted, from=4-4, to=3-4]
	\arrow["{\scriptsize{\text{central cycle}}}"{description}, dotted, no head, from=4-9, to=4-11]
	\arrow[dotted, from=4-11, to=3-11]
\end{tikzcd}}

\medskip
 Under the mutation $\mu_v$, the $(l,h)$-central cycle \((va\ldots a')\) in $Q$ is preserved as an $(l,h)$-central cycle in $Q'$. \\
% In addition, since there are two central arrows of color zero going out from $v$ to $b$ and $b'$, the cliques $Q[a,v,v_1,\ldots,v_k,b]$ and $Q[a’,v,u_1,\ldots,u_r,b’]$ in Q exchange the vertices $b$ and $b’$, forming the boundary cliques $Q’[a,v,v_1,\ldots,v_k,b’]$ and $Q’[a’,v,u_1,\ldots,u_r,b]$ in $Q’$. 

    %\item 
    
(3) $v$ and $b$ are central vertices  and $b'$ is peripheral.   In this case, the mutation $\mu_v$ acts as illustrated in the following figure: \\

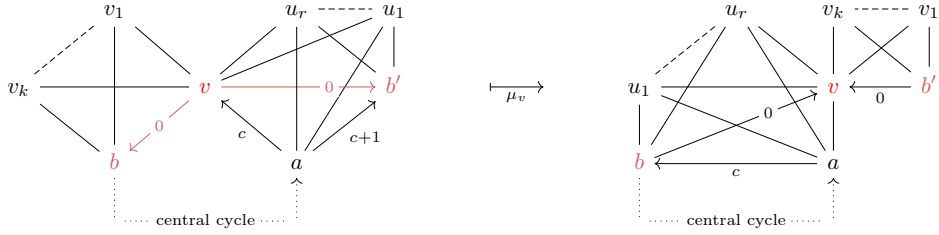
\adjustbox{width=1\linewidth,center}{\begin{tikzcd}
	& {v_1} && {u_r} & {u_1} &&&& {u_r} & {v_k} & {v_1} \\
	{v_k} && {\textcolor{red}{v}} && \textcolor{rgb,255:red,214;green,92;blue,92}{{b'}} & {} & {} & {u_1} && {\textcolor{red}{v}} & \textcolor{rgb,255:red,214;green,92;blue,92}{{b'}} \\
	& \textcolor{rgb,255:red,214;green,92;blue,92}{b} && a &&&& \textcolor{rgb,255:red,214;green,92;blue,92}{b} && a \\
	& {} && {} &&&& {} && {}
	\arrow[dashed, no head, from=1-2, to=2-1]
	\arrow[no head, from=1-2, to=3-2]
	\arrow[dashed, no head, from=1-4, to=1-5]
	\arrow[no head, from=1-5, to=2-5]
	\arrow[no head, from=1-5, to=3-4]
	\arrow[dashed, no head, from=1-9, to=2-8]
	\arrow[no head, from=1-9, to=3-8]
	\arrow[no head, from=1-9, to=3-10]
	\arrow[no head, from=1-10, to=2-11]
	\arrow[dashed, no head, from=1-11, to=1-10]
	\arrow[no head, from=1-11, to=2-10]
	\arrow[no head, from=1-11, to=2-11]
	\arrow[no head, from=2-3, to=1-2]
	\arrow[no head, from=2-3, to=1-4]
	\arrow[no head, from=2-3, to=1-5]
	\arrow[no head, from=2-3, to=2-1]
	\arrow["0"{description, pos=0.7}, color={rgb,255:red,214;green,92;blue,92}, from=2-3, to=2-5]
	\arrow["0"{description}, color={rgb,255:red,214;green,92;blue,92}, from=2-3, to=3-2]
	\arrow[no head, from=2-5, to=1-4]
	\arrow["{\mu_v}"', maps to, from=2-6, to=2-7]
	\arrow[no head, from=2-10, to=1-9]
	\arrow[no head, from=2-10, to=1-10]
	\arrow[no head, from=2-10, to=2-8]
	\arrow["0", from=2-11, to=2-10]
	\arrow[no head, from=3-2, to=2-1]
	\arrow[dotted, no head, from=3-2, to=4-2]
	\arrow[no head, from=3-4, to=1-4]
	\arrow["c", from=3-4, to=2-3]
	\arrow["{c+1}"', from=3-4, to=2-5]
	\arrow[no head, from=3-8, to=2-8]
	\arrow["0"{description, pos=0.7}, from=3-8, to=2-10]
	\arrow[dotted, no head, from=3-8, to=4-8]
	\arrow[no head, from=3-10, to=2-8]
	\arrow[no head, from=3-10, to=2-10]
	\arrow["c", from=3-10, to=3-8]
	\arrow["{\scriptsize{\text{central cycle}}}"{description}, dotted, no head, from=4-2, to=4-4]
	\arrow[dotted, from=4-4, to=3-4]
	\arrow["{\scriptsize{\text{central cycle}}}"{description}, dotted, no head, from=4-8, to=4-10]
	\arrow[dotted, from=4-10, to=3-10]
\end{tikzcd}}

\medskip
After mutating at $v$ the $(l,h)$-central cycle \((b\ldots av)\) becames the $(l-1,h)$-central cycle \((b\ldots a)\). \\ %The vertices $b$ and $b'$ exchange their positions, implying that the sizes of the cliques involved remain unchanged.

 %\item
 
 (4) $v$ is peripheral, $b'$ is central and $b$ is neither a central nor peripheral vertex. The mutation $\mu_v$, in this case, acts as illustrated in the following figure: \\

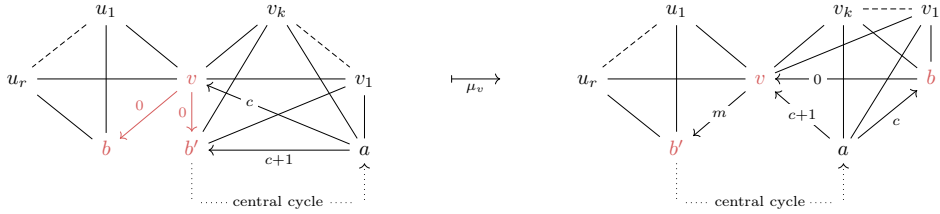
\adjustbox{width=1\linewidth, left}{\begin{tikzcd}
	& {u_1} && {v_k} &&&&& {u_1} && {v_k} & {v_1} \\
	{u_r} && \textcolor{rgb,255:red,214;green,92;blue,92}{v} && {v_1} & {} & {} & {u_r} && \textcolor{rgb,255:red,214;green,92;blue,92}{v} && \textcolor{rgb,255:red,214;green,92;blue,92}{b} \\
	& \textcolor{rgb,255:red,214;green,92;blue,92}{b} & \textcolor{rgb,255:red,214;green,92;blue,92}{{b'}} && a &&&& \textcolor{rgb,255:red,214;green,92;blue,92}{{b'}} && a \\
	&& {} && {} &&&& {} && {}
	\arrow[dashed, no head, from=1-2, to=2-1]
	\arrow[no head, from=1-2, to=2-3]
	\arrow[no head, from=1-2, to=3-2]
	\arrow[no head, from=1-4, to=2-3]
	\arrow[dashed, no head, from=1-4, to=2-5]
	\arrow[no head, from=1-4, to=3-3]
	\arrow[no head, from=1-4, to=3-5]
	\arrow[dashed, no head, from=1-9, to=2-8]
	\arrow[no head, from=1-9, to=2-10]
	\arrow[no head, from=1-9, to=3-9]
	\arrow[dashed, no head, from=1-11, to=1-12]
	\arrow[no head, from=1-11, to=2-10]
	\arrow[no head, from=1-11, to=2-12]
	\arrow[no head, from=1-11, to=3-11]
	\arrow[no head, from=1-12, to=2-12]
	\arrow[no head, from=1-12, to=3-11]
	\arrow[no head, from=2-1, to=2-3]
	\arrow[no head, from=2-1, to=3-2]
	\arrow[no head, from=2-3, to=2-5]
	\arrow["0"', color={rgb,255:red,214;green,92;blue,92}, from=2-3, to=3-2]
	\arrow["0"', color={rgb,255:red,214;green,92;blue,92}, from=2-3, to=3-3]
	\arrow[no head, from=2-5, to=3-3]
	\arrow[no head, from=2-5, to=3-5]
	\arrow["{\mu_v}"', maps to, from=2-6, to=2-7]
	\arrow[no head, from=2-8, to=2-10]
	\arrow[no head, from=2-8, to=3-9]
	\arrow[no head, from=2-10, to=1-12]
	\arrow["m"{description}, from=2-10, to=3-9]
	\arrow["0"{description, pos=0.7}, from=2-12, to=2-10]
	\arrow[dotted, no head, from=3-3, to=4-3]
	\arrow["c"{description, pos=0.7}, from=3-5, to=2-3]
	\arrow["{c+1}", from=3-5, to=3-3]
	\arrow[dotted, no head, from=3-9, to=4-9]
	\arrow["{c+1}"{description}, from=3-11, to=2-10]
	\arrow["c"', from=3-11, to=2-12]
	\arrow[dotted, from=4-11, to=3-11]
	\arrow["{{\scriptsize \text{central cycle}}}"{description}, dotted, no head, from=4-3, to=4-5]
	\arrow[dotted, from=4-5, to=3-5]
	\arrow["{\text{central cycle}}"{description}, dotted, no head, from=4-9, to=4-11]
\end{tikzcd}}

\medskip

After mutating at $v$ the $(l,h)$-central cycle \((b'\ldots a)\) becames a $(l+1,h+1)$-central cycle \((b'\ldots av)\).\\ 
%The vertices $b$ and $b'$ exchange their positions, implying that the sizes of the cliques involved remain unchanged.

    % \item $v$ is peripheral and  $b,b'$ are  central. este caso no puede pasar. This case is not possible since the triangle $(v b b’)$ belongs to the class $\mathcal{Q}_3^m$, and therefore it is impossible for two zero-colored arrows to originate from the same vertex.

    % \item $v$ and $b$ are peripheral and  $b'$ is  central. este caso te falta. tampoco puede pasar

    %\item
    
    (5) $v$ and $b$ are peripheral and $b'$ is neither a central nor peripheral vertex.  The mutation $\mu_v$ acts as shown in the following figure: \\

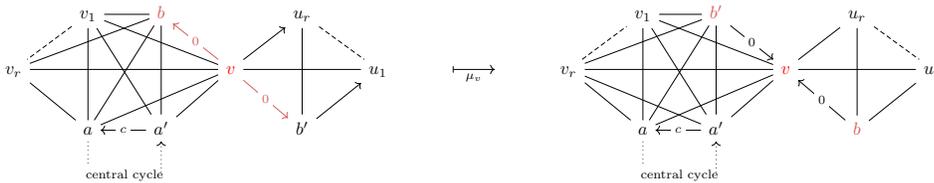
\adjustbox{width=1\linewidth,center}{\begin{tikzcd}
	& {v_1} & \textcolor{rgb,255:red,214;green,92;blue,92}{b} && {u_r} &&&&& {v_1} & \textcolor{rgb,255:red,214;green,92;blue,92}{{b'}} && {u_r} \\
	{v_r} &&& {\textcolor{red}{v}} && {u_1} & {} & {} & {v_r} &&& {\textcolor{red}{v}} && {u_1} \\
	& a & {a'} && {b'} &&&&& a & {a'} && \textcolor{rgb,255:red,214;green,92;blue,92}{b} \\
	& {} & {} &&&&&&& {} & {}
	\arrow[dashed, no head, from=1-2, to=2-1]
	\arrow[no head, from=1-3, to=1-2]
	\arrow[no head, from=1-3, to=2-1]
	\arrow[dashed, no head, from=1-5, to=2-6]
	\arrow[no head, from=1-10, to=1-11]
	\arrow[no head, from=1-10, to=2-12]
	\arrow[no head, from=1-10, to=3-10]
	\arrow[no head, from=1-10, to=3-11]
	\arrow["0"{description}, from=1-11, to=2-12]
	\arrow[no head, from=1-11, to=3-10]
	\arrow[no head, from=1-11, to=3-11]
	\arrow[dashed, no head, from=1-13, to=2-14]
	\arrow[no head, from=1-13, to=3-13]
	\arrow[no head, from=2-1, to=2-4]
	\arrow[no head, from=2-1, to=3-2]
	\arrow[no head, from=2-4, to=1-2]
	\arrow["0"{description}, color={rgb,255:red,214;green,92;blue,92}, from=2-4, to=1-3]
	\arrow[from=2-4, to=1-5]
	\arrow[no head, from=2-4, to=2-6]
	\arrow[no head, from=2-4, to=3-2]
	\arrow[no head, from=2-4, to=3-3]
	\arrow["0"{description}, color={rgb,255:red,214;green,92;blue,92}, from=2-4, to=3-5]
	\arrow["{\mu_v}"', maps to, from=2-7, to=2-8]
	\arrow[dashed, no head, from=2-9, to=1-10]
	\arrow[no head, from=2-9, to=1-11]
	\arrow[no head, from=2-9, to=2-12]
	\arrow[no head, from=2-9, to=3-10]
	\arrow[no head, from=2-9, to=3-11]
	\arrow[no head, from=2-12, to=1-13]
	\arrow[no head, from=2-12, to=2-14]
	\arrow[no head, from=3-2, to=1-2]
	\arrow[no head, from=3-2, to=1-3]
	\arrow[dotted, no head, from=3-2, to=4-2]
	\arrow[no head, from=3-3, to=1-2]
	\arrow[no head, from=3-3, to=1-3]
	\arrow["c"{description}, from=3-3, to=3-2]
	\arrow[no head, from=3-5, to=1-5]
	\arrow[from=3-5, to=2-6]
	\arrow[no head, from=3-10, to=2-12]
	\arrow[dotted, no head, from=3-10, to=4-10]
	\arrow[no head, from=3-11, to=2-12]
	\arrow["c"{description}, from=3-11, to=3-10]
	\arrow["0"{description}, from=3-13, to=2-12]
	\arrow[no head, from=3-13, to=2-14]
	\arrow["{{\scriptsize \text{central cycle}}}"{description}, dotted, no head, from=4-2, to=4-3]
	\arrow[dotted, from=4-3, to=3-3]
	\arrow["{{\scriptsize \text{central cycle}}}"{description}, dotted, no head, from=4-10, to=4-11]
	\arrow[dotted, from=4-11, to=3-11]
\end{tikzcd}}

\medskip
The mutation $\mu_v$, does not change the $(l,h)$-central cycle \((a\ldots a')\). The vertices $b$ and $b'$ exchange their positions, implying that the sizes of the cliques involved remain unchanged, as in every previous case.

%\end{enumerate}
\end{proof}

We conclude this section by proving that the colored mutation is invertible restricted to the class $\mathcal{Q}^m_{p,q}$. 

\begin{lem}
The colored quiver mutation restricted to the class $\mathcal{Q}^m_{p,q}$ is invertible. Moreover, if $Q \in \mathcal{Q}^m_{p,q}$ and $v$ is a vertex of $Q$, then \(\mu_v^{m+1}(Q) = Q. \)
\end{lem}

\begin{proof}

Let $Q$ be a quiver belonging to the class $\mathcal{Q}^m_{p,q}$ and let $v$ be a vertex in  $Q$. If $v$ has a unique neighbor in $Q$ the result is trivial, even when $l=2$ and $v$ is central. We thus assume that $v$ has at least two neighbors $u_1$ and $u_2$ and define the subquiver $Q' = Q[v,u_1,u_2]$. As established in \cite[Lemma~4.8]{gubitosi2024coloured}, to prove that $\mu_v^{m+1}(Q) = Q$, it suffices to show that $\mu_v^{m+1}(Q') = Q'$ for every such subquiver $Q'$.  
It is clear that for $l\geq 4$, all vertices locally exhibit the same structural properties as the vertices of the class $\mathcal{Q}^m_{n}$. Specifically, if $u_1$ and $u_2$ are two neighbors of $v$, their adjacency is satisfied if and only if they belong to the same $m$-admissible triangle. Therefore, Proposition \autoref{ref:lema colores} establishes that, under this condition, 
$\mu^{m+1}_v(Q') = Q'$.

In consequence, to definitively prove that $\mu^{m+1}_v(Q') = Q'$ holds for any vertex $v$, we must analyze the sole remaining case where two vertices $u_1$ and $u_2$, are adjacent yet violate the conditions for the class $\mathcal{Q}^m_{n}$. This exceptional situation occurs exclusively when $l=3$, and the vertices $u_1$, $u_2$, and $v$ induce the cycle $(vu_1u_2)=\Delta^m_3$, which in not an $m$-admissible triangle. Let $c_1$ be the color of the arrow $v\longrightarrow u_1$, $c_2$ be the color of the arrow  $v\longrightarrow u_2$ and $c$ be the color of the arrow  $u_1\longrightarrow u_2$. The relation between these colors is given by \(c_1 +c+ m- c_2 = hm,\) with $0 < h < 3$.  Unlike the $\mathbb{A}_n$-case, here the colors $c_1$ and $c_2$ may coincide. We proceed by analyzing two cases. 

First, assume that $c_1=c_2$. In this scenario, $c=m$ and, we have the following situation: 

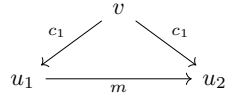
\adjustbox{scale=.9,center}
{\begin{tikzcd}[column sep=small]
	&& v \\
	{u_1} &&&& {u_2}
	\arrow["c_1"', from=1-3, to=2-1]
	\arrow["c_1", from=1-3, to=2-5]
	\arrow["{m}"', from=2-1, to=2-5]
\end{tikzcd}}

\medskip
Clearly, the adjacency between $u_1$ and $u_2$  is preserved through every mutation at $v$. While the colors of the arrows change during the process, they revert to their original values after $m+1$ mutations.

Finally, consider the case where $c_1\neq c_2$.  We can assume without loss of generality that $c_1< c_2$. The relationship between the colors is defined by the  equations  \(c+c_1 +m-c_2 = m \) or $c+c_1 +m-c_2 = 2m $. In the second possibility \((m - c) + (m - c_1) + c_2 =2m - m = m\). Thus, we restrict our analysis to the case where
 \(c+c_1 +m-c_2 = m \) yielding $c=c_2-c_1$. This situation is illustrated in the following figure:

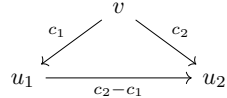
\adjustbox{scale=.9,center}
{\begin{tikzcd}[column sep=small]
	&& v \\
	{u_1} &&&& {u_2}
	\arrow["c_1"', from=1-3, to=2-1]
	\arrow["c_2", from=1-3, to=2-5]
	\arrow["{c_2-c_1}"', from=2-1, to=2-5]
\end{tikzcd}}

\medskip

The application of $\mu_v^{c_1}$ results in the following:

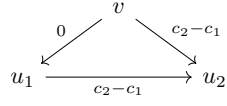
\adjustbox{scale=.9,center}
{\begin{tikzcd}[column sep=small]
	&& v \\
	{u_1} &&&& {u_2}
	\arrow["0"', from=1-3, to=2-1]
	\arrow["c_2-c_1", from=1-3, to=2-5]
	\arrow["{c_2-c_1}"', from=2-1, to=2-5]
\end{tikzcd}}

\medskip
Applying $\mu_v^{c_1 + 1}$ to $Q'$ results in the addition of an arrow $u_1\longrightarrow u_2$ with color $c_2-c_1$, while the colors of the remaining arrows are changed as follows:

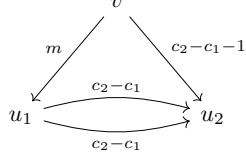
\adjustbox{scale=.9,center}
{\begin{tikzcd}[column sep=small]
	&& v \\
	\\
	{u_1} &&&& {u_2}
	\arrow["{m}"', from=1-3, to=3-1]
	\arrow["{c_2-c_1-1}", from=1-3, to=3-5]
	\arrow["c_2-c_1", bend left=15, from=3-1, to=3-5]
	\arrow["c_2-c_1"', bend right=15, from=3-1, to=3-5]
\end{tikzcd}}

\medskip
The shape of $Q'$ remains invariant under the sequence of mutations  $\mu_v^{c_1 + j}$ for  all $1 \leq j \leq c_2 - c_1$. However, when $j=c_2-c_1+1$ the sequence of mutations  $\mu_v^{c_1 + j}$ introduces an arrow $u_1\longrightarrow u_2$ of color $c_2-c_1-1$. This new arrow annihilates one of the arrows of color $c_2-c_1$ present in the subquiver. Consequently, the subquiver 
$\mu_v^{c_1 + (c_2-c_1+1)}(Q')$ takes the following form:\\

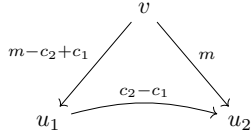
\adjustbox{scale=.9,center}
{\begin{tikzcd}[column sep=small]
	&& v \\
	\\
	{u_1} &&&& {u_2}
	\arrow["{m-c_2+c_1}"', from=1-3, to=3-1]
	\arrow["{m}", from=1-3, to=3-5]
	\arrow["c_2-c_1", bend left=15, from=3-1, to=3-5]
	%\arrow["c_2-c_1"', bend right=15, from=3-1, to=3-5]
\end{tikzcd}}

\medskip
The quiver's shape remains invariant under the subsequent mutations, and the colors revert to their original state after $m-c_2$ more mutations showing that 
$\mu_v^{m+1}(Q')=\mu_v^{(c_1 + c_2-c_1+1)+(m-c_2)}(Q')=Q'$. 
\end{proof}

\section{Characterization of the $\widetilde{\mathbb{A}}_{p,q}$ mutation class}

We begin this section with a series of lemmas that will allow us to prove the main theorem of this paper.

\begin{lem}
\label{ref:lema3}
 Let $Q$ be a quiver in the class $\mathcal{Q}^m_{p,q}$. Let $w_1, \cdots , w_r$ be peripheral vertices in $Q$ and assume that the vertices $a$ and $a'$ are central vertices in such a way that the clique $Q[w_1,w_2, \cdots,  w_r ,a,a']$ is a boundary subquiver of $Q$ of size $r+2$. We illustrate the situation in the following quiver:

% https://q.uiver.app/#q=WzAsMTQsWzMsMiwiYV8xIl0sWzUsMiwiYV8yIl0sWzMsMSwid18xIl0sWzQsMCwid18yIl0sWzIsMSwia18xXjEiXSxbMCwxLCJrX3tuXzF9XjEiXSxbMywwLCJrXzFeMiJdLFsxLDAsImteMl97bl8yfSJdLFs1LDEsIndfciJdLFs2LDEsImtfMV5uIl0sWzcsMSwiXFxjZG90cyJdLFs4LDEsImtfe25fcn1eciJdLFsxLDEsIlxcY2RvdHMiXSxbMiwwLCJcXGNkb3RzIl0sWzAsMSwiXFxhbHBoYV8xIiwwLHsiY29sb3VyIjpbMCw2MCw2MF19XSxbMiwwXSxbMiwzLCJjIl0sWzQsMiwiMCJdLFs2LDMsIjAiXSxbMywwXSxbMywxXSxbMiwxXSxbMyw4LCIiLDIseyJzdHlsZSI6eyJib2R5Ijp7Im5hbWUiOiJkYXNoZWQifX19XSxbOCwxXSxbOCwwXSxbMiw4XSxbOSw4LCIwIiwyXSxbMTAsOSwiMCIsMl0sWzExLDEwLCIwIiwyXSxbNSwxMiwiMCJdLFsxMiw0LCIwIl0sWzcsMTMsIjAiXSxbMTMsNiwiMCJdLFswLDEsIlxcdGV4dHtjZW50ZXIgY3ljbGV9IiwyLHsib2Zmc2V0IjotMywiY3VydmUiOjUsImNvbG91ciI6WzAsNjAsNjBdfSxbMCw2MCw2MCwxXV1d
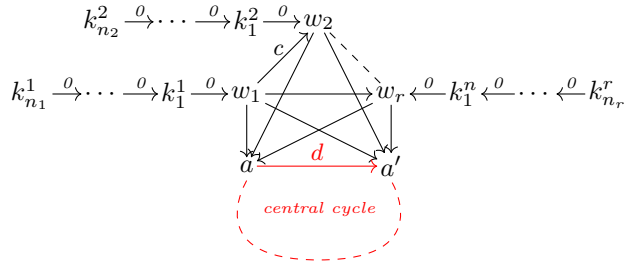
\[\begin{tikzpicture}[scale=0.8, x=1.2cm,y=1.2cm,every node/.style={inner sep=1pt}]
  %%% NODOS
  % Fila 1
  \node (r1c2) at (2,0) {$k^2_{n_2}$};
  \node (r1c3) at (3,0) {$\cdots$};
  \node (r1c4) at (4,0) {$k_1^2$};
  \node (r1c5) at (5,0) {$w_2$};

  % Fila 2
  \node (r2c1) at (1,-1) {$k_{n_1}^1$};
  \node (r2c2) at (2,-1) {$\cdots$};
  \node (r2c3) at (3,-1) {$k_1^1$};
  \node (r2c4) at (4,-1) {$w_1$};
  \node (r2c6) at (6,-1) {$w_r$};
  \node (r2c7) at (7,-1) {$k_1^n$};
  \node (r2c8) at (8,-1) {$\cdots$};
  \node (r2c9) at (9,-1) {$k_{n_r}^r$};

  % Fila 3
  \node (r3c4) at (4,-2) {$a$};
  \node (r3c6) at (6,-2) {$a^{\prime}$};

  %%% FLECHAS
  % Fila 1
  \draw[->] (r1c2) -- node[above, yshift=1pt]{\tiny  0} (r1c3);
  \draw[->] (r1c3) -- node[above, yshift=1pt]{\tiny  0} (r1c4);
  \draw[->] (r1c4) -- node[above, yshift=1pt]{\tiny  0} (r1c5);

  % Conexiones desde w2
  \draw[dashed,-] (r1c5) -- (r2c6);
  \draw[->] (r1c5) -- (r3c4);
  \draw[->] (r1c5) -- (r3c6);

  % Fila 2 izquierda
  \draw[->] (r2c1) -- node[above, yshift=1pt]{\tiny  \text{0}} (r2c2);
  \draw[->] (r2c2) -- node[above, yshift=1pt]{\tiny  0} (r2c3);
  \draw[->] (r2c3) -- node[above, yshift=1pt]{\tiny  0} (r2c4);
  \draw[->] (r2c4) -- node[above, xshift=-2pt]{\small c} (r1c5);

  % w1 conexiones
  \draw[->] (r2c4) -- (r2c6);
  \draw[->] (r2c4) -- (r3c4);
  \draw[->] (r2c4) -- (r3c6);

  % Desde wr
  \draw[->] (r2c6) -- (r3c4);
  \draw[->] (r2c6) -- (r3c6);

  % Cadena derecha
  \draw[->] (r2c7) -- node[above,swap,yshift=1pt]{\tiny  0} (r2c6);
  \draw[->] (r2c8) -- node[above,swap,yshift=1pt]{\tiny  0} (r2c7);
  \draw[->] (r2c9) -- node[above,swap,yshift=1pt]{\tiny  0} (r2c8);

  % Flecha roja α
  \draw[red,->] (r3c4) -- node[above,red, yshift=1pt]{\small $d$} (r3c6);

  % Curva roja "central cycle"
  \draw[red, smooth, dashed, tension=3] plot coordinates{(4,-2.2) (5,-3.3) (6,-2.2)};
  \node[red] at (5,-2.6) {\tiny central cycle};
\end{tikzpicture}\]

 Assume that $\overline{w_1w_2}=c$ is the minimum among the colors of the arrows starting at $w_1$. Then, \( Q \) is mutation-equivalent (without performing mutations on \( a \) or \( a^{\prime} \)) to the quiver $Q'$:

% https://q.uiver.app/#q=WzAsMTQsWzcsMiwiYV8xIl0sWzksMiwiYV8yIl0sWzcsMCwid18xIl0sWzksMCwid19yIl0sWzMsMCwid18yIl0sWzYsMCwia18xXjEiXSxbNSwwLCJcXGNkb3RzIl0sWzQsMCwia197bl8xfV4xIl0sWzIsMCwia18xXjIiXSxbMSwwLCJcXGNkb3RzIl0sWzAsMCwia14yX3tuXzJ9Il0sWzEwLDAsImtfMV5yIl0sWzExLDAsIlxcY2RvdHMiXSxbMTIsMCwia197bl9yfV5yIl0sWzAsMSwiIiwwLHsiY29sb3VyIjpbMCw2MCw2MF19XSxbMiwwXSxbMSwzXSxbNSwyLCIwIl0sWzYsNV0sWzcsNiwiMCJdLFs4LDQsIjAiXSxbOSw4LCIwIl0sWzEwLDksIjAiXSxbMTEsMywiMCIsMl0sWzEyLDExLCIwIiwyXSxbMiwzLCIiLDAseyJzdHlsZSI6eyJib2R5Ijp7Im5hbWUiOiJkYXNoZWQifX19XSxbMiwxXSxbMywwXSxbNCw3LCIwIl0sWzEzLDEyLCIwIiwyXSxbMCwxLCJjZW50ZXIiLDEseyJvZmZzZXQiOi0zLCJjdXJ2ZSI6NSwiY29sb3VyIjpbMCw2MCw2MF19LFswLDYwLDYwLDFdXV0=
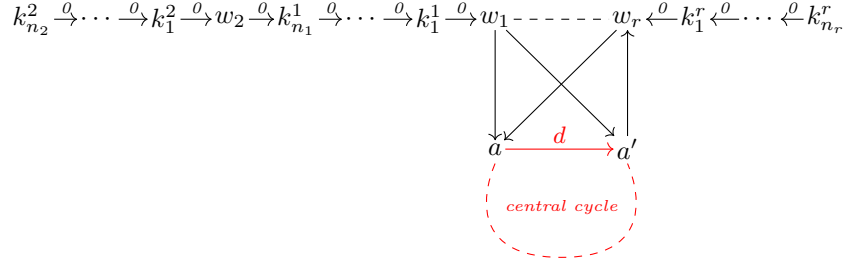
\[\begin{tikzpicture}[scale=0.8, x=1.1cm,y=1.2cm, every node/.style={inner sep=1pt}]
  %%% FILA 1
  \node (r1c1) at (1,0) {$k^2_{n_2}$};
  \node (r1c2) at (2,0) {$\cdots$};
  \node (r1c3) at (3,0) {$k_1^2$};
  \node (r1c4) at (4,0) {$w_2$};
  \node (r1c5) at (5,0) {$k_{n_1}^1$};
  \node (r1c6) at (6,0) {$\cdots$};
  \node (r1c7) at (7,0) {$k_1^1$};
  \node (r1c8) at (8,0) {$w_1$};
  \node (r1c10) at (10,0) {$w_r$};
  \node (r1c11) at (11,0) {$k_1^r$};
  \node (r1c12) at (12,0) {$\cdots$};
  \node (r1c13) at (13,0) {$k_{n_r}^r$};

  %%% FILA 3 (a y a')
  \node (r3c8) at (8,-1.8) {$a$};
  \node (r3c10) at (10,-1.8) {$a^\prime$};

  %%% FLECHAS
  \draw[->] (r1c1) -- node[above,yshift=1pt]{\tiny  0} (r1c2);
  \draw[->] (r1c2) -- node[above,yshift=1pt]{\tiny  0} (r1c3);
  \draw[->] (r1c3) -- node[above,yshift=1pt]{\tiny  0} (r1c4);
  \draw[->] (r1c4) -- node[above,yshift=1pt]{\tiny  0} (r1c5);
  \draw[->] (r1c5) -- node[above,yshift=1pt]{\tiny  0} (r1c6);
  \draw[->] (r1c6) -- node[above,yshift=1pt]{\tiny  0} (r1c7);
  \draw[->] (r1c7) -- node[above,yshift=1pt]{\tiny  0} (r1c8);

  % dashed w1 → wr
  \draw[dashed,-] (r1c8) -- (r1c10);

  % conexiones verticales
  \draw[->] (r1c8) -- (r3c8);
  \draw[->] (r1c8) -- (r3c10);
  \draw[->] (r1c10) -- (r3c8);

  % derecha (r)
  \draw[->] (r1c11) -- node[above,yshift=1pt,swap]{\tiny  0} (r1c10);
  \draw[->] (r1c12) -- node[above,yshift=1pt,swap]{\tiny  0} (r1c11);
  \draw[->] (r1c13) -- node[above,yshift=1pt,swap]{\tiny  0} (r1c12);

  % flecha roja α
  \draw[red,->] (r3c8) -- node[above,red, yshift=1pt]{\small $d$} (r3c10);

  % curva roja "central cycle"
  \draw[red, smooth, dashed, tension=3] plot coordinates{(8,-2) (9,-3.3) (10,-2)};
  \node[red] at (9,-2.6) {\tiny central cycle};

  % conexión vertical extra
  \draw[->] (r3c10) -- (r1c10);
\end{tikzpicture}\]
where  the clique $Q[w_1, w_3, \cdots, w_r,a,a']$ is a boundary subquiver of $Q'$ of size $r+1$ with $a$ and $a^{\prime}$  central vertices.
\end{lem}

\begin{proof}
Since \( c \) is the minimal color of the arrows starting at \( w_1 \), we may assume \( c = 0 \) (if not, we apply \( \mu_{w_1}^c \) to \( Q \)). Performing the mutation \( \mu_{w_1} \), the vertex \( w_2 \) is disconnected from the clique, yielding the quiver

% https://q.uiver.app/#q=WzAsMTQsWzQsMywidl8xIl0sWzYsMywidl8yIl0sWzQsMSwid18xIl0sWzYsMSwid18zIl0sWzMsMSwid18yIl0sWzMsMCwia18xXjEiXSxbMiwwLCJcXGNkb3QiXSxbMSwwLCJrX3ttXzF9XjEiXSxbMiwxLCJrXzFeMiJdLFsxLDEsIlxcY2RvdCJdLFswLDEsImteMl97bV8yfSJdLFs3LDEsImtfMV4zIl0sWzgsMSwiXFxjZG90Il0sWzksMSwia197bV8zfV4zIl0sWzAsMSwiIiwwLHsiY29sb3VyIjpbMCw2MCw2MF19XSxbMiwwXSxbMSwzXSxbNSwyLCIxIl0sWzYsNSwiIiwwLHsic3R5bGUiOnsiYm9keSI6eyJuYW1lIjoiZGFzaGVkIn19fV0sWzcsNiwiMCJdLFs4LDQsIjAiXSxbOSw4LCIiLDAseyJzdHlsZSI6eyJib2R5Ijp7Im5hbWUiOiJkYXNoZWQifX19XSxbMTAsOSwiMCJdLFsxMSwzLCIwIiwyXSxbMTIsMTEsIiIsMCx7InN0eWxlIjp7ImJvZHkiOnsibmFtZSI6ImRhc2hlZCJ9fX1dLFsxMywxMiwiMCIsMl0sWzIsMywiY157XFxwcmltZX0iXSxbMiwxXSxbMywwXSxbNCwyLCIwIiwyXSxbNSw0LCIwIiwyXV0=
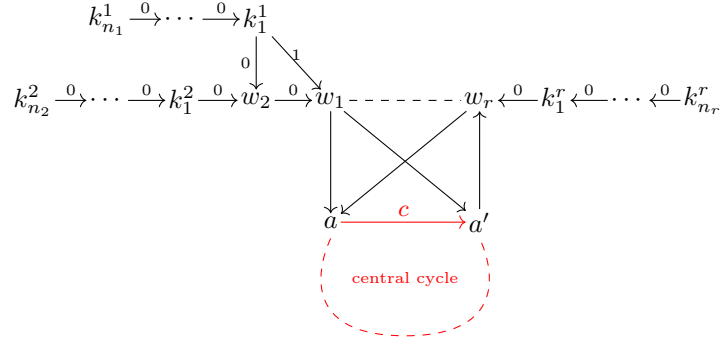
\[\begin{tikzpicture}[scale=0.9, x=1.1cm,y=1.2cm, every node/.style={inner sep=1pt}]
  %%% FILA 1
  \node (r1c2) at (2,0) {$k_{n_1}^1$};
  \node (r1c3) at (3,0) {$\cdots$};
  \node (r1c4) at (4,0) {$k_1^1$};

  %%% FILA 2
  \node (r2c1) at (1,-1) {$k^2_{n_2}$};
  \node (r2c2) at (2,-1) {$\cdots$};
  \node (r2c3) at (3,-1) {$k_1^2$};
  \node (r2c4) at (4,-1) {$w_2$};
  \node (r2c5) at (5,-1) {$w_1$};
  \node (r2c7) at (7,-1) {$w_r$};
  \node (r2c8) at (8,-1) {$k_1^r$};
  \node (r2c9) at (9,-1) {$\cdots$};
  \node (r2c10) at (10,-1) {$k_{n_r}^r$};

  %%% FILA 4
  \node (r4c5) at (5,-2.5) {$a$};
  \node (r4c7) at (7,-2.5) {$a^{\prime}$};

  %%% FLECHAS
  % fila 1
  \draw[->] (r1c2) -- node[above,yshift=1pt]{\tiny  0} (r1c3);
  \draw[->] (r1c3) -- node[above,yshift=1pt]{\tiny  0} (r1c4);
  \draw[->] (r1c4) -- node[left,xshift=-1pt]{\tiny  0} (r2c4);
  \draw[->] (r1c4) -- node[above]{\tiny  1} (r2c5);

  % fila 2 izquierda
  \draw[->] (r2c1) -- node[above,yshift=1pt]{\tiny  0} (r2c2);
  \draw[->] (r2c2) -- node[above,yshift=1pt]{\tiny  0} (r2c3);
  \draw[->] (r2c3) -- node[above,yshift=1pt]{\tiny  0} (r2c4);
  \draw[->] (r2c4) -- node[above,yshift=1pt,swap]{\tiny  0} (r2c5);

  % conexiones desde w1
  \draw[-, dashed] (r2c5) -- (r2c7);
  \draw[->] (r2c5) -- (r4c5);
  \draw[->] (r2c5) -- (r4c7);

  % conexiones desde wr
  \draw[->] (r2c7) -- (r4c5);

  % cadena derecha
  \draw[->] (r2c8) -- node[above,yshift=1pt,swap]{\tiny  0} (r2c7);
  \draw[->] (r2c9) -- node[above,yshift=1pt,swap]{\tiny  0} (r2c8);
  \draw[->] (r2c10) -- node[above,yshift=1pt,swap]{\tiny  0} (r2c9);

  % flecha roja α
  \draw[red,->] (r4c5) -- node[above,red, yshift=1pt]{\small $c$} (r4c7);

  % central cycle como curva suave (sin flecha), según tu indicación
  \draw[red, smooth, dashed, tension=3] plot coordinates{(5,-2.7) (6,-3.9) (7,-2.7)};
  \node[red] at (6,-3.2) {\tiny central cycle};

  % conexión vertical final
  \draw[->] (r4c7) -- (r2c7);
\end{tikzpicture}\]
By performing the mutations
\(
\mu_{k_{n_1}^1} \circ \cdots \circ \mu_{k_{1}^1}
\)
on this new quiver, we obtain the desired quiver.

\end{proof}

\begin{obs}
    \label{ref:obs1}
    Without loss of generality, fixing $w_1$ as in the initial quiver of Lemma \ref{ref:lema3}, this lemma allows us to conclude inductively that one can assume the minimum color $c$ of the arrows starting at $w_1$ corresponds to an arrow connecting to $a$ or $a^\prime$. Otherwise, the lemma allows us to reduce the size of the clique until we are left with the highlighted triangle $(w_1 aa^\prime),$ where the mutations performed are never applied to the vertices $a$ or $a^\prime.$

\end{obs}

\begin{lem}
\label{ref:lema ingreso1}
Let $Q$ be a quiver in the class $\mathcal{Q}^m_{p,q}$. Let $w_1, \cdots , w_r$ be peripheral vertices in $Q$ and assume that the vertices $a$ and $a'$ are central vertices in such a way that the clique $Q[w_1,w_2, \cdots,  w_r ,a,a']$ is a boundary subquiver of $Q$ of size $r+2$. We illustrate the situation in the following quiver:
% https://q.uiver.app/#q=WzAsMTQsWzMsMiwidl8xIl0sWzUsMiwidl8yIl0sWzMsMSwid18xIl0sWzQsMCwid18yIl0sWzIsMSwia18xXjEiXSxbMCwxLCJrX3ttXzF9XjEiXSxbMywwLCJrXzFeMiJdLFsxLDAsImteMl97bV8yfSJdLFs1LDEsIndfbiJdLFs2LDEsImtfMV5uIl0sWzcsMSwiXFxjZG90cyJdLFs4LDEsImtfe21fbn1ebiJdLFsxLDEsIlxcY2RvdHMiXSxbMiwwLCJcXGNkb3RzIl0sWzAsMSwiXFxhbHBoYV8xIiwyLHsiY29sb3VyIjpbMCw2MCw2MF19XSxbMiwwXSxbMiwzXSxbNCwyLCIwIl0sWzYsMywiMCJdLFszLDBdLFszLDFdLFsyLDFdLFszLDgsIiIsMix7InN0eWxlIjp7ImJvZHkiOnsibmFtZSI6ImRhc2hlZCJ9fX1dLFs4LDFdLFs4LDBdLFsyLDhdLFs5LDgsIjAiLDJdLFsxMCw5LCIwIiwyXSxbMTEsMTAsIjAiLDJdLFs1LDEyLCIwIl0sWzEyLDQsIjAiXSxbNywxMywiMCJdLFsxMyw2LCIwIl1d
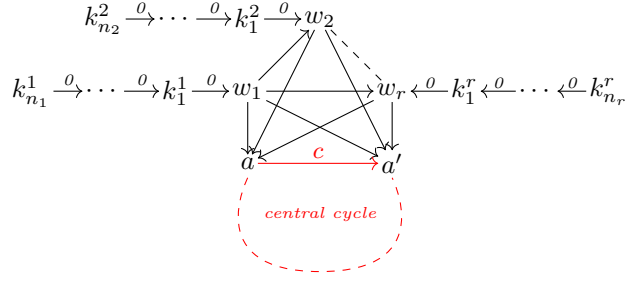
\[\begin{tikzpicture}[scale=.8, x=1.2cm,y=1.2cm, every node/.style={inner sep=1pt}]
  %%% NODOS
  % Fila 1
  \node (r1c2) at (2,0) {$k^2_{n_2}$};
  \node (r1c3) at (3,0) {$\cdots$};
  \node (r1c4) at (4,0) {$k_1^2$};
  \node (r1c5) at (5,0) {$w_2$};

  % Fila 2
  \node (r2c1) at (1,-1) {$k_{n_1}^1$};
  \node (r2c2) at (2,-1) {$\cdots$};
  \node (r2c3) at (3,-1) {$k_1^1$};
  \node (r2c4) at (4,-1) {$w_1$};
  \node (r2c6) at (6,-1) {$w_r$};
  \node (r2c7) at (7,-1) {$k_1^r$};
  \node (r2c8) at (8,-1) {$\cdots$};
  \node (r2c9) at (9,-1) {$k_{n_r}^r$};

  % Fila 3
  \node (r3c4) at (4,-2) {$a$};
  \node (r3c6) at (6,-2) {$a^{\prime}$};

  %%% FLECHAS DEL DIAGRAMA
  % Fila 1
  \draw[->] (r1c2) -- node[above, yshift=1pt]{\tiny  0} (r1c3);
  \draw[->] (r1c3) -- node[above, yshift=1pt]{\tiny  0} (r1c4);
  \draw[->] (r1c4) -- node[above, yshift=1pt]{\tiny  0} (r1c5);

  % Conexiones desde w2
  \draw[dashed,-] (r1c5) -- (r2c6);
  \draw[->] (r1c5) -- (r3c4);
  \draw[->] (r1c5) -- (r3c6);

  % Fila 2: cadena izquierda
  \draw[->] (r2c1) -- node[above, yshift=1pt]{\tiny  0} (r2c2);
  \draw[->] (r2c2) -- node[above, yshift=1pt]{\tiny  0} (r2c3);
  \draw[->] (r2c3) -- node[above, yshift=1pt]{\tiny  0} (r2c4);

  % Desde w1
  \draw[->] (r2c4) -- (r1c5);
  \draw[->] (r2c4) -- (r2c6);
  \draw[->] (r2c4) -- (r3c4);
  \draw[->] (r2c4) -- (r3c6);

  % Desde w_r
  \draw[->] (r2c6) -- (r3c4);
  \draw[->] (r2c6) -- (r3c6);

  % Cadena derecha
  \draw[->] (r2c7) -- node[above,swap]{\tiny  0} (r2c6);
  \draw[->] (r2c8) -- node[above,swap]{\tiny  0} (r2c7);
  \draw[->] (r2c9) -- node[above,swap]{\tiny  0} (r2c8);

  % Flecha roja alpha entre a y a'
  \draw[red,->] (r3c4) -- node[above,red, yshift=1pt]{\small $c$} (r3c6);

  %%% central cycle (entre a y a') COMO CURVA SUAVE (sin flecha)
  % Coordenadas tomadas exactamente entre (4,-2) y (6,-2) con un punto medio más bajo:
  \draw[red, smooth, dashed, tension=3] plot coordinates{(4,-2.2) (5,-3.5) (6,-2.2)};
  \node[red] at (5,-2.7) {\tiny central cycle};

\end{tikzpicture}\]
If $\overline{(w_1aa^{\prime})} = m-1$, then $Q$ is mutation-equivalent (without performing mutations on $a$ or $a'$)
to the quiver $Q'$:

% https://q.uiver.app/#q=WzAsMTQsWzYsMiwiYV57XFxwcmltZX0iXSxbNCwyLCJ3XzEiXSxbNCwwLCJ3XzIiXSxbMywyLCJrXzFeMSJdLFsyLDIsIlxcY2RvdHMiXSxbMSwyLCJrX3tuXzF9XjEiXSxbMywwLCJrXzFeMiJdLFsxLDAsImteMl97bl8yfSJdLFswLDIsImEiXSxbMiwwLCJcXGNkb3RzIl0sWzYsMCwid19yIl0sWzcsMCwia18xXnIiXSxbOCwwLCJcXGNkb3RzIl0sWzksMCwia197bl9yfV5yIl0sWzEsMl0sWzMsMSwiMCIsMCx7ImNvbG91ciI6WzAsNjAsNjBdfSxbMCw2MCw2MCwxXV0sWzQsMywiMCIsMCx7ImNvbG91ciI6WzAsNjAsNjBdfSxbMCw2MCw2MCwxXV0sWzUsNCwiMCIsMCx7ImNvbG91ciI6WzAsNjAsNjBdfSxbMCw2MCw2MCwxXV0sWzYsMiwiMCJdLFsyLDBdLFsxLDAsIlxcYWxwaGEiLDAseyJjb2xvdXIiOlswLDYwLDYwXX0sWzAsNjAsNjAsMV1dLFs4LDUsIjAiLDAseyJjb2xvdXIiOlswLDYwLDYwXX0sWzAsNjAsNjAsMV1dLFs3LDksIjAiXSxbOSw2LCIwIl0sWzIsMTAsIiIsMSx7InN0eWxlIjp7ImJvZHkiOnsibmFtZSI6ImRhc2hlZCJ9fX1dLFsxMCwwXSxbMTAsMV0sWzExLDEwLCIwIiwyXSxbMTIsMTEsIjAiLDJdLFsxMywxMiwiMCIsMl0sWzgsMCwiXFx0ZXh0e2NlbnRlciBjeWNsZX0iLDIseyJjdXJ2ZSI6NSwiY29sb3VyIjpbMCw2MCw2MF19LFswLDYwLDYwLDFdXV0=
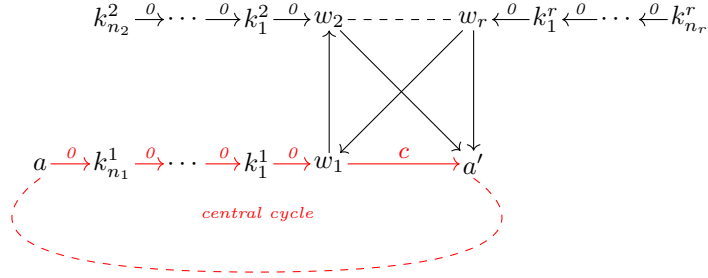
\[\begin{tikzpicture}[scale=0.8, x=1.2cm,y=1.2cm, every node/.style={inner sep=1pt}]
  %%% NODOS
  % Fila 1
  \node (r1c2) at (2,0) {$k^2_{n_2}$};
  \node (r1c3) at (3,0) {$\cdots$};
  \node (r1c4) at (4,0) {$k_1^2$};
  \node (r1c5) at (5,0) {$w_2$};
  \node (r1c7) at (7,0) {$w_r$};
  \node (r1c8) at (8,0) {$k_1^r$};
  \node (r1c9) at (9,0) {$\cdots$};
  \node (r1c10) at (10,0) {$k_{n_r}^r$};

  % Fila 3
  \node (r3c1) at (1,-2) {$a$};
  \node (r3c2) at (2,-2) {$k_{n_1}^1$};
  \node (r3c3) at (3,-2) {$\cdots$};
  \node (r3c4) at (4,-2) {$k_1^1$};
  \node (r3c5) at (5,-2) {$w_1$};
  \node (r3c7) at (7,-2) {$a^{\prime}$};

  %%% FLECHAS (con etiquetas como en el tikzcd)
  % Fila 1 izquierda
  \draw[->] (r1c2) -- node[above,yshift=1pt]{\tiny  0} (r1c3);
  \draw[->] (r1c3) -- node[above,yshift=1pt]{\tiny  0} (r1c4);
  \draw[->] (r1c4) -- node[above,yshift=1pt]{\tiny  0} (r1c5);

  % Enlaces desde w2
  \draw[dashed,-] (r1c5) -- (r1c7);
  \draw[->] (r1c5) -- (r3c7);

  % Desde w_r
  \draw[->] (r1c7) -- (r3c5);
  \draw[->] (r1c7) -- (r3c7);

  % Cadena derecha fila 1 (r)
  \draw[->] (r1c8) -- node[above,yshift=1pt,swap]{\tiny  0} (r1c7);
  \draw[->] (r1c9) -- node[above,yshift=1pt,swap]{\tiny  0} (r1c8);
  \draw[->] (r1c10) -- node[above,yshift=1pt,swap]{\tiny  0} (r1c9);

  % Fila 3 (rojo) a lo largo de la base
  \draw[red,->] (r3c1) -- node[above,red,yshift=1pt]{\tiny  0} (r3c2);
  \draw[red,->] (r3c2) -- node[above,red,yshift=1pt]{\tiny  0} (r3c3);
  \draw[red,->] (r3c3) -- node[above,red,yshift=1pt]{\tiny  0} (r3c4);
  \draw[red,->] (r3c4) -- node[above,red,yshift=1pt]{\tiny  0} (r3c5);
  \draw[->] (r3c5) -- (r1c5);
  \draw[red,->] (r3c5) -- node[above,red, yshift=1pt]{\small $c$} (r3c7);

  %%% central cycle: curva roja suave entre a y a'
  % Usa exactamente los extremos en a=(1,-2) y a'=(7,-2) con un punto medio más bajo para curvatura.
  \draw[red,, smooth, dashed, tension=3] plot coordinates{(1,-2.2) (4,-3.5) (7,-2.2)};
  \node[red] at (4,-2.7) {\tiny central cycle};

\end{tikzpicture}\]
where  the clique $Q[w_1, w_2, \cdots, w_r,a']$ is a boundary subquiver of $Q'$ of size $r+1$ and the vertices  $a, k_{n_1}^1, \cdots , k_{1}^1, w_1, a^{\prime}$ are  central vertices.
\end{lem}

\begin{proof}
With the considerations mentioned in Remark \autoref{ref:obs1}, we can affirm that the minimal color $c$ of the arrows leaving $w_1$ have $a$ or $a^\prime$ as their target. However, due to Lemma \autoref{ref:lema3}, the arrow connecting $w_1 \rightarrow a$ is the one with minimal color among those starting at $w_1$. Therefore, applying the mutation $\mu_{w_1}^{c+1}$, we obtain the quiver
% https://q.uiver.app/#q=WzAsMTQsWzYsMywiYV57XFxwcmltZX0iXSxbNCwzLCJ3XzEiXSxbNCwxLCJ3XzIiXSxbMiwxLCJrXzFeMSJdLFsxLDEsIlxcY2RvdHMiXSxbMCwxLCJrX3tuXzF9XjEiXSxbMywwLCJrXzFeMiJdLFsyLDAsIlxcY2RvdHMiXSxbMSwwLCJrXjJfe25fMn0iXSxbMiwzLCJhIl0sWzYsMSwid19yIl0sWzcsMCwia18xXnIiXSxbOCwwLCJcXGNkb3RzIl0sWzksMCwia197bl9yfV5yIl0sWzEsMl0sWzQsMywiMCJdLFs1LDQsIjAiXSxbNiwyLCIwIl0sWzcsNiwiMCJdLFs4LDcsIjAiXSxbMiwwXSxbMSwwLCJcXGFscGhhXzEiLDIseyJjb2xvdXIiOlswLDYwLDYwXX0sWzAsNjAsNjAsMV1dLFs5LDEsIjAiLDIseyJjb2xvdXIiOlswLDYwLDYwXX0sWzAsNjAsNjAsMV1dLFszLDksIjAiLDJdLFszLDEsIjEiXSxbMiwxMCwiIiwwLHsic3R5bGUiOnsiYm9keSI6eyJuYW1lIjoiZGFzaGVkIn19fV0sWzEwLDBdLFsxMCwxXSxbMTEsMTAsIjAiLDJdLFsxMiwxMSwiMCIsMl0sWzEzLDEyLCIwIiwyXSxbOSwwLCJcXHRleHR7Y2VudGVyIGN5Y2xlfSIsMix7ImN1cnZlIjo1LCJjb2xvdXIiOlswLDYwLDYwXX0sWzAsNjAsNjAsMV1dXQ==
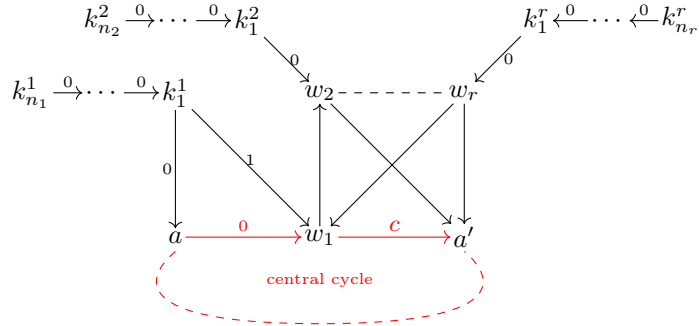
\[\begin{tikzpicture}[scale=0.8, x=1.2cm,y=1.2cm, every node/.style={inner sep=1pt}]
  %%% NODOS
  % Fila 1 (y=0)
  \node (r1c2) at (2,0) {$k^2_{n_2}$};
  \node (r1c3) at (3,0) {$\cdots$};
  \node (r1c4) at (4,0) {$k_1^2$};
  \node (r1c8) at (8,0) {$k_1^r$};
  \node (r1c9) at (9,0) {$\cdots$};
  \node (r1c10) at (10,0) {$k_{n_r}^r$};

  % Fila 2 (y=-1)
  \node (r2c1) at (1,-1) {$k_{n_1}^1$};
  \node (r2c2) at (2,-1) {$\cdots$};
  \node (r2c3) at (3,-1) {$k_1^1$};
  \node (r2c5) at (5,-1) {$w_2$};
  \node (r2c7) at (7,-1) {$w_r$};

  % Fila 4 (y=-3)  (hay una fila vacía entre medio)
  \node (r4c3) at (3,-3) {$a$};
  \node (r4c5) at (5,-3) {$w_1$};
  \node (r4c7) at (7,-3) {$a^{\prime}$};

  %%% FLECHAS
  % fila 1 izquierda
  \draw[->] (r1c2) -- node[above,yshift=1pt]{\tiny  0} (r1c3);
  \draw[->] (r1c3) -- node[above,yshift=1pt]{\tiny  0} (r1c4);
  \draw[->] (r1c4) -- node[right]{\tiny  0} (r2c5);

  % fila 1 derecha
  \draw[->] (r1c9) -- node[above,yshift=1pt,swap]{\tiny  0} (r1c8);
  \draw[->] (r1c10) -- node[above,yshift=1pt,swap]{\tiny  0} (r1c9);
  \draw[->] (r1c8) -- node[right,xshift=1pt]{\tiny  0} (r2c7);

  % fila 2 izquierda
  \draw[->] (r2c1) -- node[above,yshift=1pt]{\tiny  0} (r2c2);
  \draw[->] (r2c2) -- node[above,yshift=1pt]{\tiny  0} (r2c3);
  \draw[->] (r2c3) -- node[left]{\tiny  0} (r4c3); % 2-3 -> 4-3
  \draw[->] (r2c3) -- node[above]{\tiny  1} (r4c5); % 2-3 -> 4-5

  % enlaces entre w2 y wr
  \draw[dashed,-] (r2c5) -- (r2c7);
  \draw[->] (r2c5) -- (r4c7); % 2-5 -> 4-7

  % desde wr
  \draw[->] (r2c7) -- (r4c5);
  \draw[->] (r2c7) -- (r4c7);

  % base (roja) a -> w1
  \draw[red,->] (r4c3) -- node[above,red,yshift=1pt]{\tiny  0} (r4c5);

  % alpha roja w1 -> a'
  \draw[red,->] (r4c5) -- node[above,red, yshift=1pt]{\small $c$} (r4c7);

  % central cycle (curva roja suave entre a y a')
  \draw[red, smooth, dashed, tension=3] plot coordinates{(3,-3.2) (5,-4.2) (7,-3.2)};
  \node[red] at (5,-3.6) {\tiny central cycle};

  % enlace w1 -> w2 (4-5 -> 2-5)
  \draw[->] (r4c5) -- (r2c5);
\end{tikzpicture}\]
Performing the composition \(\mu_{k_{m_1}^1} \circ \cdots \circ \mu_{k_{1}^1}\) on this new quiver we get the desired quiver.
 
\end{proof}

%\newpage
\begin{lem}
\label{ref:lema ingreso2}
Let $Q$ be a quiver in the class $\mathcal{Q}^m_{p,q}$. Let $w_1, \cdots , w_r$ be peripheral vertices in $Q$ and assume that the vertices $a$ and $a'$ are central vertices in such a way that the clique $Q[w_1,w_2, \cdots,  w_r ,a,a']$ is a boundary subquiver of $Q$ of size $r+2$. We illustrate the situation in the following quiver:

% https://q.uiver.app/#q=WzAsMTQsWzMsMiwiYSJdLFs1LDIsImFee1xccHJpbWV9Il0sWzMsMSwid18xIl0sWzQsMCwid18yIl0sWzIsMSwia18xXjEiXSxbMCwxLCJrX3tuXzF9XjEiXSxbMywwLCJrXzFeMiJdLFsxLDAsImteMl97bl8yfSJdLFs1LDEsIndfciJdLFs2LDEsImtfMV5yIl0sWzcsMSwiXFxjZG90cyJdLFs4LDEsImtfe25fcn1eciJdLFsxLDEsIlxcY2RvdHMiXSxbMiwwLCJcXGNkb3RzIl0sWzAsMSwiXFxhbHBoYSIsMCx7ImNvbG91ciI6WzAsNjAsNjBdfSxbMCw2MCw2MCwxXV0sWzIsMF0sWzIsM10sWzQsMiwiMCJdLFs2LDMsIjAiXSxbMywwXSxbMywxXSxbMiwxXSxbMyw4LCIiLDIseyJzdHlsZSI6eyJib2R5Ijp7Im5hbWUiOiJkYXNoZWQifX19XSxbOCwxXSxbOCwwXSxbMiw4XSxbOSw4LCIwIiwyXSxbMTAsOSwiMCIsMl0sWzExLDEwLCIwIiwyXSxbNSwxMiwiMCJdLFsxMiw0LCIwIl0sWzcsMTMsIjAiXSxbMTMsNiwiMCJdLFswLDEsIlxcdGV4dHtjZW50ZXIgY3ljbGV9IiwyLHsib2Zmc2V0IjotMywiY3VydmUiOjUsImNvbG91ciI6WzAsNjAsNjBdfSxbMCw2MCw2MCwxXV1d
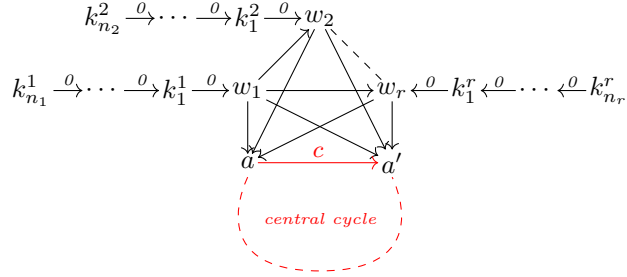
\[\begin{tikzpicture}[scale=0.8, x=1.2cm,y=1.2cm, every node/.style={inner sep=1pt}]
  %%% NODOS
  % Fila 1
  \node (r1c2) at (2,0) {$k^2_{n_2}$};
  \node (r1c3) at (3,0) {$\cdots$};
  \node (r1c4) at (4,0) {$k_1^2$};
  \node (r1c5) at (5,0) {$w_2$};

  % Fila 2
  \node (r2c1) at (1,-1) {$k_{n_1}^1$};
  \node (r2c2) at (2,-1) {$\cdots$};
  \node (r2c3) at (3,-1) {$k_1^1$};
  \node (r2c4) at (4,-1) {$w_1$};
  \node (r2c6) at (6,-1) {$w_r$};
  \node (r2c7) at (7,-1) {$k_1^r$};
  \node (r2c8) at (8,-1) {$\cdots$};
  \node (r2c9) at (9,-1) {$k_{n_r}^r$};

  % Fila 3
  \node (r3c4) at (4,-2) {$a$};
  \node (r3c6) at (6,-2) {$a^{\prime}$};

  %%% FLECHAS
  % Fila 1 (izquierda)
  \draw[->] (r1c2) -- node[above, yshift=1pt]{\tiny  0} (r1c3);
  \draw[->] (r1c3) -- node[above, yshift=1pt]{\tiny  0} (r1c4);
  \draw[->] (r1c4) -- node[above, yshift=1pt]{\tiny  0} (r1c5);

  % Conexiones desde w2
  \draw[dashed,-] (r1c5) -- (r2c6);
  \draw[->] (r1c5) -- (r3c4);
  \draw[->] (r1c5) -- (r3c6);

  % Fila 2 (cadena izquierda)
  \draw[->] (r2c1) -- node[above, yshift=1pt]{\tiny  0} (r2c2);
  \draw[->] (r2c2) -- node[above, yshift=1pt]{\tiny  0} (r2c3);
  \draw[->] (r2c3) -- node[above, yshift=1pt]{\tiny  0} (r2c4);

  % Desde w1
  \draw[->] (r2c4) -- (r1c5);
  \draw[->] (r2c4) -- (r2c6);
  \draw[->] (r2c4) -- (r3c4);
  \draw[->] (r2c4) -- (r3c6);

  % Desde w_r
  \draw[->] (r2c6) -- (r3c4);
  \draw[->] (r2c6) -- (r3c6);

  % Fila 2 (cadena derecha)
  \draw[->] (r2c7) -- node[above,swap]{\tiny  0} (r2c6);
  \draw[->] (r2c8) -- node[above,swap]{\tiny  0} (r2c7);
  \draw[->] (r2c9) -- node[above,swap]{\tiny  0} (r2c8);

  % Flecha roja alpha entre a y a'
  \draw[red,->] (r3c4) -- node[above,red,yshift=1pt]{\small $c$} (r3c6);

  % central cycle: curva roja suave entre a y a' (sin flecha)
  \draw[red, dashed, smooth, tension=3] plot coordinates{(4,-2.2) (5,-3.5) (6,-2.2)};
  \node[red] at (5,-2.8) {\tiny central cycle};

\end{tikzpicture}\]
If $\overline{(w_1aa^{\prime})} = 2m+1$ then,  $Q$ is mutation-equivalent to the quiver $Q'$ without performing mutations
on $a$ or $a^\prime$. 
% https://q.uiver.app/#q=WzAsMTQsWzUsMiwid18xIl0sWzMsMiwiYSJdLFszLDAsIndfMiJdLFs2LDIsImtfMV4xIl0sWzcsMiwiXFxjZG90cyJdLFs4LDIsImtfe25fMX1eMSJdLFsyLDAsImtfMV4yIl0sWzEsMCwiXFxjZG90cyJdLFswLDAsImteMl97bl8yfSJdLFs5LDIsImFeXFxwcmltZSJdLFs1LDAsIndfciJdLFs2LDAsImtfMV5yIl0sWzcsMCwiXFxjZG90cyJdLFs4LDAsImtfe25fcn1eciJdLFsxLDJdLFs0LDMsIjAiLDIseyJjb2xvdXIiOlswLDYwLDYwXX0sWzAsNjAsNjAsMV1dLFs1LDQsIjAiLDIseyJjb2xvdXIiOlswLDYwLDYwXX0sWzAsNjAsNjAsMV1dLFs2LDIsIjAiXSxbNyw2LCIwIl0sWzgsNywiMCJdLFsyLDBdLFsxLDAsIlxcYWxwaGEiLDAseyJjb2xvdXIiOlswLDYwLDYwXX0sWzAsNjAsNjAsMV1dLFs5LDUsIjAiLDIseyJjb2xvdXIiOlswLDYwLDYwXX0sWzAsNjAsNjAsMV1dLFszLDAsIjAiLDIseyJjb2xvdXIiOlswLDYwLDYwXX0sWzAsNjAsNjAsMV1dLFsyLDEwLCIiLDIseyJzdHlsZSI6eyJib2R5Ijp7Im5hbWUiOiJkYXNoZWQifX19XSxbMTAsMF0sWzEwLDFdLFsxMiwxMSwiMCIsMl0sWzEzLDEyLCIwIiwyXSxbMTEsMTAsIjAiLDJdLFsxLDksIlxcdGV4dHtjZW50ZXIgY3ljbGV9IiwyLHsiY3VydmUiOjUsImNvbG91ciI6WzAsNjAsNjBdfSxbMCw2MCw2MCwxXV1d
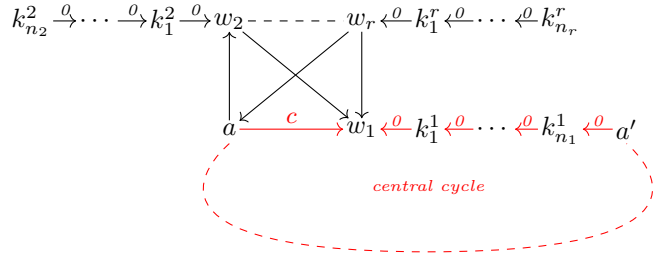
\[\begin{tikzpicture}[scale=0.8, x=1.1cm,y=1.2cm, every node/.style={inner sep=1pt}]
  %%% FILA 1
  \node (r1c1) at (1,0) {$k^2_{n_2}$};
  \node (r1c2) at (2,0) {$\cdots$};
  \node (r1c3) at (3,0) {$k_1^2$};
  \node (r1c4) at (4,0) {$w_2$};
  \node (r1c6) at (6,0) {$w_r$};
  \node (r1c7) at (7,0) {$k_1^r$};
  \node (r1c8) at (8,0) {$\cdots$};
  \node (r1c9) at (9,0) {$k_{n_r}^r$};

  %%% FILA 3
  \node (r3c4) at (4,-1.5) {$a$};
  \node (r3c6) at (6,-1.5) {$w_1$};
  \node (r3c7) at (7,-1.5) {$k_1^1$};
  \node (r3c8) at (8,-1.5) {$\cdots$};
  \node (r3c9) at (9,-1.5) {$k_{n_1}^1$};
  \node (r3c10) at (10,-1.5) {$a^{\prime}$};

  %%% FLECHAS
  % Fila 1 izquierda
  \draw[->] (r1c1) -- node[above, yshift=1pt]{\tiny  0} (r1c2);
  \draw[->] (r1c2) -- node[above, yshift=1pt]{\tiny  0} (r1c3);
  \draw[->] (r1c3) -- node[above, yshift=1pt]{\tiny  0} (r1c4);

  % Conexiones w2
  \draw[dashed,-] (r1c4) -- (r1c6);
  \draw[->] (r1c4) -- (r3c6);

  % Conexiones w_r
  \draw[->] (r1c6) -- (r3c4);
  \draw[->] (r1c6) -- (r3c6);

  % Cadena derecha fila 1
  \draw[->] (r1c7) -- node[above,swap]{\tiny  0} (r1c6);
  \draw[->] (r1c8) -- node[above,swap]{\tiny  0} (r1c7);
  \draw[->] (r1c9) -- node[above,swap]{\tiny  0} (r1c8);

  % Desde a
  \draw[->] (r3c4) -- (r1c4);

  % Flecha roja α
  \draw[red,->] (r3c4) -- node[above,red,yshift=1pt]{\small $c$} (r3c6);

  % Cadena roja derecha (k_i^1)
  \draw[red,->] (r3c7) -- node[above,red]{\tiny  0} (r3c6);
  \draw[red,->] (r3c8) -- node[above,red]{\tiny  0} (r3c7);
  \draw[red,->] (r3c9) -- node[above,red]{\tiny  0} (r3c8);
  \draw[red,->] (r3c10) -- node[above,red]{\tiny  0} (r3c9);

  % central cycle (curva suave entre a y a')
  \draw[red, dashed, smooth, tension=3] plot coordinates{(4,-1.7) (7,-3.2) (10,-1.7)};
  \node[red] at (7,-2.3) {\tiny central cycle};
\end{tikzpicture}\]
where  the clique $Q[w_1, w_2, \cdots, w_r,a]$ is a boundary subquiver of $Q'$ of size $r+1$ and the vertices  $a, w_1,  k_{1}^1, \cdots, k_{n_1}^1, a^{\prime}$ are  central vertices of $Q'$.
\end{lem}

\begin{proof}
The argument is analogous to the proof of Lemma~\autoref{ref:lema ingreso1}.
\end{proof}

Now, we are ready to prove the main result of this paper.

\begin{teo}
\label{ref:teo principal}
    A connected $m$-colored quiver $Q$ is of type $\widetilde{\mathbb{A}}_{p,q}$ if and only if $Q$ belongs to the class $\mathcal{Q}^m_{p,q}.$
\end{teo}

\begin{proof}

Suppose that $Q$ is a connected $m$-colored quiver mutation-equivalent to $\widetilde{A}_{p,q}$  for $p,q\geq 1$. Since $\widetilde{A}_{p,q}$ belongs to $\mathcal{Q}^m_{p,q}$ and this class is closed under mutations, then $Q$ also belongs to $\mathcal{Q}^m_{p,q}$.

Conversely, suppose that $Q\in \mathcal{Q}^m_{p,q}$. Then, $Q$ has a $(l,h)$-central cycle called the central cycle and a set $W$ of peripheral vertices. For each  $w\in W$,  there exits  two central vertices $a,a'$ which form a triangle $T_w=(aa'w)$ and a subquiver $Q_w$ with $n_w$ vertices which belongs to the class $Q^m_{n_w}$. Moreover,  $W=W_p\cup W_q$ with \(W_p=\left\{w\in W \,\, | \,\, \overline{T_w}=m-1\right\}\) and  \(W_q=\left\{w\in W\,\, | \,\, \overline{T_w}=2m+1\right\}\). If we take   the quantities
\(x_{p}=\sum_{w\in W_p} n_{w}\) and \(x_{q}=\sum_{w\in W_q} n_{w},\) we  have that  $p=l-h+x_{p}$ and  $q= h+x_{q}$. We want to show that $Q$ is mutation equivalent to $\tilde{A}_{p,q}$.

For each peripheral vertex $w\in W$, due to Lemma \ref{Qw es mut eq a Anw}, we may assume that $Q_w$ has underlying graph of Dynkin type $\mathbb{A}_{n_{w}},$ for $n_{w}\geq 0.$ Therefore, every admissible triangle $T_w=(aa^{\prime}w)$,  with $a,a'$ central vertices, belongs to a subquiver of $Q$ as follows:
% https://q.uiver.app/#q=WzAsNixbNSwwLCJhX2kiXSxbNSw0LCJhX3tpKzF9Il0sWzMsMiwidyJdLFswLDIsImtfe25fdy0xfSJdLFsyLDIsImtfMSJdLFsxLDIsIlxcY2RvdHMiXSxbMCwxLCJcXGFscGhhX2kiLDAseyJjb2xvdXIiOlswLDYwLDYwXX0sWzAsNjAsNjAsMV1dLFsyLDBdLFsyLDFdLFsxLDAsIlxcdGV4dHtjZW50ZXIgY3ljbGV9IiwyLHsib2Zmc2V0IjotMywiY3VydmUiOjUsImNvbG91ciI6WzAsNjAsNjBdfSxbMCw2MCw2MCwxXV0sWzIsNF0sWzQsNV0sWzUsM11d
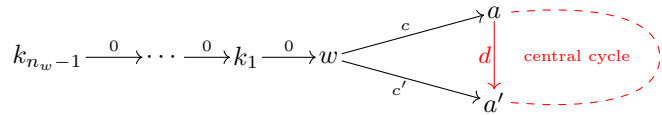
\[\begin{tikzpicture}[x=1.1cm,y=1.2cm, every node/.style={inner sep=1pt}]
  % NODOS
  \node (r1c6) at (6,-1) {$a$};
  \node (r3c1) at (0.6,-1.5) {$k_{n_w-1}$};
  \node (r3c2) at (2,-1.5) {$\cdots$};
  \node (r3c3) at (3,-1.5) {$k_1$};
  \node (r3c4) at (4,-1.5) {$w$};
  \node (r5c6) at (6,-2) {$a^{\prime}$};

  % FLECHAS
  \draw[red,->] (r1c6) -- node[left,red]{\small $d$} (r5c6);
  \draw[->] (r3c1) -- node[above,swap, yshift=1pt]{\tiny  0} (r3c2);
  \draw[->] (r3c2) -- node[above,swap, yshift=1pt]{\tiny  0} (r3c3);
  \draw[->] (r3c4) -- node[left,yshift=3pt]{\tiny  $c$} (r1c6);
  \draw[->] (r3c3) -- node[above,swap, yshift=1pt]{\tiny  0} (r3c4);
  \draw[->] (r3c4) -- node[left,swap,yshift=-3pt]{\tiny  $c^\prime$} (r5c6);

  % central cycle (curva roja suave entre a y a')
  \draw[red, smooth, dashed, tension=2.5] plot coordinates{(6.2,-2) (8,-1.5) (6.2,-1)};
  \node[red] at (7,-1.5) {\tiny central cycle};
\end{tikzpicture}\]

If  $\overline{T_w}=m-1$, Lemma \ref{ref:lema ingreso1} implies that the subquiver $Q_w$ "enters" the central cycle (via mutations), transforming the $(l,h)$-central cycle into an $(l+n_w,h)$-central cycle, modifying the sequence of central arrows as follows:
% https://q.uiver.app/#q=WzAsNSxbMywwLCJhX2kiXSxbMywyLCJhX3tpKzF9Il0sWzEsMiwidyJdLFsxLDAsImtfe25fdy0xfSJdLFswLDEsImtfMSJdLFsxLDAsIlxcdGV4dHtjZW50ZXIgY3ljbGV9IiwyLHsib2Zmc2V0IjotMywiY3VydmUiOjUsImNvbG91ciI6WzAsNjAsNjBdfSxbMCw2MCw2MCwxXV0sWzIsMSwiXFxhbHBoYV9pIiwyXSxbMCwzLCIwIiwyXSxbMyw0LCIwIiwyLHsic3R5bGUiOnsiYm9keSI6eyJuYW1lIjoiZGFzaGVkIn19fV0sWzQsMiwiMCIsMl1d
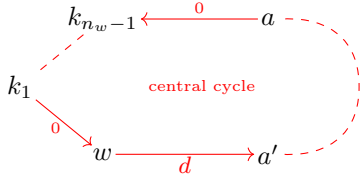
\[\begin{tikzpicture}[scale=1,x=1.1cm,y=1.8cm, every node/.style={inner sep=1pt}]
  % NODOS
  \node (r1c2) at (2,0) {$k_{n_w-1}$};
  \node (r1c4) at (4,0) {$a$};
  \node (r2c1) at (1,-0.5) {$k_1$};
  \node (r3c2) at (2,-1) {$w$};
  \node (r3c4) at (4,-1) {$a^\prime$};

  % FLECHAS ROJAS
  \draw[red,dashed,-] (r1c2) -- node[above,red]{} (r2c1);
  \draw[red,->] (r1c4) -- node[above,red, yshift=1pt]{\tiny  0} (r1c2);
  \draw[red,->] (r2c1) -- node[left, yshift=-2pt, red]{\tiny  0} (r3c2);
  \draw[red,->] (r3c2) -- node[red,yshift=-4pt]{\small $d$} (r3c4);

  % central cycle (curva roja suave entre a y a')
  \draw[red, smooth, dashed, tension=2] plot coordinates{(4.2,-1) (5.1,-0.5) (4.2,0)};
  \node[red] at (3.2,-0.5) {\tiny central cycle};
\end{tikzpicture}\]

On the other hand, if  $\overline{T_w}=2m+1$, Lemma \ref{ref:lema ingreso2} implies that the subquiver $Q_w$ "enters" the central cycle (via mutations), transforming the $(l,h)$-central cycle into an $(l+n_w,h+n_w)$-central cycle, modifying the sequence of central arrows as follows: 
% https://q.uiver.app/#q=WzAsNSxbMywwLCJhX2kiXSxbMywyLCJhX3tpKzF9Il0sWzEsMCwicV9qXmkiXSxbMSwyLCJrX3ttX3tpLGp9LTF9XntpLGp9Il0sWzAsMSwia18xXntpLGp9Il0sWzEsMCwiXFx0ZXh0e2NlbnRlcn0iLDEseyJvZmZzZXQiOi0zLCJjdXJ2ZSI6NSwiY29sb3VyIjpbMCw2MCw2MF19LFswLDYwLDYwLDFdXSxbMCwyLCJcXGFscGhhX2kiLDJdLFsyLDQsIm0iLDJdLFs0LDMsIm0iLDIseyJzdHlsZSI6eyJib2R5Ijp7Im5hbWUiOiJkYXNoZWQifX19XSxbMywxLCJtIiwyXV0=
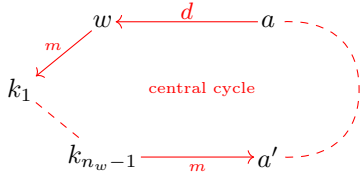
\[\begin{tikzpicture}[scale=1, x=1.1cm,y=1.8cm, every node/.style={inner sep=1pt}]
  % NODOS
  \node (r1c2) at (2,0) {$w$};
  \node (r1c4) at (4,0) {$a$};
  \node (r2c1) at (1,-0.5) {$k_1$};
  \node (r3c2) at (2,-1) {$k_{n_w-1}$};
  \node (r3c4) at (4,-1) {$a^\prime$};

  % FLECHAS ROJAS
  \draw[red,->] (r1c2) -- node[above,red,swap, xshift=-4pt,]{\tiny  $m$} (r2c1);
  \draw[red,->] (r1c4) -- node[red,yshift=4pt]{\small $d$} (r1c2);
  \draw[red,dashed,-] (r2c1) -- node[left, red]{} (r3c2);
  \draw[red,->] (r3c2) -- node[red,yshift=-4pt]{\tiny  $m$} (r3c4);

  % central cycle (curva roja suave entre a y a')
\draw[red, smooth, dashed, tension=2] plot coordinates{(4.2,0) (5.1,-0.5) (4.2,-1)};
 \node[red] at (3.2,-0.5) {\tiny central cycle};
\end{tikzpicture}\]

Applying the previous process to each $T_w$, the quiver $Q$ is mutation–equivalent to an $(l+x_p+x_q,h+x_q)$–central cycle, which, by Proposition \ref{ref:ciclo central}, is mutation–equivalent to $\widetilde{A}_{p,q}.$
\end{proof}

We illustrate the previous procedure in the following example.

\begin{ejem}
Let $m=2$ and consider the $2$–colored quiver $Q$
% https://q.uiver.app/#q=WzAsMTEsWzIsMywiXFx0ZXh0Y29sb3J7cmVkfXthXzR9Il0sWzIsMSwiXFx0ZXh0Y29sb3J7cmVkfXthXzF9Il0sWzAsMSwiXFx0ZXh0Y29sb3J7cmVkfXthXzJ9Il0sWzAsMywiXFx0ZXh0Y29sb3J7cmVkfXthXzN9Il0sWzQsMSwiXFx0ZXh0Y29sb3J7Ymx1ZX17dn0iXSxbNCwzLCJcXHRleHRjb2xvcntibHVlfXt3fSJdLFszLDAsInhfMiJdLFs1LDAsInhfMSJdLFs1LDQsInlfMiJdLFszLDQsInlfMSJdLFs2LDQsInlfMyJdLFswLDEsIjEiLDAseyJjb2xvdXIiOlswLDYwLDYwXX0sWzAsNjAsNjAsMV1dLFsxLDIsIjAiLDAseyJjb2xvdXIiOlswLDYwLDYwXX0sWzAsNjAsNjAsMV1dLFsxLDQsIjAiXSxbNCwwLCIwIiwxLHsibGFiZWxfcG9zaXRpb24iOjIwfV0sWzUsMCwiMiIsMl0sWzUsNCwiMSIsMl0sWzEsNSwiMiIsMSx7ImxhYmVsX3Bvc2l0aW9uIjoyMH1dLFs2LDQsIjIiLDJdLFs0LDcsIjIiLDJdLFs3LDYsIjEiLDJdLFs4LDUsIjEiLDJdLFs1LDksIjAiLDJdLFs5LDgsIjAiLDJdLFsxMCw4LCIwIiwyXSxbMiwzLCIyIiwwLHsiY29sb3VyIjpbMCw2MCw2MF19LFswLDYwLDYwLDFdXSxbMywwLCIxIiwwLHsiY29sb3VyIjpbMCw2MCw2MF19LFswLDYwLDYwLDFdXV0=
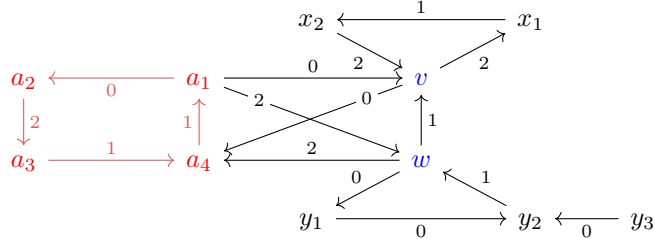
\[\begin{tikzcd}[row sep=small]
	&&& {x_2} && {x_1} \\
	{\textcolor{red}{a_2}} && {\textcolor{red}{a_1}} && {\textcolor{blue}{v}} \\
	\\
	{\textcolor{red}{a_3}} && {\textcolor{red}{a_4}} && {\textcolor{blue}{w}} \\
	&&& {y_1} && {y_2} & {y_3}
	\arrow["2"', from=1-4, to=2-5]
	\arrow["1"', from=1-6, to=1-4]
	\arrow["2", color={rgb,255:red,214;green,92;blue,92}, from=2-1, to=4-1]
	\arrow["0", color={rgb,255:red,214;green,92;blue,92}, from=2-3, to=2-1]
	\arrow["0", from=2-3, to=2-5]
	\arrow["2"{description, pos=0.2}, from=2-3, to=4-5]
	\arrow["2"', from=2-5, to=1-6]
	\arrow["0"{description, pos=0.2}, from=2-5, to=4-3]
	\arrow["1", color={rgb,255:red,214;green,92;blue,92}, from=4-1, to=4-3]
	\arrow["1", color={rgb,255:red,214;green,92;blue,92}, from=4-3, to=2-3]
	\arrow["1"', from=4-5, to=2-5]
	\arrow["2"', from=4-5, to=4-3]
	\arrow["0"', from=4-5, to=5-4]
	\arrow["0"', from=5-4, to=5-6]
	\arrow["1"', from=5-6, to=4-5]
	\arrow["0", from=5-7, to=5-6]
\end{tikzcd}\]
The $4-$cycle $(a_1\,a_2\,a_3\,a_4)$ is $(4,2)$–central and the vertices $w$ and $v$ are peripheral. Moreover, $v\in Q_v=Q[v,x_1,x_2]$ a subquiver of type $\mathbb{A}_3$ and  $w\in Q_w=Q[w,y_1,y_2,y_3]$ a subquiver of type $\mathbb{A}_4$.  In this case $W_p=\left\{v\right\}$ and $W_q=\left\{w\right\}$, 
hence \(x_{p}=3\) and \(x_{q}=4\). 

The composition $\mu_{x_1}\circ \mu_{y_1}$ transforms the quiver $Q$ into the  quiver $Q^\prime$:
% https://q.uiver.app/#q=WzAsMTEsWzIsMiwiXFx0ZXh0Y29sb3J7cmVkfXthXzR9Il0sWzIsMCwiXFx0ZXh0Y29sb3J7cmVkfXthXzF9Il0sWzAsMCwiXFx0ZXh0Y29sb3J7cmVkfXthXzJ9Il0sWzAsMiwiXFx0ZXh0Y29sb3J7cmVkfXthXzN9Il0sWzQsMCwiXFx0ZXh0Y29sb3J7Ymx1ZX17dn0iXSxbNCwyLCJcXHRleHRjb2xvcntibHVlfXt3fSJdLFs2LDAsInhfMiJdLFs1LDAsInhfMSJdLFs2LDIsInlfMiJdLFs1LDIsInlfMSJdLFs3LDIsInlfMyJdLFswLDEsIjEiLDAseyJjb2xvdXIiOlswLDYwLDYwXX0sWzAsNjAsNjAsMV1dLFsxLDIsIjAiLDAseyJjb2xvdXIiOlswLDYwLDYwXX0sWzAsNjAsNjAsMV1dLFsxLDQsIjAiXSxbNCwwLCIwIiwxLHsibGFiZWxfcG9zaXRpb24iOjIwfV0sWzUsMCwiMiIsMl0sWzUsNCwiMSIsMl0sWzEsNSwiMiIsMSx7ImxhYmVsX3Bvc2l0aW9uIjoyMH1dLFs0LDcsIjAiXSxbNyw2LCIwIl0sWzUsOSwiMSIsMl0sWzEwLDgsIjAiXSxbMiwzLCIyIiwwLHsiY29sb3VyIjpbMCw2MCw2MF19LFswLDYwLDYwLDFdXSxbMywwLCIxIiwwLHsiY29sb3VyIjpbMCw2MCw2MF19LFswLDYwLDYwLDFdXSxbOCw5LCIwIl1d
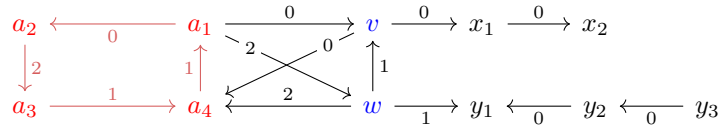
\[\begin{tikzcd}[row sep=small]
	{\textcolor{red}{a_2}} && {\textcolor{red}{a_1}} && {\textcolor{blue}{v}} & {x_1} & {x_2} \\
	\\
	{\textcolor{red}{a_3}} && {\textcolor{red}{a_4}} && {\textcolor{blue}{w}} & {y_1} & {y_2} & {y_3}
	\arrow["2", color={rgb,255:red,214;green,92;blue,92}, from=1-1, to=3-1]
	\arrow["0", color={rgb,255:red,214;green,92;blue,92}, from=1-3, to=1-1]
	\arrow["0", from=1-3, to=1-5]
	\arrow["2"{description, pos=0.2}, from=1-3, to=3-5]
	\arrow["0", from=1-5, to=1-6]
	\arrow["0"{description, pos=0.2}, from=1-5, to=3-3]
	\arrow["0", from=1-6, to=1-7]
	\arrow["1", color={rgb,255:red,214;green,92;blue,92}, from=3-1, to=3-3]
	\arrow["1", color={rgb,255:red,214;green,92;blue,92}, from=3-3, to=1-3]
	\arrow["1"', from=3-5, to=1-5]
	\arrow["2"', from=3-5, to=3-3]
	\arrow["1"', from=3-5, to=3-6]
	\arrow["0", from=3-7, to=3-6]
	\arrow["0", from=3-8, to=3-7]
\end{tikzcd}\]
Applying the composition of mutations $\mu_{x_2}^2\circ\mu_{x_1}^2\circ\mu_{x_2}^2$ we obtain the quiver
% https://q.uiver.app/#q=WzAsMTEsWzIsMiwiXFx0ZXh0Y29sb3J7cmVkfXthXzR9Il0sWzIsMCwiXFx0ZXh0Y29sb3J7cmVkfXthXzF9Il0sWzAsMCwiXFx0ZXh0Y29sb3J7cmVkfXthXzJ9Il0sWzAsMiwiXFx0ZXh0Y29sb3J7cmVkfXthXzN9Il0sWzQsMCwiXFx0ZXh0Y29sb3J7Ymx1ZX17dn0iXSxbNCwyLCJcXHRleHRjb2xvcntibHVlfXt3fSJdLFs2LDAsInhfMiJdLFs1LDAsInhfMSJdLFs2LDIsInlfMiJdLFs1LDIsInlfMSJdLFs3LDIsInlfMyJdLFswLDEsIjEiLDAseyJjb2xvdXIiOlswLDYwLDYwXX0sWzAsNjAsNjAsMV1dLFsxLDIsIjAiLDAseyJjb2xvdXIiOlswLDYwLDYwXX0sWzAsNjAsNjAsMV1dLFsxLDQsIjAiXSxbNCwwLCIwIiwxLHsibGFiZWxfcG9zaXRpb24iOjIwfV0sWzUsMCwiMiIsMl0sWzUsNCwiMSIsMl0sWzEsNSwiMiIsMSx7ImxhYmVsX3Bvc2l0aW9uIjoyMH1dLFs1LDksIjEiLDJdLFsxMCw4LCIwIl0sWzIsMywiMiIsMCx7ImNvbG91ciI6WzAsNjAsNjBdfSxbMCw2MCw2MCwxXV0sWzMsMCwiMSIsMCx7ImNvbG91ciI6WzAsNjAsNjBdfSxbMCw2MCw2MCwxXV0sWzgsOSwiMCJdLFs3LDQsIjAiLDJdLFs2LDcsIjAiLDJdXQ==
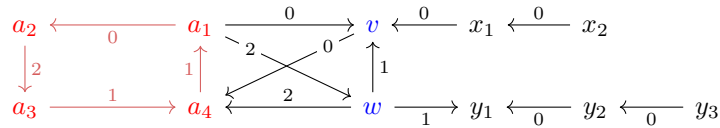
\[\begin{tikzcd}[row sep=small]
	{\textcolor{red}{a_2}} && {\textcolor{red}{a_1}} && {\textcolor{blue}{v}} & {x_1} & {x_2} \\
	\\
	{\textcolor{red}{a_3}} && {\textcolor{red}{a_4}} && {\textcolor{blue}{w}} & {y_1} & {y_2} & {y_3}
	\arrow["2", color={rgb,255:red,214;green,92;blue,92}, from=1-1, to=3-1]
	\arrow["0", color={rgb,255:red,214;green,92;blue,92}, from=1-3, to=1-1]
	\arrow["0", from=1-3, to=1-5]
	\arrow["2"{description, pos=0.2}, from=1-3, to=3-5]
	\arrow["0"{description, pos=0.2}, from=1-5, to=3-3]
	\arrow["0"', from=1-6, to=1-5]
	\arrow["0"', from=1-7, to=1-6]
	\arrow["1", color={rgb,255:red,214;green,92;blue,92}, from=3-1, to=3-3]
	\arrow["1", color={rgb,255:red,214;green,92;blue,92}, from=3-3, to=1-3]
	\arrow["1"', from=3-5, to=1-5]
	\arrow["2"', from=3-5, to=3-3]
	\arrow["1"', from=3-5, to=3-6]
	\arrow["0", from=3-7, to=3-6]
	\arrow["0", from=3-8, to=3-7]
\end{tikzcd}\]
Furthermore, applying the composition of mutations  $\mu_{y_3}\circ\mu_{y_2}\circ\mu_{y_1}$ on this last quiver we obtain the quiver $Q^{\prime\prime}$
% https://q.uiver.app/#q=WzAsMTEsWzIsMiwiXFx0ZXh0Y29sb3J7cmVkfXthXzR9Il0sWzIsMCwiXFx0ZXh0Y29sb3J7cmVkfXthXzF9Il0sWzAsMCwiXFx0ZXh0Y29sb3J7cmVkfXthXzJ9Il0sWzAsMiwiXFx0ZXh0Y29sb3J7cmVkfXthXzN9Il0sWzQsMCwiXFx0ZXh0Y29sb3J7Ymx1ZX17dn0iXSxbNCwyLCJcXHRleHRjb2xvcntibHVlfXt3fSJdLFs2LDAsInhfMiJdLFs1LDAsInhfMSJdLFs2LDIsInlfMiJdLFs1LDIsInlfMSJdLFs3LDIsInlfMyJdLFswLDEsIjEiLDAseyJjb2xvdXIiOlswLDYwLDYwXX0sWzAsNjAsNjAsMV1dLFsxLDIsIjAiLDAseyJjb2xvdXIiOlswLDYwLDYwXX0sWzAsNjAsNjAsMV1dLFsxLDQsIjAiXSxbNCwwLCIwIiwxLHsibGFiZWxfcG9zaXRpb24iOjIwfV0sWzUsMCwiMiIsMl0sWzUsNCwiMSIsMl0sWzEsNSwiMiIsMSx7ImxhYmVsX3Bvc2l0aW9uIjoyMH1dLFsxMCw4LCIwIl0sWzIsMywiMiIsMCx7ImNvbG91ciI6WzAsNjAsNjBdfSxbMCw2MCw2MCwxXV0sWzgsOSwiMCJdLFs5LDUsIjAiXSxbNyw0LCIwIiwyXSxbNiw3LCIwIiwyXSxbMywwLCIxIiwwLHsiY29sb3VyIjpbMCw2MCw2MF19LFswLDYwLDYwLDFdXV0=
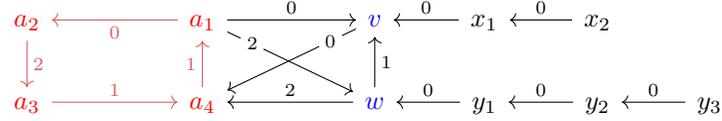
\[\begin{tikzcd}[row sep=small]
	{\textcolor{red}{a_2}} && {\textcolor{red}{a_1}} && {\textcolor{blue}{v}} & {x_1} & {x_2} \\
	\\
	{\textcolor{red}{a_3}} && {\textcolor{red}{a_4}} && {\textcolor{blue}{w}} & {y_1} & {y_2} & {y_3}
	\arrow["2", color={rgb,255:red,214;green,92;blue,92}, from=1-1, to=3-1]
	\arrow["0", color={rgb,255:red,214;green,92;blue,92}, from=1-3, to=1-1]
	\arrow["0", from=1-3, to=1-5]
	\arrow["2"{description, pos=0.2}, from=1-3, to=3-5]
	\arrow["0"{description, pos=0.2}, from=1-5, to=3-3]
	\arrow["0"', from=1-6, to=1-5]
	\arrow["0"', from=1-7, to=1-6]
	\arrow["1", color={rgb,255:red,214;green,92;blue,92}, from=3-1, to=3-3]
	\arrow["1", color={rgb,255:red,214;green,92;blue,92}, from=3-3, to=1-3]
	\arrow["1"', from=3-5, to=1-5]
	\arrow["2"', from=3-5, to=3-3]
	\arrow["0"', from=3-6, to=3-5]
	\arrow["0"', from=3-7, to=3-6]
	\arrow["0"', from=3-8, to=3-7]
\end{tikzcd}\]

Performing the mutation $\mu_v$ on $Q^{\prime\prime}$, the vertex $v$ becomes a central vertex as shown in the following figure
% https://q.uiver.app/#q=WzAsMTEsWzMsNCwiXFx0ZXh0Y29sb3J7cmVkfXthXzR9Il0sWzIsMSwiXFx0ZXh0Y29sb3J7cmVkfXthXzF9Il0sWzAsMiwiXFx0ZXh0Y29sb3J7cmVkfXthXzJ9Il0sWzEsNCwiXFx0ZXh0Y29sb3J7cmVkfXthXzN9Il0sWzQsMiwiXFx0ZXh0Y29sb3J7Ymx1ZX17dn0iXSxbNCwwLCJcXHRleHRjb2xvcntibHVlfXt3fSJdLFs2LDQsInhfMiJdLFs1LDQsInhfMSJdLFs2LDAsInlfMiJdLFs1LDAsInlfMSJdLFs3LDAsInlfMyJdLFsxLDIsIjAiLDAseyJjb2xvdXIiOlswLDYwLDYwXX0sWzAsNjAsNjAsMV1dLFs1LDQsIjIiXSxbMSw1LCIyIl0sWzEwLDgsIjAiXSxbMiwzLCIyIiwwLHsiY29sb3VyIjpbMCw2MCw2MF19LFswLDYwLDYwLDFdXSxbOCw5LCIwIl0sWzksNSwiMCJdLFs3LDQsIjEiLDJdLFs2LDcsIjAiXSxbNywwLCIwIl0sWzAsNCwiMCIsMCx7ImNvbG91ciI6WzAsNjAsNjBdfSxbMCw2MCw2MCwxXV0sWzQsMSwiMSIsMCx7ImNvbG91ciI6WzAsNjAsNjBdfSxbMCw2MCw2MCwxXV0sWzMsMCwiMSIsMix7ImNvbG91ciI6WzAsNjAsNjBdfSxbMCw2MCw2MCwxXV1d
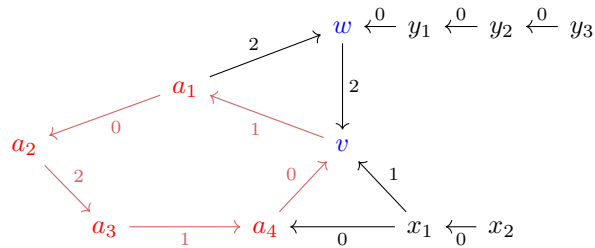
\[\begin{tikzcd}[row sep=small, column sep=small]
	&&&& {\textcolor{blue}{w}} & {y_1} & {y_2} & {y_3} \\
	&& {\textcolor{red}{a_1}} \\
	{\textcolor{red}{a_2}} &&&& {\textcolor{blue}{v}} \\
	\\
	& {\textcolor{red}{a_3}} && {\textcolor{red}{a_4}} && {x_1} & {x_2}
	\arrow["2", from=1-5, to=3-5]
	\arrow["0"', from=1-6, to=1-5]
	\arrow["0"', from=1-7, to=1-6]
	\arrow["0"', from=1-8, to=1-7]
	\arrow["2", from=2-3, to=1-5]
	\arrow["0", color={rgb,255:red,214;green,92;blue,92}, from=2-3, to=3-1]
	\arrow["2", color={rgb,255:red,214;green,92;blue,92}, from=3-1, to=5-2]
	\arrow["1", color={rgb,255:red,214;green,92;blue,92}, from=3-5, to=2-3]
	\arrow["1"', color={rgb,255:red,214;green,92;blue,92}, from=5-2, to=5-4]
	\arrow["0", color={rgb,255:red,214;green,92;blue,92}, from=5-4, to=3-5]
	\arrow["1"', from=5-6, to=3-5]
	\arrow["0", from=5-6, to=5-4]
	\arrow["0", from=5-7, to=5-6]
\end{tikzcd}\]
Applying Lemma \autoref{ref:lema ingreso1}, the composition $\mu_{x_2}\circ \mu_{x_1}$ applied on this last quiver  transforms the vertices  $x_1$ and $x_2$ into central vertices, as shown in the following figure
% https://q.uiver.app/#q=WzAsMTEsWzAsMiwiXFx0ZXh0Y29sb3J7cmVkfXthXzR9Il0sWzMsMCwiXFx0ZXh0Y29sb3J7cmVkfXthXzF9Il0sWzEsMCwiXFx0ZXh0Y29sb3J7cmVkfXthXzJ9Il0sWzAsMSwiXFx0ZXh0Y29sb3J7cmVkfXthXzN9Il0sWzQsMiwiXFx0ZXh0Y29sb3J7Ymx1ZX17dn0iXSxbNSwwLCJcXHRleHRjb2xvcntibHVlfXt3fSJdLFszLDMsInhfMSJdLFsxLDMsInhfMiJdLFs3LDAsInlfMiJdLFs2LDAsInlfMSJdLFs4LDAsInlfMyJdLFsxLDIsIjAiLDAseyJjb2xvdXIiOlswLDYwLDYwXX0sWzAsNjAsNjAsMV1dLFs1LDQsIjIiXSxbMSw1LCIyIl0sWzEwLDgsIjAiXSxbMiwzLCIyIiwwLHsiY29sb3VyIjpbMCw2MCw2MF19LFswLDYwLDYwLDFdXSxbOCw5LCIwIl0sWzksNSwiMCJdLFs0LDEsIjEiLDAseyJjb2xvdXIiOlswLDYwLDYwXX0sWzAsNjAsNjAsMV1dLFszLDAsIjEiLDAseyJjb2xvdXIiOlswLDYwLDYwXX0sWzAsNjAsNjAsMV1dLFswLDcsIjAiLDAseyJjb2xvdXIiOlswLDYwLDYwXX0sWzAsNjAsNjAsMV1dLFs3LDYsIjAiLDAseyJjb2xvdXIiOlswLDYwLDYwXX0sWzAsNjAsNjAsMV1dLFs2LDQsIjAiLDAseyJjb2xvdXIiOlswLDYwLDYwXX0sWzAsNjAsNjAsMV1dXQ==
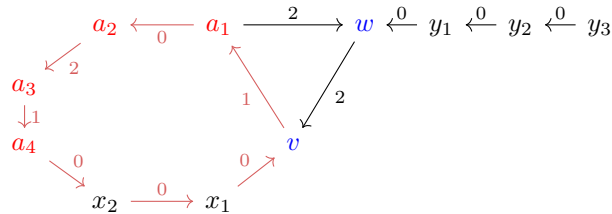
\[\begin{tikzcd}[row sep=small, column sep=small]
	& {\textcolor{red}{a_2}} && {\textcolor{red}{a_1}} && {\textcolor{blue}{w}} & {y_1} & {y_2} & {y_3} \\
	{\textcolor{red}{a_3}} \\
	{\textcolor{red}{a_4}} &&&& {\textcolor{blue}{v}} \\
	& {x_2} && {x_1}
	\arrow["2", color={rgb,255:red,214;green,92;blue,92}, from=1-2, to=2-1]
	\arrow["0", color={rgb,255:red,214;green,92;blue,92}, from=1-4, to=1-2]
	\arrow["2", from=1-4, to=1-6]
	\arrow["2", from=1-6, to=3-5]
	\arrow["0"', from=1-7, to=1-6]
	\arrow["0"', from=1-8, to=1-7]
	\arrow["0"', from=1-9, to=1-8]
	\arrow["1", color={rgb,255:red,214;green,92;blue,92}, from=2-1, to=3-1]
	\arrow["0", color={rgb,255:red,214;green,92;blue,92}, from=3-1, to=4-2]
	\arrow["1", color={rgb,255:red,214;green,92;blue,92}, from=3-5, to=1-4]
	\arrow["0", color={rgb,255:red,214;green,92;blue,92}, from=4-2, to=4-4]
	\arrow["0", color={rgb,255:red,214;green,92;blue,92}, from=4-4, to=3-5]
\end{tikzcd}\]

Similarly, proceeding as indicated in Lemma \autoref{ref:lema ingreso2}, the composition $\mu_{y_3}\circ\mu_{y_2}\circ\mu_{y_1}\circ\mu_{w}$ applied  on this last quiver transforms the vertices $w,y_1,y_2,y_3$ into central vertices as shown in the following figure:
% https://q.uiver.app/#q=WzAsMTEsWzAsMywiXFx0ZXh0Y29sb3J7cmVkfXthXzR9Il0sWzEsMCwiXFx0ZXh0Y29sb3J7cmVkfXthXzF9Il0sWzAsMSwiXFx0ZXh0Y29sb3J7cmVkfXthXzJ9Il0sWzAsMiwiXFx0ZXh0Y29sb3J7cmVkfXthXzN9Il0sWzMsNCwiXFx0ZXh0Y29sb3J7Ymx1ZX17dn0iXSxbNCwzLCJcXHRleHRjb2xvcntibHVlfXt3fSJdLFsyLDQsInhfMSJdLFsxLDQsInhfMiJdLFs0LDEsInlfMiJdLFs0LDIsInlfMSJdLFszLDAsInlfMyJdLFsxLDIsIjAiLDAseyJjb2xvdXIiOlswLDYwLDYwXX0sWzAsNjAsNjAsMV1dLFsxMCw4LCIwIiwyLHsiY29sb3VyIjpbMCw2MCw2MF19LFswLDYwLDYwLDFdXSxbMiwzLCIyIiwwLHsiY29sb3VyIjpbMCw2MCw2MF19LFswLDYwLDYwLDFdXSxbOCw5LCIwIiwyLHsiY29sb3VyIjpbMCw2MCw2MF19LFswLDYwLDYwLDFdXSxbOSw1LCIwIiwyLHsiY29sb3VyIjpbMCw2MCw2MF19LFswLDYwLDYwLDFdXSxbMywwLCIxIiwwLHsiY29sb3VyIjpbMCw2MCw2MF19LFswLDYwLDYwLDFdXSxbMCw3LCIwIiwwLHsiY29sb3VyIjpbMCw2MCw2MF19LFswLDYwLDYwLDFdXSxbNyw2LCIwIiwwLHsiY29sb3VyIjpbMCw2MCw2MF19LFswLDYwLDYwLDFdXSxbNiw0LCIwIiwwLHsiY29sb3VyIjpbMCw2MCw2MF19LFswLDYwLDYwLDFdXSxbNCw1LCIxIiwwLHsiY29sb3VyIjpbMCw2MCw2MF19LFswLDYwLDYwLDFdXSxbMSwxMCwiMCIsMix7ImNvbG91ciI6WzAsNjAsNjBdfSxbMCw2MCw2MCwxXV1d
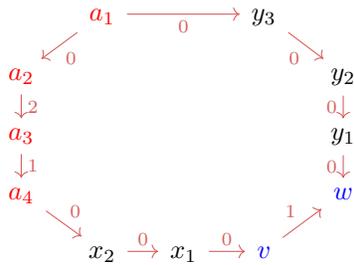
\[\begin{tikzcd}[column sep=small, row sep=small]
	& {\textcolor{red}{a_1}} && {y_3} \\
	{\textcolor{red}{a_2}} &&&& {y_2} \\
	{\textcolor{red}{a_3}} &&&& {y_1} \\
	{\textcolor{red}{a_4}} &&&& {\textcolor{blue}{w}} \\
	& {x_2} & {x_1} & {\textcolor{blue}{v}}
	\arrow["0"', color={rgb,255:red,214;green,92;blue,92}, from=1-2, to=1-4]
	\arrow["0", color={rgb,255:red,214;green,92;blue,92}, from=1-2, to=2-1]
	\arrow["0"', color={rgb,255:red,214;green,92;blue,92}, from=1-4, to=2-5]
	\arrow["2", color={rgb,255:red,214;green,92;blue,92}, from=2-1, to=3-1]
	\arrow["0"', color={rgb,255:red,214;green,92;blue,92}, from=2-5, to=3-5]
	\arrow["1", color={rgb,255:red,214;green,92;blue,92}, from=3-1, to=4-1]
	\arrow["0"', color={rgb,255:red,214;green,92;blue,92}, from=3-5, to=4-5]
	\arrow["0", color={rgb,255:red,214;green,92;blue,92}, from=4-1, to=5-2]
	\arrow["0", color={rgb,255:red,214;green,92;blue,92}, from=5-2, to=5-3]
	\arrow["0", color={rgb,255:red,214;green,92;blue,92}, from=5-3, to=5-4]
	\arrow["1", color={rgb,255:red,214;green,92;blue,92}, from=5-4, to=4-5]
\end{tikzcd}\]

Note that this is an $(l^\prime,h^\prime)$–central cycle for $l^\prime=11$ and $h^\prime=6$. Applying the composition $\mu_{x_2}\circ \mu_{x_1}\circ \mu_{v}$ to this $(l^\prime,h^\prime)$–central cycle we obtain the quiver

% https://q.uiver.app/#q=WzAsMTEsWzAsMywiXFx0ZXh0Y29sb3J7cmVkfXthXzR9Il0sWzEsMCwiXFx0ZXh0Y29sb3J7cmVkfXthXzF9Il0sWzAsMSwiXFx0ZXh0Y29sb3J7cmVkfXthXzJ9Il0sWzAsMiwiXFx0ZXh0Y29sb3J7cmVkfXthXzN9Il0sWzMsNCwiXFx0ZXh0Y29sb3J7Ymx1ZX17dn0iXSxbNCwzLCJcXHRleHRjb2xvcntibHVlfXt3fSJdLFsyLDQsInhfMSJdLFsxLDQsInhfMiJdLFs0LDEsInlfMiJdLFs0LDIsInlfMSJdLFszLDAsInlfMyJdLFsxLDIsIjAiLDAseyJjb2xvdXIiOlswLDYwLDYwXX0sWzAsNjAsNjAsMV1dLFsxMCw4LCIwIiwyLHsiY29sb3VyIjpbMCw2MCw2MF19LFswLDYwLDYwLDFdXSxbMiwzLCIyIiwwLHsiY29sb3VyIjpbMCw2MCw2MF19LFswLDYwLDYwLDFdXSxbOCw5LCIwIiwyLHsiY29sb3VyIjpbMCw2MCw2MF19LFswLDYwLDYwLDFdXSxbOSw1LCIwIiwyLHsiY29sb3VyIjpbMCw2MCw2MF19LFswLDYwLDYwLDFdXSxbMywwLCIxIiwwLHsiY29sb3VyIjpbMCw2MCw2MF19LFswLDYwLDYwLDFdXSxbMCw3LCIxIiwwLHsiY29sb3VyIjpbMCw2MCw2MF19LFswLDYwLDYwLDFdXSxbNyw2LCIwIiwwLHsiY29sb3VyIjpbMCw2MCw2MF19LFswLDYwLDYwLDFdXSxbNiw0LCIwIiwwLHsiY29sb3VyIjpbMCw2MCw2MF19LFswLDYwLDYwLDFdXSxbNCw1LCIwIiwwLHsiY29sb3VyIjpbMCw2MCw2MF19LFswLDYwLDYwLDFdXSxbMSwxMCwiMCIsMix7ImNvbG91ciI6WzAsNjAsNjBdfSxbMCw2MCw2MCwxXV1d
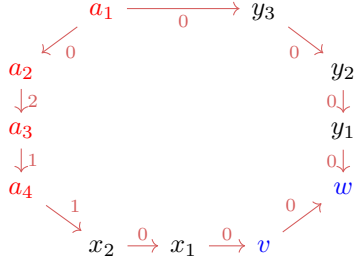
\[\begin{tikzcd}[row sep=small, column sep=small]
	& {\textcolor{red}{a_1}} && {y_3} \\
	{\textcolor{red}{a_2}} &&&& {y_2} \\
	{\textcolor{red}{a_3}} &&&& {y_1} \\
	{\textcolor{red}{a_4}} &&&& {\textcolor{blue}{w}} \\
	& {x_2} & {x_1} & {\textcolor{blue}{v}}
	\arrow["0"', color={rgb,255:red,214;green,92;blue,92}, from=1-2, to=1-4]
	\arrow["0", color={rgb,255:red,214;green,92;blue,92}, from=1-2, to=2-1]
	\arrow["0"', color={rgb,255:red,214;green,92;blue,92}, from=1-4, to=2-5]
	\arrow["2", color={rgb,255:red,214;green,92;blue,92}, from=2-1, to=3-1]
	\arrow["0"', color={rgb,255:red,214;green,92;blue,92}, from=2-5, to=3-5]
	\arrow["1", color={rgb,255:red,214;green,92;blue,92}, from=3-1, to=4-1]
	\arrow["0"', color={rgb,255:red,214;green,92;blue,92}, from=3-5, to=4-5]
	\arrow["1", color={rgb,255:red,214;green,92;blue,92}, from=4-1, to=5-2]
	\arrow["0", color={rgb,255:red,214;green,92;blue,92}, from=5-2, to=5-3]
	\arrow["0", color={rgb,255:red,214;green,92;blue,92}, from=5-3, to=5-4]
	\arrow["0", color={rgb,255:red,214;green,92;blue,92}, from=5-4, to=4-5]
\end{tikzcd}\]

Finally, applying the mutation $\mu_{a_4}$ to the last quiver we obtain the quiver $\widetilde{A}_{l^{\prime}-h^{\prime},h^{\prime}}=\widetilde{A}_{5,6}:$
% https://q.uiver.app/#q=WzAsMTEsWzAsMywiXFx0ZXh0Y29sb3J7cmVkfXthXzR9Il0sWzEsMCwiXFx0ZXh0Y29sb3J7cmVkfXthXzF9Il0sWzAsMSwiXFx0ZXh0Y29sb3J7cmVkfXthXzJ9Il0sWzAsMiwiXFx0ZXh0Y29sb3J7cmVkfXthXzN9Il0sWzMsNCwiXFx0ZXh0Y29sb3J7Ymx1ZX17dn0iXSxbNCwzLCJcXHRleHRjb2xvcntibHVlfXt3fSJdLFsyLDQsInhfMSJdLFsxLDQsInhfMiJdLFs0LDEsInlfMiJdLFs0LDIsInlfMSJdLFszLDAsInlfMyJdLFsxLDIsIjAiLDAseyJjb2xvdXIiOlswLDYwLDYwXX0sWzAsNjAsNjAsMV1dLFsxMCw4LCIwIiwyXSxbOCw5LCIwIiwyXSxbOSw1LCIwIiwyXSxbMCw3LCIwIiwwLHsiY29sb3VyIjpbMCw2MCw2MF19LFswLDYwLDYwLDFdXSxbNyw2LCIwIiwwLHsiY29sb3VyIjpbMCw2MCw2MF19LFswLDYwLDYwLDFdXSxbNiw0LCIwIiwwLHsiY29sb3VyIjpbMCw2MCw2MF19LFswLDYwLDYwLDFdXSxbNCw1LCIwIiwwLHsiY29sb3VyIjpbMCw2MCw2MF19LFswLDYwLDYwLDFdXSxbMSwxMCwiMCIsMl0sWzMsMiwiMCIsMl0sWzAsMywiMCIsMl1d
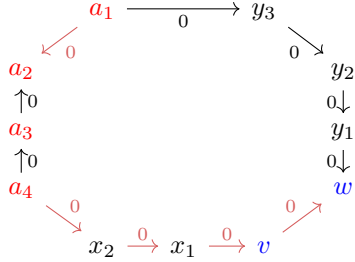
\[\begin{tikzcd}[row sep=small, column sep=small]
	& {\textcolor{red}{a_1}} && {y_3} \\
	{\textcolor{red}{a_2}} &&&& {y_2} \\
	{\textcolor{red}{a_3}} &&&& {y_1} \\
	{\textcolor{red}{a_4}} &&&& {\textcolor{blue}{w}} \\
	& {x_2} & {x_1} & {\textcolor{blue}{v}}
	\arrow["0"', from=1-2, to=1-4]
	\arrow["0", color={rgb,255:red,214;green,92;blue,92}, from=1-2, to=2-1]
	\arrow["0"', from=1-4, to=2-5]
	\arrow["0"', from=2-5, to=3-5]
	\arrow["0"', from=3-1, to=2-1]
	\arrow["0"', from=3-5, to=4-5]
	\arrow["0"', from=4-1, to=3-1]
	\arrow["0", color={rgb,255:red,214;green,92;blue,92}, from=4-1, to=5-2]
	\arrow["0", color={rgb,255:red,214;green,92;blue,92}, from=5-2, to=5-3]
	\arrow["0", color={rgb,255:red,214;green,92;blue,92}, from=5-3, to=5-4]
	\arrow["0", color={rgb,255:red,214;green,92;blue,92}, from=5-4, to=4-5]
\end{tikzcd}\]
\end{ejem}
\medskip

\section{Further consequences}

As a first consequence of Theorem \autoref{ref:teo principal}, we will show that the class obtained in this work for $m=1$ coincides with the mutation class of $\tilde{\mathbb{A}}_n$–quivers developed in \cite{bastian2012mutation}. We now rewrite the definitions of the parameters $r_1^{\prime}, r_2^{\prime}, r_1^{\prime\prime}, r_2^{\prime\prime}$ from \cite{bastian2012mutation}. 

\begin{enumerate}
    \item $r_1^{\prime}$ is the number of central arrows of color 0 counterclockwise oriented that do not belong to any oriented 3–cycle.
    \item $r_2^{\prime}$ is the number of 3-cycles $\alpha\beta\gamma$ in which $\alpha$ is a central arrow of color 0 counterclockwise oriented.
    \item For each subquiver $Q_w$ with $w\in W_p$, let $r_{1,w}^{\prime\prime}$ be the number of arrows in $Q_w$ that do not belong to any oriented 3–cycle and let $r_{2,w}^{\prime\prime}$ be the number of oriented 3–cycles in $Q_w$. Then we have the parameters
    \[r_1^{\prime\prime}=\sum_{w\in W_p} r_{1,w}^{\prime\prime},\text{ and } r_2^{\prime\prime}=\sum_{w\in W_p} r_{2,w}^{\prime\prime}.\]
\end{enumerate}
Similarly, the parameters $s_1^{\prime}, s_1^{\prime \prime}, s_2^{\prime}, s_2^{\prime \prime}$ are defined by replacing counterclockwise by clockwise and $W_p$ by $W_q.$

\begin{cor}
Let $Q\in \mathcal{Q}_{p,q}^1$ with parameters $r_1^{\prime}, r_1^{\prime \prime}, r_2^{\prime}$, $r_2^{\prime \prime}$ and $s_1^{\prime}, s_1^{\prime \prime}, s_2^{\prime}$, $s_2^{\prime \prime}$. Then
    \[p=(r_1^{\prime}+r_1^{\prime\prime})+2(r_2^{\prime}+r_2^{\prime\prime}),\quad q=(s_1^{\prime}+s_1^{\prime\prime})+2(s_2^{\prime}+s_2^{\prime\prime}).\]
\end{cor}

\begin{proof}

For each peripheral vertex $w\in W_p$, $n_w$ satisfies \[n_w=r_{1,w}^{\prime\prime}+2r_{2,w}^{\prime\prime}+1,\] therefore, \[x_p=\sum_{w\in W_p} n_w = \sum_{w\in W_p} \left(r_{1,w}^{\prime\prime}+2r_{2,w}^{\prime\prime}+1\right) = r_1^{\prime\prime}+2r_2^{\prime\prime}+ r_2^{\prime}.\] Note that, if $\Delta_l$ denotes the $(l,h)$–central cycle of $Q$, then $l=r_1^{\prime}+r_2^{\prime}+s_1^{\prime}+s_2^{\prime}$ and $h=s_1^{\prime}+s_2^{\prime}$, hence, \[p=l-h+x_p=(r_1^{\prime}+r_2^{\prime})+(r_1^{\prime\prime}+2r_2^{\prime\prime}+ r_2^{\prime})=(r_1^{\prime}+r_1^{\prime\prime})+2(r_2^{\prime}+r_2^{\prime\prime}).\]
Similarly, it is easy to see that 
\[q=(s_1^{\prime}+s_1^{\prime\prime})+2(s_2^{\prime}+s_2^{\prime\prime}).\]
\end{proof}

On the other hand, as in the $\mathbb{A}$-case, by \cite{buan}, any quiver of an $m$–cluster tilted algebra of type $\tilde{\mathbb{A}}$ can be obtained by considering the $0$–colored part of a quiver in the mutation class of type $\tilde{\mathbb{A}}$. Note that in \cite{gubitosi2024coloured} it is shown that for any quiver in the $m$–colored mutation class of type $\mathbb{A}_n$, every $k$–cycle in the $0$–colored part is oriented and $k=m+2.$

\medskip

In \cite{gubitosi2018m}, the $m$–cluster tilted algebras of type $\tilde{\mathbb{A}}$, with $m\geq 2$, are characterized through their bound quiver $(Q,I)$, so that if $A=KQ/I$ is a connected component of an $m$–cluster tilted algebra of type $\tilde{\mathbb{A}}$, then $Q$ satisfies certain conditions which we can deduce independently from Theorem \autoref{ref:teo principal}, as can be seen in Corollary \autoref{ref:corolario 2}. 

\medskip

Recall that $\chi(Q)=|Q_1|-|Q_0|+1$, and denote by $Q-\alpha$ the subquiver of $Q$ obtained by removing the arrow $\alpha$ from $Q$, that is, $(Q-\alpha)_0=Q_0$ and $(Q-\alpha)_1=Q_1\setminus\left\{\alpha\right\}.$ Note that $\chi(Q-\alpha)=\chi(Q)-1.$

\medskip

\begin{cor}
\label{ref:corolario 2}
Let $Q\in \mathcal{Q}^m_{p,q}$, and suppose that $Q^{\prime}$ is a connected component of the $0$–colored part of $Q$. Then $Q^{\prime}$ satisfies:
    \begin{enumerate}
        \item[(a)] for every vertex $x\in Q_0^{\prime}$, the sets $s^{-1}(x)$ and $t^{-1}(x)$ have cardinality at most two;
        \item[(b)] it has no loops;
        \item[(c)] it may contain a cycle $\Delta$ that is not a oriented $(m+2)$–cycle;
        \item[(d)] if it contains $\Delta$, then it contains $\chi(Q^{\prime})-1$ oriented cycles of length $m+2$;
        \item[(e)] if it does not contain $\Delta$, then the only cycles it contains are oriented cycles of length $m+2$.
\end{enumerate}
\end{cor}

\begin{proof}
Let $Q\in \mathcal{Q}^m_{p,q}$ and let $Q^{\prime}$ be a connected component of the $0$–colored part of $Q$. Clearly $Q^{\prime}$ has no loops by Definition \autoref{ref: definicion clase A tilde}, which verifies condition $(b).$ On the other hand, in \cite{gubitosi2018m} it is shown that the $0$–colored part of every quiver belonging to the mutation class of quivers of type $\mathbb{A}$ satisfies condition $(a)$. Therefore, considering that the neighbors of every vertex in $Q$ are distributed according to the same conditions as the vertices in the mutation class of quivers of type $\mathbb{A}$, and that condition $(a)$ is local, in the sense that it is analyzed on the neighbors of the vertices of $Q$, then $(a)$ holds thanks to the proof given in \cite{gubitosi2018m} for the case of quivers of type $\mathbb{A}.$

We now prove conditions $(c), (d)$ and $(e)$. If the $(l,h)$–central cycle $\Delta_l$ of $Q$ does not remain in $Q^{\prime}$, then $Q^{\prime}$ is the disjoint union of quivers $Q_w$ of type $\mathbb{A}_{n_w}$, with $n_w\geq 1$, and therefore \cite{gubitosi2018m} shows that the only possible oriented $k$–cycles in $Q^{\prime}$ satisfy $k=m+2$. If the $(l,h)$–central cycle $\Delta_l$ of $Q$ remains in $Q^{\prime}$, then $\Delta_l$ is a non-oriented cycle in $Q^{\prime}$. It remains to see, in this case, that $Q^{\prime}$ contains $\chi(Q^{\prime})-1$ oriented cycles of length $m+2$. Let $R$ denote the number of $(m+2)$–oriented cycles in $Q^{\prime}$. First, suppose that there exists a central arrow $\alpha$ in $Q^{\prime}$ such that $D_\alpha^{\prime}$ is a clique of size $2$; in this case, the subquiver $Q^{\prime}-\alpha$ of $Q^{\prime}$ is a quiver of type $\mathbb{A}$. Therefore, if $R^{\prime}$ denotes the number of $(m+2)$–oriented cycles in $Q^{\prime}-\alpha$, we have that $R^{\prime}=\chi(Q^\prime-\alpha).$ In addition,  the number of $(m+2)$–oriented cycles in $Q^{\prime}$ remains unchanged in $Q^{\prime}-\alpha$, that is, $R=R^{\prime}.$ Therefore,
\[R=R^{\prime}=\chi(Q^{\prime}-\alpha)=\chi(Q^{\prime})-1.\]

On the other hand, suppose that for every central arrow $\alpha$, $D_\alpha^{\prime}$ in $Q^{\prime}$ has size greater than $2$. In this case, take $\alpha$ a central arrow of $D_\alpha^{\prime}$ and $\beta\neq \alpha$ any other arrow of $D_\alpha^{\prime}$. Let $R^{\prime}$ and $R^{\prime\prime}$ denote the number of oriented cycles of length $m+2$ in $Q^{\prime}-\alpha$ and $(Q-\alpha)-\beta$, respectively. Since removing $\alpha$ from $Q^{\prime}$ removes an oriented cycle of length $m+2$, we get $R^{\prime}=R-1,$ while removing $\beta$ from $Q-\alpha$ does not change the number of oriented cycles of length $m+2$ in $(Q-\alpha)-\beta$, that is, $R^{\prime}=R^{\prime\prime}.$ Note that $(Q-\alpha)-\beta$ is  a quiver of type $\mathbb{A}$ and therefore $R^{\prime\prime}=\chi((Q^\prime-\alpha)-\beta).$ We conclude that
\[R=R^{\prime}+1=R^{\prime\prime}+1=\chi((Q^\prime-\alpha)-\beta)+1=\chi(Q^\prime-\alpha)=\chi(Q^\prime)-1.\]
\end{proof}

\begin{ejem}
Consider $m=3$ and the following $3$–colored quiver $Q$
% https://q.uiver.app/#q=WzAsMTEsWzIsNCwiXFx0ZXh0Y29sb3J7cmVkfXthXzR9Il0sWzIsMiwiXFx0ZXh0Y29sb3J7cmVkfXthXzF9Il0sWzAsMiwiXFx0ZXh0Y29sb3J7cmVkfXthXzJ9Il0sWzAsNCwiXFx0ZXh0Y29sb3J7cmVkfXthXzN9Il0sWzQsMiwiXFx0ZXh0Y29sb3J7Ymx1ZX17d18yfSJdLFs0LDQsIlxcdGV4dGNvbG9ye2JsdWV9e3dfMX0iXSxbMSwxLCJcXHRleHRjb2xvcntibHVlfXt3XzN9Il0sWzUsMiwieF8xIl0sWzYsMiwieF8yIl0sWzUsNCwieSJdLFsxLDAsInoiXSxbMCwxLCIwIiwwLHsiY29sb3VyIjpbMCw2MCw2MF19LFswLDYwLDYwLDFdXSxbMSw0LCIwIl0sWzQsMCwiMiIsMSx7ImxhYmVsX3Bvc2l0aW9uIjoyMH1dLFs1LDAsIjEiLDJdLFsxLDUsIjEiLDEseyJsYWJlbF9wb3NpdGlvbiI6MjB9XSxbMiw2LCIxIl0sWzIsMywiMCIsMCx7ImNvbG91ciI6WzAsNjAsNjBdfSxbMCw2MCw2MCwxXV0sWzIsMSwiMCJdLFszLDAsIjAiLDIseyJjb2xvdXIiOlswLDYwLDYwXX0sWzAsNjAsNjAsMV1dLFs0LDUsIjAiXSxbMSw2LCIwIiwyXSxbNCw3LCIwIl0sWzcsOCwiMSJdLFs1LDksIjAiXSxbNiwxMCwiMCJdXQ==
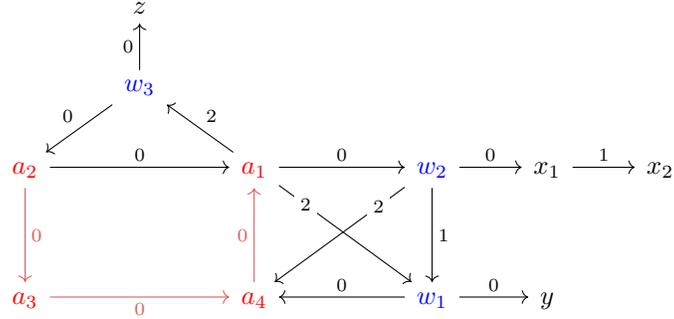
\[\begin{tikzcd}
	& z \\
	& {\textcolor{blue}{w_3}} \\
	{\textcolor{red}{a_2}} && {\textcolor{red}{a_1}} && {\textcolor{blue}{w_2}} & {x_1} & {x_2} \\
	\\
	{\textcolor{red}{a_3}} && {\textcolor{red}{a_4}} && {\textcolor{blue}{w_1}} & y
	\arrow["0", from=2-2, to=1-2]
	\arrow["0"', from=2-2, to=3-1]
	\arrow["0", from=3-1, to=3-3]
	\arrow["0", color={rgb,255:red,214;green,92;blue,92}, from=3-1, to=5-1]
	\arrow["2"', from=3-3, to=2-2]
	\arrow["0", from=3-3, to=3-5]
	\arrow["2"{description, pos=0.2}, from=3-3, to=5-5]
	\arrow["0", from=3-5, to=3-6]
	\arrow["2"{description, pos=0.2}, from=3-5, to=5-3]
	\arrow["1", from=3-5, to=5-5]
	\arrow["1", from=3-6, to=3-7]
	\arrow["0"', color={rgb,255:red,214;green,92;blue,92}, from=5-1, to=5-3]
	\arrow["0", color={rgb,255:red,214;green,92;blue,92}, from=5-3, to=3-3]
	\arrow["0"', from=5-5, to=5-3]
	\arrow["0", from=5-5, to=5-6]
\end{tikzcd}\]
Then its $0$–colored part is the quiver $Q^{\prime}$:
% https://q.uiver.app/#q=WzAsMTEsWzIsNCwiXFx0ZXh0Y29sb3J7cmVkfXthXzR9Il0sWzIsMiwiXFx0ZXh0Y29sb3J7cmVkfXthXzF9Il0sWzAsMiwiXFx0ZXh0Y29sb3J7cmVkfXthXzJ9Il0sWzAsNCwiXFx0ZXh0Y29sb3J7cmVkfXthXzN9Il0sWzQsMiwiXFx0ZXh0Y29sb3J7Ymx1ZX17d18yfSJdLFs0LDQsIlxcdGV4dGNvbG9ye2JsdWV9e3dfMX0iXSxbMSwxLCJcXHRleHRjb2xvcntibHVlfXt3XzN9Il0sWzUsMiwieF8xIl0sWzYsMiwieF8yIl0sWzUsNCwieSJdLFsxLDAsInoiXSxbMCwxLCIwIiwwLHsiY29sb3VyIjpbMCw2MCw2MF19LFswLDYwLDYwLDFdXSxbMSw0LCIwIl0sWzIsMywiMCIsMCx7ImNvbG91ciI6WzAsNjAsNjBdfSxbMCw2MCw2MCwxXV0sWzIsMSwiMCJdLFszLDAsIjAiLDIseyJjb2xvdXIiOlswLDYwLDYwXX0sWzAsNjAsNjAsMV1dLFs0LDUsIjAiXSxbMSw2LCIwIiwyXSxbNCw3LCIwIl0sWzUsOSwiMCJdLFs2LDEwLCIwIl1d
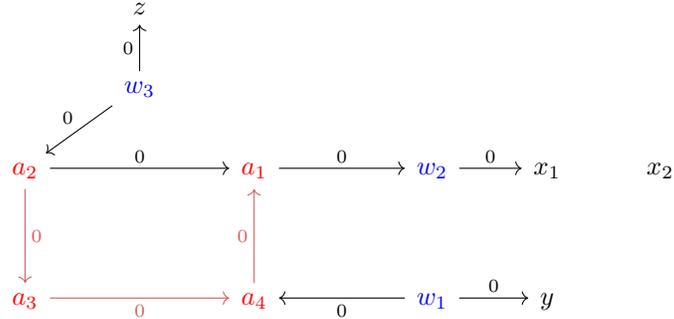
\[\begin{tikzcd}
	& z \\
	& {\textcolor{blue}{w_3}} \\
	{\textcolor{red}{a_2}} && {\textcolor{red}{a_1}} && {\textcolor{blue}{w_2}} & {x_1} & {x_2} \\
	\\
	{\textcolor{red}{a_3}} && {\textcolor{red}{a_4}} && {\textcolor{blue}{w_1}} & y
	\arrow["0", from=2-2, to=1-2]
	\arrow["0", from=3-1, to=3-3]
	\arrow["0", color={rgb,255:red,214;green,92;blue,92}, from=3-1, to=5-1]
	\arrow["0"', from=2-2, to=3-1]
	\arrow["0", from=3-3, to=3-5]
	\arrow["0", from=3-5, to=3-6]
	\arrow["0", from=5-5, to=5-3]
	\arrow["0"', color={rgb,255:red,214;green,92;blue,92}, from=5-1, to=5-3]
	\arrow["0", color={rgb,255:red,214;green,92;blue,92}, from=5-3, to=3-3]
	\arrow["0", from=5-5, to=5-6]
\end{tikzcd}\]
$Q^{\prime}$ clearly satisfies the conditions $(a)-(e)$ of Corollary \ref{ref:corolario 2}.
\end{ejem}

\section*{Acknowledges}

The second author acknowledges the support of the doctoral program of PEDECIBA at Universidad de la República, Montevideo, Uruguay.

\bibliographystyle{amsplain}
\bibliography{references}

\end{document}

%% file: preambulo.tex
\usepackage{mathtools}

\usepackage[english]{babel}
\usepackage[T1]{fontenc}
\usepackage{amsmath}
\usepackage{amsthm}
\usepackage{amsfonts}
\usepackage{hyperref}
 \usepackage{xfrac}  
 \usepackage{tikz}
\usepackage{tikz-cd}
\usepackage{amsfonts}
\usepackage{amssymb}
\usepackage{amsrefs}
\usepackage{rotating}
\usetikzlibrary{babel}
\usepackage{listings}
\usepackage{amssymb}
\usepackage{extarrows}
\usepackage{graphicx,float}
\usepackage{multirow}
\usepackage{tikz-cd}
\usepackage{tasks}
\usepackage{color}
\usepackage{cancel}
\usepackage[all]{xy}
\usetikzlibrary{fit,backgrounds,calc}
\usetikzlibrary{patterns}
\usepackage{pgfplots}
\usepgfplotslibrary{patchplots}
\usetikzlibrary{patterns.meta, positioning, arrows, backgrounds}
\pgfplotsset{compat=1.15}
\usepackage{adjustbox}

\newcommand{\C}{\mathcal{C}}

\newcommand{\adel}[1]{\left\{#1\right\}}

\newcommand{\func}[3]{#1\colon  #2  \to  #3}

\newcommand\restr[2]{{% we make the whole thing an ordinary symbol
  \left.\kern-\nulldelimiterspace % automatically resize the bar with \right
  #1 % the function
  \vphantom{\big|} % pretend it's a little taller at normal size
  \right|_{#2} % this is the delimiter
  }}
  
\theoremstyle{plain}% default
\newtheorem{teo}{Theorem}[section]
\newtheorem{lem}[teo]{Lemma}
\newtheorem{prop}[teo]{Proposition}
\newtheorem{cor}[teo]{Corollary}

\theoremstyle{definition}
\newtheorem{defi}[teo]{Definition}
\newtheorem{obs}[teo]{Remark}

\newtheorem{ejem}[teo]{Example}